\documentclass[aos,preprint]{imsart}

\RequirePackage[OT1]{fontenc}
\RequirePackage{amsthm,amsmath,natbib}
\RequirePackage[colorlinks,citecolor=blue,urlcolor=blue]{hyperref}
\RequirePackage{hypernat}

\startlocaldefs
\numberwithin{equation}{section}
\theoremstyle{plain}

\usepackage{graphicx,amsmath,amssymb,epsfig}
\usepackage{colortbl}
\long\def\comment#1{}
\long\def\comment#1{}
\newtheorem{theorem}{Theorem}

\newtheorem{lemma}{Lemma}

\theoremstyle{definition}

\newtheorem{remark}{Comment}[section]
\newtheorem{example}{Example}

\newcommand{\ceil}[1]{\left\lceil #1 \right\rceil}

\newcommand{\be}{\begin{eqnarray}}
\newcommand{\ee}{\end{eqnarray}}

\newcommand{\underf}{\underline{f}}
\newcommand{\uglyf}{\underf}
\newcommand{\uglyfh}{\underf}

\newcommand{\bm}{\theta_m}
\newcommand{\emperror}{\epsilon}

\newcommand{\ba}{\begin{array}}
\newcommand{\ea}{\end{array}}
\newcommand{\bs}{\begin{align}\begin{split}\nonumber}
\newcommand{\bsnumber}{\begin{align}\begin{split}}
\newcommand{\es}{\end{split}\end{align}}

\renewcommand{\(}{\left(}
\renewcommand{\)}{\right)}
\renewcommand{\[}{\left[}
\renewcommand{\]}{\right]}
\renewcommand{\hat}{\widehat}

\newcommand{\Gn}{\mathbb{G}_n}

\newcommand{\Pn}{\mathbb{P}_n}

\newcommand{\En}{{\mathbb{E}_n}}
\newcommand{\Ep}{\mathrm{E}}
\newcommand{\W}{W_{\mathcal{U}}}
\newcommand{\sign}{ {\rm sign}}

\def\RR{{\rm I\kern-0.18em R}}

\def\supp{{\rm support}}
\def\mmu{{\mu}}
\def\nn{{m_0}}

\def\hs{{\widehat s}}

\newcommand{\mA}{\mathcal{A}}
\newcommand{\mB}{\mathcal{B}}
\newcommand{\mC}{\mathcal{C}}
\newcommand{\bG}{\mathbb{G}}
\newcommand{\semin}[1]{\varphi_{{\rm min}}(#1)}
\newcommand{\semax}[1]{\varphi_{{\rm max}}(#1)}

\endlocaldefs

\begin{document}

\begin{frontmatter}
\title{$\ell_1$-Penalized Quantile Regression in
High-Dimensional Sparse Models\protect}

\begin{aug}
\author{\fnms{Alexandre } \snm{Belloni}
\ead[label=e1]{abn5@duke.edu}}
\and \author{\fnms{Victor} \snm{Chernozhukov}\thanksref{t1,t2}
\ead[label=e2]{vchern@mit.edu}}


\thankstext{t1}{First version: December, 2007, This version: \today.}
\thankstext{t2}{The authors gratefully acknowledge research support from the National Science Foundation.}

\affiliation{Duke University and Massachusetts Institute of Technology}


\end{aug}

\begin{abstract}
%

We consider median regression and, more generally, a possibly infinite collection of quantile regressions in high-dimensional sparse models. In these models the number of regressors $p$ is very large, possibly larger than the sample size $n$, but only at most $s$ regressors have a non-zero impact on each conditional quantile of the response variable, where $s$ grows more slowly than $n$.  Since ordinary quantile regression is not consistent in this case, we consider $\ell_1$-penalized quantile regression ($\ell_1$-QR), which penalizes the $\ell_1$-norm of regression coefficients, as well as the post-penalized QR estimator (post-$\ell_1$-QR), which applies ordinary QR to the model selected by $\ell_1$-QR. First, we show that under general conditions $\ell_1$-QR is consistent at the near-oracle rate $\sqrt{s/n} \sqrt{\log (p \vee n)}$, uniformly in the compact set $\mathcal{U} \subset (0,1)$ of quantile indices.  In deriving this result, we propose a partly pivotal, data-driven choice of the penalty level and show that it satisfies the requirements for achieving this rate. Second, we show that under similar conditions post-$\ell_1$-QR is consistent at the near-oracle rate $\sqrt{s/n} \sqrt{\log (p\vee n)}$,  uniformly over $\mathcal{U}$, even if the $\ell_1$-QR-selected models miss some
components of the true models, and the rate could be even closer to the oracle rate otherwise.  Third, we characterize conditions under which $\ell_1$-QR contains the true model as a submodel, and derive bounds on the dimension of the selected model, uniformly over $\mathcal{U}$; we also provide conditions under which hard-thresholding selects the minimal true model, uniformly over $\mathcal{U}$.

\end{abstract}



\end{frontmatter}

\section{Introduction}\label{Sec:Intro}

Quantile regression is an important statistical method for analyzing the impact of regressors on the conditional distribution
of a response variable (cf. \cite{Laplace},  \cite{KB78}). It captures the heterogeneous
impact of regressors on different parts of the
distribution \cite{Buchinsky1994}, exhibits robustness to outliers
\cite{K2005},  has excellent computational properties
\cite{PornoyKoenker1997}, and has wide applicability
\cite{K2005}.  The asymptotic theory for quantile regression has been developed under both a fixed number of regressors and an
increasing number of regressors.  The asymptotic theory under
a fixed number of regressors is given in \cite{KB78},
\cite{Portnoy:mess}, \cite{GJ1992}, \cite{Knight1998},
\cite{C2005} and others. The asymptotic theory under an increasing
number of regressors is given in \cite{HS00} and
\cite{BC-MCMC,BC-qrID}, covering the case where the number of
regressors $p$ is negligible relative to the sample size $n$
(i.e., $p=o(n)$).

In this paper, we consider quantile regression in high-dimensional sparse models (HDSMs).  In  such models, the overall number of regressors $p$ is very large, possibly much larger than the sample size $n$. However, the number of significant regressors for each conditional quantile of interest \comment{-- those having a non-zero impact on the response variable --} is  at most $s$, which is smaller than the sample size, that is, $s = o(n)$. HDSMs (\cite{CandesTao2007,MY2007,BickelRitovTsybakov2009}) have emerged to deal with many new applications arising in biometrics, signal processing, machine learning, econometrics, and other areas of data analysis where high-dimensional data sets have become widely available.

A number of papers have begun to investigate estimation of HDSMs, focusing primarily on penalized mean regression, with  the $\ell_1$-norm acting as a penalty function \cite{BickelRitovTsybakov2009,CandesTao2007,Koltchinskii2009,MY2007,vdGeer,ZhangHuang2006}.   \cite{BickelRitovTsybakov2009,CandesTao2007,Koltchinskii2009,MY2007,ZhangHuang2006} demonstrated the fundamental result that $\ell_1$-penalized least squares estimators achieve the rate $\sqrt{s/n} \sqrt{\log p}$, which is very close to the oracle rate $\sqrt{s/n}$ achievable when the true model is known.   \cite{vdGeer} demonstrated a similar fundamental result on the excess forecasting error loss under both quadratic and non-quadratic loss functions. Thus the estimator can be consistent and can have excellent forecasting performance even under very rapid, nearly exponential, growth of the total number of regressors $p$.  See \cite{FanLv2006,BickelRitovTsybakov2009,BuneaTsybakovWegkamp2006,BuneaTsybakovWegkamp2007,BuneaTsybakovWegkamp2007b,LouniciPontilTsybakovvandeGeer2009,RosenbaumTsybakov2008}
  for many other interesting developments and a detailed review of the existing literature.

Our paper's contribution is to develop a set of results on model selection and rates of convergence for quantile regression within the HDSM framework.  Since ordinary quantile regression is inconsistent in HDSMs, we consider quantile regression penalized by the $\ell_1$-norm of parameter coefficients, denoted $\ell_1$-QR. First, we show that under general conditions $\ell_1$-QR estimates of regression coefficients and regression functions are consistent at the near-oracle rate $\sqrt{s/n} \sqrt{\log (p\vee n)}$, uniformly in a compact interval $\mathcal{U} \subset (0,1)$ of quantile indices.\footnote{Under $s \to \infty$, the oracle rate, uniformly over a proper compact interval $\mathcal{U}$, is $\sqrt{(s/n) \log n}$, cf. \cite{BC-qrID}; the oracle rate for a single quantile index is $\sqrt{s/n}$, cf. \cite{HS00}.}  (This result is different from and hence complementary to \cite{vdGeer}'s fundamental results on the rates for excess forecasting error loss.)  Second, in order to make $\ell_1$-QR practical, we propose a partly pivotal, data-driven choice of the penalty level, and show that this choice leads to the same sharp convergence rate.  Third, we show that $\ell_1$-QR correctly selects the true model as a valid submodel when the non-zero coefficients of the true model are well separated from zero. Fourth, we also propose and analyze the post-penalized estimator (post-$\ell_1$-QR), which applies ordinary, unpenalized quantile regression to the model selected by the penalized estimator, and thus aims at reducing the regularization bias of the penalized estimator.  We show that under similar conditions post-$\ell_1$-QR can perform as well as $\ell_1$-QR in terms of the rate of convergence, uniformly over $\mathcal{U}$, even if the $\ell_1$-QR-based model selection misses some components of the true models.  This occurs because $\ell_1$-QR-based model selection only misses those components that have relatively small coefficients.  Moreover, post-$\ell_1$-QR can perform better than $\ell_1$-QR if the $\ell_1$-QR-based model selection correctly includes all components of the true model as a subset.  (Obviously, post-$\ell_1$-QR can perform as well as the oracle if the $\ell_1$-QR perfectly selects the true model, which is, however, unrealistic for many  designs of interest.)  Fifth, we illustrate the use of $\ell_1$-QR and post-$\ell_1$-QR with a Monte Carlo experiment and an international economic growth example.  To the best of our knowledge, all of the above results are new and contribute to the literature on HDSMs.  Our results on post-penalized estimators and some proof techniques could also be of interest in other problems.  We provide further technical comparisons to the literature in Section \ref{Sec:Setting}.

\subsection{Notation} In what follows, we implicitly index all parameter values
by the sample size $n$, but we omit the index whenever this does
not cause confusion. We  use the empirical process
notation as defined in \cite{vdV-W}. In particular, given a random sample $Z_1,...,Z_n$, let
$\Gn (f) = \Gn (f(Z_i)) :=
n^{-1/2} \sum_{i=1}^n (f(Z_i) - \Ep[f(Z_i)])$ and $\mathbb{E}_n f = \mathbb{E}_n f(Z_i) := \sum_{i=1}^n f(Z_i)/n$.
We use the notation $a \lesssim b$ to denote $a = O(b)$, that is, $a \leq c b$ for some constant $c>0$ that
does not depend on $n$; and $a\lesssim_P b$ to denote $a=O_P(b)$. We also use the notation $a \vee b =
\max\{ a, b\}$ and $a \wedge b = \min\{ a , b \}$. We denote
the $\ell_2$-norm by $\|\cdot\|$, $\ell_1$-norm by $\|\cdot\|_1$,
$\ell_{\infty}$-norm by $\|\cdot\|_\infty$, and the
$\ell_0$-``norm" by $\|\cdot\|_0$ (i.e., the number of non-zero
components). We denote by $\|\beta\|_{1,n} =
\sum_{j=1}^p \hat \sigma_j |\beta_j|$ the $\ell_1$-norm weighted
by $\hat\sigma_j$'s. Finally, given a vector $\delta \in \RR^p$,
and a set of indices $T \subset \{1,\ldots,p\}$, we denote by
$\delta_T$ the vector in which $\delta_{Tj} = \delta_j$ if $j\in
T$, $\delta_{Tj}=0$ if $j \notin T$.
\section{The Estimator, the Penalty Level, and Overview of Rate Results}\label{Sec:Setting}
In this section we formulate the setting and the estimator, and state  primitive regularity conditions. We also provide an overview of the main results.

\subsection{Basic Setting} The setting of interest
corresponds to a parametric quantile regression model,
where the dimension $p$ of the underlying model increases with
the sample size $n$. Namely, we consider a response variable $y$ and $p$-dimensional covariates $x$ such that the $u$-th conditional quantile function of $y$ given $x$ is given by
 \begin{equation}\label{lin model}
F^{-1}_{y_i|x_i}(u|x_i)= x'\beta(u),  \ \
\beta(u) \in \RR^p, \ \ \mbox{for all} \ u \in \mathcal{U},
 \end{equation} where $\mathcal{U}\subset(0,1)$ is a compact set of quantile indices. Recall that the $u$-th conditional quantile $F^{-1}_{y_i|x_i}(u|x)$ is the  inverse of the conditional distribution function $F_{y_i|x_i}(y|x_i)$
 of $y_i$ given $x_i$.  We consider the case where the dimension $p$ of the model is
large, possibly much larger than the available sample size $n$,
but the true model $\beta(u)$ has a sparse support $$T_u = \supp(\beta(u)) = \{ j
\in \{1,\ldots,p\} \ : \ |\beta_j(u)|>0 \}
$$ having only $s_u \leq s \leq n/\log(n \vee p)$ non-zero components  for all $u\in\mathcal{U}$.

The population coefficient $\beta(u)$ is known to minimize the criterion function
 \be\label{define beta}
Q_u(\beta) = \Ep[\rho_{u} (y-x'\beta)],
 \ee where $\rho_{u} (t) = (u - 1\{t\leq 0\})t$ is the asymmetric absolute deviation  function \cite{KB78}.  Given a random sample  $(y_1,x_1),\ldots,(y_n,x_n)$, the quantile regression estimator of $\beta(u)$ is defined as a minimizer of the empirical analog of (\ref{define beta}):
\begin{equation}\label{QR}
\hat Q_u(\beta) = \En\[ \rho_u (y_i - x_i'\beta) \].
\end{equation}

In high-dimensional settings, particularly when $p \geq n$, ordinary quantile regression is generally inconsistent, which motivates the use of penalization in order to remove all, or at least nearly all, regressors whose population coefficients are zero, thereby possibly restoring consistency.  A penalization that has proven quite useful in least squares settings is the $\ell_1$-penalty leading to the Lasso estimator \cite{T1996}.

\subsection{Penalized and Post-Penalized Estimators}
The $\ell_1$-penalized quantile regression estimator $\hat \beta (u)$ is a solution to the following optimization problem: \begin{equation}\label{Def:L1QR}
\min_{\beta \in \mathbb{R}^p} \ \hat Q_u(\beta) + \frac{ \lambda\sqrt{u(1-u)}}{n}\sum_{j=1}^p
\hat \sigma_j|\beta_j|
\end{equation} where $\hat\sigma_j^2 = \En[x_{ij}^2]$. The criterion function in (\ref{Def:L1QR}) is the sum of the criterion function (\ref{QR}) and a penalty function given by a scaled $\ell_1$-norm of the parameter vector. The overall penalty level $\lambda\sqrt{u(1-u)}$ depends on each quantile index $u$, while $\lambda$ will depend on the set $\mathcal{U}$ of quantile indices of interest. The $\ell_1$-penalized quantile regression  has been considered in \cite{KnightFu2000} under small (fixed) $p$ asymptotics.  It is important to note that the penalized quantile regression problem (\ref{Def:L1QR}) is equivalent to a linear programming problem (see Appendix \ref{App:LP}) with a dual version that is useful for analyzing the sparsity of the solution. When the solution is not unique, we define $\hat \beta(u)$ as any optimal basic feasible solution (see, e.g., \cite{BertsimasTsitsiklis}). Therefore, the problem  (\ref{Def:L1QR}) can be solved in polynomial time, avoiding the computational curse of dimensionality.  Our goal is to derive the rate of convergence and model selection properties of this estimator.

The post-penalized estimator (post-$\ell_1$-QR) applies ordinary quantile regression to  the model $\widehat T_u$ selected by the $\ell_1$-penalized quantile regression. Specifically, set
$$\widehat T_u = \supp( \hat \beta(u) ) = \{ j \in \{1,\ldots,p\} \ : \ |\hat\beta_j(u)| > 0\},$$
and define the post-penalized estimator $\widetilde \beta(u)$ as \begin{equation}\label{Def:TwoStep}
\widetilde \beta(u) \in \arg\min_{\beta \in \mathbb{R}^p:  \beta_{\widehat T_u^c} = 0} \widehat Q_u(\beta),
\end{equation}
which removes from further estimation the regressors that were not selected. If the model selection works perfectly -- that is, $\widehat T_u = T_u$ -- then this estimator is simply the oracle
estimator, whose properties are well known. However, perfect model
selection might be unlikely for many designs of interest.
Rather, we are interested in the more realistic scenario
where the first-step estimator $\hat \beta(u)$ fails to select
some components of $ \beta(u)$. Our goal is to derive the rate of convergence for the post-penalized estimator and show it
can perform well under this scenario.

\subsection{The choice of the penalty level $\lambda$}
In order to describe our choice of the penalty level $\lambda$, we
introduce the random variable \begin{equation}\label{almost sure}
\Lambda = n \sup_{u\in \mathcal{U}}\max_{1\leq j\leq p}\left|
\En\[ \frac{x_{ij}( u - 1\{u_i \leq u
\})}{\hat\sigma_j\sqrt{u(1-u)}}\] \right| ,
\end{equation}
where $u_1,\ldots, u_n$ are i.i.d. uniform $(0,1)$ random
variables, independently distributed from the regressors,
$x_1,\ldots, x_n$.  The random variable $\Lambda$ has a known, that
is, pivotal, distribution conditional on  $X= [x_1,\ldots, x_n]'$.
We then set
\begin{equation}\label{Def:LambdaPivotal00}
\lambda = c \cdot  \Lambda(1-\alpha|X),
 \end{equation}
where $\Lambda(1-\alpha|X) := (1-\alpha)$-quantile of $\Lambda$
conditional  on $X$, and the constant $c>1$ depends on the design.\footnote{$c$ depends only on the constant $c_0$ appearing in condition D.4; when $c_0 \geq 9$, it suffices to set $c=2$.}
Thus the penalty level depends on the pivotal quantity $\Lambda(1-\alpha|X)$ and  the design.
  Under assumptions D.1-D.4 we can set $c=2$,
 similar to  \cite{BickelRitovTsybakov2009}'s choice for least squares.  Furthermore, we recommend computing $\Lambda(1-\alpha|X)$ using simulation of $\Lambda$.\footnote{We also provide analytical bounds on $\Lambda(1-\alpha|X)$  of the
form $ C(\alpha,\mathcal{U}) \sqrt{n\log p}$ for some numeric constant $C(\alpha,\mathcal{U})$. We recommend simulation because it accounts for correlation among the columns of $X$ in the sample.}  Our concrete recommendation for practice is to set $1-\alpha =0.9$.

The parameter $1-\alpha$ is the confidence level in the sense
 that, as in \cite{BickelRitovTsybakov2009}, our (non-asymptotic) bounds on the estimation error will contract
at the optimal rate with this probability.
 We refer the reader to Koenker \cite{Koenker2010}
for an implementation of our choice of penalty level and practical suggestions concerning the
confidence level.  In particular, both here and in Koenker \cite{Koenker2010}, the confidence level $1-\alpha =0.9$ gave good
performance results in terms of balancing regularization bias with estimation variance.
Cross-validation may also be used to choose the confidence level
$1-\alpha$.  Finally, we should note that, as in \cite{BickelRitovTsybakov2009},  our theoretical bounds
allow for any choice of $1-\alpha$ and are stated as a function of $1-\alpha$.

The formal rationale behind the choice (\ref{Def:LambdaPivotal00}) for the penalty level $\lambda$ is that this choice leads precisely to the optimal rates of convergence for $\ell_1$-QR.
(The same or slightly higher choice of $\lambda$
also guarantees  good performance of post-$\ell_1$-QR.)  Our general strategy for choosing $\lambda$ follows
\cite{BickelRitovTsybakov2009}, who recommend selecting $\lambda$
so that it dominates a relevant measure of noise in the sample criterion function, specifically the supremum norm of a suitably rescaled
gradient of the sample criterion function evaluated at the true parameter
value.    In our case this general strategy leads precisely to the choice
(\ref{Def:LambdaPivotal00}).  Indeed,  a (sub)gradient $\widehat S_u(\beta(u)) = \En[( u - 1\{y_i \leq
x_i'\beta(u)\})x_i] \in \partial \widehat Q_u (\beta(u))$ of the quantile
regression objective function evaluated at the truth has a pivotal
representation, namely  $\widehat S_u(\beta(u))=  \En[( u - 1\{u_i \leq u\})x_i]$ for $u_1,\ldots,u_n$ i.i.d. uniform $(0,1)$ conditional on $X$, and so
we can represent $\Lambda$ as in (\ref{almost sure}), and, thus, choose $\lambda$ as in (\ref{Def:LambdaPivotal00}).

\subsection{General Regularity Conditions}\label{Sec:Primitive}

We consider the following conditions on a sequence of models indexed by $n$ with parameter dimension $p=p_n \to \infty$.
In these conditions all constants can depend on $n$, but we omit the explicit indexing by $n$ to ease exposition.

\textbf{D.1.} \textit{Sampling and Smoothness.}   \textit{Data
$(y_i,x_i')', i=1,\ldots,n,$ are an i.i.d. sequence of real
$(1+p)$-vectors, with the conditional $u$-quantile function given
by (\ref{lin model}) for each $u \in \mathcal{U}$, with  the first component of $x_i$ equal to one, and  $n \wedge p \geq 3$. For each value  $x$ in the support of $x_i$, the conditional density $f_{y_i|x_i}(y|x)$ is continuously differentiable in $y$ at each $y \in \Bbb{R}$, and $f_{y_i|x_i}(y|x)$ and $\frac{\partial}{\partial y} f_{y_i|x_i}(y|x)$ are bounded in absolute value
by constants $\bar{f}$ and $\bar{f'}$, uniformly in $y\in \mathbb{R}$ and $x$ in the support $x_i$.  Moreover, the conditional
density of $y_i$ evaluated at the conditional quantile $x_i'\beta(u)$ is bounded away from zero uniformly in $\mathcal{U}$, that is  $f_{y_i|x_i}(x'\beta(u)|x)> \underf >0$  uniformly in $u \in \mathcal{U}$ and $x$ in the support of $x_i$.}

Condition D.1 imposes only mild smoothness assumptions on the conditional
density of the response variable given regressors, and does not impose
any normality or homoscedasticity assumptions\comment{ commonly made in the literature on HDSMs}.
The assumption that the conditional density is bounded below at the conditional
quantile is standard, but we can replace it by the slightly more general condition $\inf_{u \in \mathcal{U}} \inf_{\delta \neq 0} (\delta' J_u \delta)/(\delta'\Ep[x_i x_i'] \delta) \geq \underf>0,$ on the Jacobian matrices $$J_u = \Ep[f_{y_i|x_i}(x_i'\beta(u)|x_i)x_ix_i'] \ \ \mbox{for all} \ \ u \in \mathcal{U},$$
throughout the paper.

\textbf{D.2.} \textit{Sparsity and Smoothness of $u \mapsto
\beta(u)$.}\textit{ Let $\mathcal{U}$ be a compact subset of $(0,1)$. The
coefficients $\beta(u)$ in (\ref{lin model}) are sparse and smooth
with respect to $u\in \mathcal{U}$:
$$ \sup_{u \in \mathcal{U}} \| \beta(u) \|_0 \leq s \ \ \ \ \mbox{and} \ \ \ \ \|\beta(u) - \beta(u')\| \leq L|u-u'|, \ \ \ \mbox{for all} \ \ u, u' \in \mathcal{U}$$ where $ s \geq 1$, and $\log L \leq C_L \log (p\vee n)$ for some constant $C_L$.
}

Condition D.2 imposes sparsity and smoothness on the behavior
of the quantile regression coefficients $\beta(u)$  as we vary the
quantile index $u$.

\textbf{D.3.} \textit{Well-behaved Covariates.} \textit{Covariates are
normalized such that $\sigma_j^2 = \Ep[x_{ij}^2] = 1$ for all $j =
1,\ldots,p$, and $\hat
\sigma_j^2 = \En[x^2_{ij}]$ obeys $P(\max_{1\leq j \leq p} \left| \hat \sigma_{j} - 1 \right| \leq 1/2) \geq 1-\gamma \to 1$ as $n \to \infty$. }

Condition D.3 requires  that $\hat \sigma_j$ does not deviate too much
 from  $\sigma_j$ and normalizes $\sigma_j^2 =  1$.

In order to state the next assumption, for some $c_0 \geq 0$ and each $u \in \mathcal{U}$, define
$$ A_u:=\{ \delta
\in \RR^p : \ \|\delta_{T_u^c}\|_1 \leq c_0 \|\delta_{T_u}\|_1,
\|\delta_{T_u^c}\|_0\leq n \},$$
which will be referred to as the restricted  set.
Define $\overline T_u(\delta,m) \subset\{1,...,p\}\setminus T_u$
as the support of the $m$ largest in absolute value components of
the vector $\delta$ outside of $T_u=\supp(\beta(u))$, where $\overline
T_u(\delta,m)$ is the empty set if $ m =0$.

\textbf{D.4.} \textit{Restricted Identifiability and
Nonlinearity.} \textit{For some
constants $m \geq 0$ and $c_0 \geq 9$, the matrix
$\Ep[x_i x_i']$    satisfies
\begin{equation}\label{Def:RE}\tag{RE($c_0, m$)} \kappa^2_m := \inf_{u \in \mathcal{U}} \inf_{\delta \in A_u, \delta \neq 0} \frac{\delta' \Ep[x_i x_i']\delta}{\|\delta_{T_u\cup \overline{T}_u(\delta,m)}\|^2} >0  \end{equation}
and $\log (\underf \kappa_0^2) \leq C_f \log (n \vee p)$ for some constant $C_f$. Moreover,
\begin{equation}\label{compute q} q:= \frac{3}{8} \frac{\underf^{3/2}}{\bar{f'}}  \inf_{u\in\mathcal{U}}  \inf_{\delta\in A_u, \delta \neq 0} \frac{\Ep[|x_i'\delta|^2]^{3/2}}{\Ep[|x_i'\delta|^3]}>0. \tag{RNI($c_0$)}
\end{equation}}

The restricted eigenvalue (RE) condition is analogous to the condition
in \cite{BickelRitovTsybakov2009} and \cite{CandesTao2007}; see
 \cite{BickelRitovTsybakov2009} and \cite{CandesTao2007} for different sufficient primitive conditions that yield bounds on $\kappa_m$. \comment{prove that the RE
condition on the design matrix $\Ep[x_ix_i']$ is quite general, and is weaker than nearly all other
design conditions used in the least squares literature;}  Also, since $\kappa_m$ is non-increasing in $m$,
 $RE(c_0,m)$ for any $m>0$ implies $RE(c_0,0)$.
 The restricted non-linear impact (RNI) coefficient $q$ appearing in D.4 is a new concept, which controls the  quality of minoration of
 the quantile regression objective function by a quadratic function over
 the restricted set.

Finally, we state another condition needed to derive results on the post-model selected estimator.
In order to state the condition, define the sparse set $
\widetilde A_u(\widetilde m) = \{ \delta \in \RR^p :
\|\delta_{T_u^c}\|_0 \leq \widetilde m \}$ for $\widetilde m \geq
0$ and $u\in\mathcal{U}$.

\textbf{D.5.} \textit{Sparse Identifiability and Nonlinearity.}
The matrix $\Ep[x_i x_i']$ satisfies for some $\widetilde m \geq
0$:\begin{equation}\label{Def:SE}\tag{SE($\widetilde m$)}
\widetilde \kappa^2_{\widetilde m}:= \inf_{u \in \mathcal{U}}
\inf_{\delta \in \widetilde A_u(\widetilde m), \delta \neq 0} \frac{
\delta'\Ep[x_ix_i']\delta}{\delta'\delta}  >0,\end{equation}
and
\begin{equation}\label{Def:SNI}\tag{SNI($\widetilde m$)} \widetilde q_{\widetilde m} :=  \frac{3}{8} \frac{\uglyf^{3/2}}{\bar{f'}} \inf_{u \in \mathcal{U}}  \inf_{\delta \in \widetilde A_u(\widetilde m), \delta \neq 0}
 \frac{\Ep[|x_i'\delta|^2]^{3/2}}{\Ep[|x_i'\delta|^3]}>0.
\end{equation}

We invoke the sparse eigenvalue (SE) condition in order to analyze
the post-penalized estimator (\ref{Def:TwoStep}). This assumption is similar to the conditions used  in \cite{MY2007} and \cite{ZhangHuang2006} to analyze Lasso. Our form of the SE condition is neither less nor
more general than the RE condition. The SNI coefficient $\widetilde q_{\widetilde m}$ controls the quality of minoration of the quantile regression objective function by a quadratic function over sparse neighborhoods of the true parameter.


\subsection{Examples of Simple Sufficient Conditions}

In order to highlight the nature and usefulness of conditions D.1-D.5 it is instructive to state some simple
sufficient conditions (note that D.1-D.5 allow for much more
general conditions). We relegate the proofs of this section to the Supplementary Material Appendix \ref{Supplementary Material} for brevity.

{\sc Design 1: Location Model with Correlated Normal Design}. Let us consider estimating a standard  location model
$$ y = x'\beta^o + \varepsilon, $$
where $\varepsilon \sim N(0,\sigma^2)$, $\sigma>0$ is fixed, $x=( 1, z' )'$,
with $z \sim N(0, \Sigma),$ where $\Sigma$ has ones in the diagonal, a minimum eigenvalue bounded away from zero by a constant $\kappa^2>0$, and a maximum eigenvalue bounded from above, uniformly in $n$.

\begin{lemma}
Under Design 1 with $\mathcal{U} = [\xi,1-\xi]$, $\xi > 0$, conditions D.1-D.5 are satisfied with
$$\bar f = 1/[\sqrt{2\pi}\sigma], \ \ \bar f' = \sqrt{e/[2\pi]}/\sigma^2, \ \ \underf  = 1/\sqrt{2\pi\xi}\sigma, $$
$$\|\beta(u)\|_0 \leq\|\beta^o\|_0 + 1, \ \ \gamma = 2p\exp(-n/24), \ \ L =\sigma /\xi$$ $$ \kappa_m \wedge \widetilde \kappa_{\widetilde m} \geq \kappa,  \ \ q\wedge \widetilde q_{\widetilde m} \geq (3/[32 \xi^{3/4}])\sqrt{\sqrt{2\pi}\sigma/e} .$$
\end{lemma}

Note that the normality of errors can be easily relaxed by allowing for the disturbance $\varepsilon$ to have a smooth density
that obeys the conditions stated in D.1.  The conditions on the population design matrix can also be replaced
by more general primitive conditions specified in Remark 2.1.

{\sc Design 2: Location-scale model with bounded regressors}. Let us consider estimating a standard  location-scale model
$$ y = x'\beta^o + x'\eta \cdot  \varepsilon, $$
where $\varepsilon \sim F$ independent of $x$, with a continuously differentiable probability density function $f$. We assume that the population design matrix $\Ep[xx']$ has ones in the diagonal and has eigenvalues uniformly bounded away from zero and from above,  $x_1 =1$, $\max_{1\leq j\leq p} |x_{j}| \leq K_B$. Moreover, the vector $\eta$ is such that $0 < \upsilon \leq  x'\eta \leq \Upsilon < \infty$ for all values of $x$.

\begin{lemma}
Under Design 2 with $\mathcal{U} = [\xi,1-\xi]$, $\xi > 0$, conditions D.1-D.5 are satisfied with $$\bar f \leq \max_{\varepsilon} f(\varepsilon)/\upsilon, \ \ \bar f' \leq \max_\varepsilon f'(\varepsilon) / \upsilon^2, \ \ \underf =  \min_{u \in \mathcal{U}} f(F^{-1}(u))/\Upsilon, \ \ $$
$$\|\beta(u)\|_0 \leq\|\beta^o\|_0 + \|\eta\|_0 + 1, \ \ \gamma = 2p\exp(-n/[8K_B^4]), $$ $$ \kappa_m \wedge \widetilde \kappa_{\widetilde m} \geq \kappa, \ \ L =  \|\eta\| \underf$$ $$ q \geq
 \frac{3}{8} \frac{\underf^{3/2}}{\bar{f'}}  \kappa /[10 K_B \sqrt{s}],  \widetilde q_{\widetilde m} \geq   \frac{3}{8} \frac{\underf^{3/2}}{\bar{f'}}  \kappa /[ K_B \sqrt{s+\widetilde m}].$$
\end{lemma}

\begin{remark}(Conditions on $\Ep[x_i x_i']$).   The conditions on the population design matrix can also be replaced
by more general primitive conditions of the form stated in  \cite{BickelRitovTsybakov2009} and \cite{CandesTao2007}. For example, conditions on sparse eigenvalues suffice as shown in \cite{BickelRitovTsybakov2009}. \comment{Note that the condition that the population covariance matrix $\Sigma$ has eigenvalues bounded away from zero and from above can be relaxed to conditions on sparse eigenvalues as shown in \cite{BickelRitovTsybakov2009}.} Denote the minimum and maximum eigenvalue of the population design matrix by
\begin{equation}\label{Def:EigSparse}
\semin{m} = \min_{\|\delta\|= 1, \|\delta\|_0\leq m} \frac{\delta'\Ep\[x_ix_i'\]\delta}{\delta'\delta} \ \ \mbox{and} \ \
\semax{m} = \max_{\|\delta\|= 1, \|\delta\|_0\leq m} \frac{\delta'\Ep\[x_ix_i'\]\delta}{\delta'\delta}.
\end{equation}
Assuming that for some $m\geq s$ we have $m\semin{m+s}\geq c_0^2s\semax{m}$, then
$$ \kappa_m  \geq \sqrt{\semin{s+m}}\( 1 - c_0\sqrt{s\semax{s}/[ m\semin{s+m} ]}\) \ \ \mbox{and} \ \   \widetilde \kappa_{\widetilde m} \geq \semin{s+m}.$$
\end{remark}
\subsection{Overview of Main Results}
Here we discuss our results under the simple setup of Design 1 and under $1/p \leq \alpha \to 0$ and $\gamma \to 0$. These simple assumptions allow us to straightforwardly compare our rate results to those obtained in the literature. We state our more general non-asymptotic results under general
conditions in the subsequent sections. Our first main rate result is that $\ell_1$-QR, with
our  choice (\ref{Def:LambdaPivotal00}) of parameter
$\lambda$, satisfies
\begin{equation}\label{result 1}
 \sup_{u \in \mathcal{U} }\| \hat \beta(u) - \beta(u)\| \lesssim _P \frac{ 1}{\underf\kappa_0 \kappa_s} \sqrt{\frac{s  \log (n \vee p)}{n}},
 \end{equation}
provided that the upper bound on the number of non-zero components
$s$ satisfies
 \begin{equation}\label{restrict s}
\frac{\sqrt{s \log (n \vee p)}}{\sqrt{n} \ \underf^{1/2}\kappa_0 q} \to 0.
\end{equation}
Note that $\kappa_0$, $\kappa_s$, $\underf$, and $q$ are bounded away from zero in this
example. Therefore,   the rate of
convergence is $\sqrt{s/n} \cdot \sqrt{\log (n \vee p)}$ uniformly in
the set of quantile indices $u \in \mathcal{U}$,  which
is very close to the oracle rate when $p$ grows polynomially in $n$.   Further, we note that our resulting
restriction (\ref{restrict s}) on the dimension $s$ of the
true models is very weak; when $p$ is polynomial in $n$, $s$ can be
of almost the same order as $n$, namely $s = o( n / \log n ) $.

Our second main result is that the dimension $\|\hat \beta(u)\|_0
$ of the model selected by the $\ell_1$-penalized estimator is of
the same stochastic order as the dimension $s$ of the true
models, namely
 \begin{equation} \label{result 2}
\sup_{u \in \mathcal{U}}\|\hat \beta(u)\|_0 \lesssim_P s.
 \end{equation}
Further, if the parameter values of the minimal true
model are well separated from zero, then with a high
probability the model selected by the $\ell_1$-penalized estimator
correctly nests the true minimal model:
\begin{equation}\label{result 3}
T_u = \supp (\beta(u))  \subseteq \widehat T_u = \supp (\hat
\beta(u)), \text { for all } u \in \mathcal{U}.
\end{equation}
Moreover, we provide conditions under which a hard-thresholded version
of the estimator selects the correct support.

Our third main result is that the post-penalized estimator, which
applies ordinary quantile regression to the selected model,  obeys
\begin{equation}\label{result 2a}
\begin{array}{rcl}
\displaystyle\sup_{u \in \mathcal{U}} \| \widetilde \beta(u) -
\beta(u)\| & \lesssim_P & \displaystyle   \frac{1}{\uglyfh\widetilde
\kappa^2_{\widehat m}}\sqrt{\frac{\widehat m  \log (n \vee
p)+ s\log n}{n}} \ \ \  +\\
& + & \displaystyle \frac{\sup_{u \in \mathcal{U}}1\{T_u
\not\subseteq \widehat T_u \}}{\underf\kappa_0\widetilde\kappa_{\widehat
m}} \sqrt{\frac{s \log(n\vee p)}{n}},
\end{array}
 \end{equation} where  $\widehat m= \sup_{u\in\mathcal{U}}\|\hat\beta_{T_u^c}(u)\|_0$ is the maximum number of wrong components selected for any quantile index $u\in\mathcal{U}$, provided that the
bound on the number of non-zero components $s$ obeys the growth condition (\ref{restrict s}) and
\begin{equation}\label{restrict s 2}
\frac{\sqrt{
\widehat m \log (n \vee p) + s\log n}}{\sqrt{n} \ \uglyfh^{1/2} \widetilde
\kappa_{\widehat m} \widetilde q_{\widehat m}} \to_P 0.
 \end{equation}
(Note that when $\mathcal{U}$ is a singleton, the $s \log n$ factor in (\ref{result 2a}) becomes $s$.)

We see from (\ref{result 2a}) that
post-$\ell_1$-QR can perform well  in terms of the rate
of convergence even if the selected model $\widehat T_u$ fails to contain the true
model $T_u$.  Indeed, since in this design $\widehat m \lesssim_P s$, post-$\ell_1$-QR has the rate of convergence
$\sqrt{s/n} \cdot \sqrt{\log (n \vee p)}$, which is
 the same as the rate of convergence of $\ell_1$-QR. The intuition for this result
 is that the $\ell_1$-QR based model selection
 can only miss covariates with relatively small
coefficients, which then permits post-$\ell_1$-QR to
perform as well or even better due to reductions in bias, as
confirmed by our computational experiments.

We also see from (\ref{result 2a}) that
post-$\ell_1$-QR can perform better than $\ell_1$-QR in terms of the rate
of convergence if the number of wrong components selected obeys $\widehat m = o_P(s)$ and the selected model contains the true
model, $\{T_u \subseteq \widehat T_u\}$ with probability converging to one.  In this case post-$\ell_1$-QR has the rate of convergence $\sqrt{ (o_P(s)/n) \log (n \vee p) +  (s/n) \log n },$
which is faster than the rate of convergence of $\ell_1$-QR.  In the extreme case of perfect model
selection,  that is, when $\widehat m = 0$, the rate of post-$\ell_1$-QR becomes $\sqrt{(s/n) \log n}$ uniformly
in $\mathcal{U}$. (When $\mathcal{U}$ is a singleton, the $\log n$ factor drops out.)  Note that inclusion $\{T_u \subseteq \widehat T_u\}$ necessarily happens when the coefficients
of the true models are well separated from zero, as we stated above.
Note also that the condition $\widehat m = o(s)$ or even $\widehat m = 0$ could occur under additional
conditions on the regressors (such as the mutual
coherence conditions that restrict the maximal pairwise correlation
of regressors).   Finally, we note that our second
restriction (\ref{restrict s 2}) on the dimension $s$ of the true models is very weak in this design; when $p$ is polynomial in $n$, $s$ can be
of almost the same order as $n$, namely $s = o( n / \log n ) $.

To the best of our knowledge, all of the results presented above are new, both for the single $\ell_1$-penalized quantile regression problem as well as for the infinite collection of $\ell_1$-penalized quantile regression problems.  These results therefore contribute to the rate results obtained for $\ell_1$-penalized mean regression and related estimators in the fundamental papers of \cite{BickelRitovTsybakov2009,CandesTao2007,Koltchinskii2009,MY2007,vdGeer,ZhangHuang2006}.  The results on post-$\ell_1$ penalized quantile regression had no analogs in the literature on mean regression, apart from the rather exceptional case of perfect model selection, in which case the post-penalized estimator is simply the oracle. Building on the current work these results have been extended to mean regression in \cite{BC-PostLASSO}.
 Our results on the sparsity of $\ell_1$-QR and model selection also contribute to the analogous results for mean regression \cite{MY2007}.  Also, our rate results for $\ell_1$-QR are different from, and hence complementary to, the fundamental results in \cite{vdGeer} on the excess forecasting loss under possibly non-quadratic loss functions, which also specializes the results to density estimation, mean regression, and logistic regression.  In principle we could apply  theorems in \cite{vdGeer} to the single quantile regression problem to derive the bounds on the excess loss $\Ep[\rho_u(y_i - x_i'\hat \beta(u))] - \Ep[\rho_u(y_i - x_i'\beta(u))]$.\footnote{Of course, such a derivation would entail some difficult work,  since we must verify  some high-level assumptions made directly on the performance of the oracle and penalized estimators in population and others (cf. \cite{vdGeer}'s  conditions I.1 and I.2, where I.2 assumes uniform in $x_i$ consistency of the penalized estimator in the population, and does not hold in our main examples, e.g., in Design 1 with normal regressors.)} However, these bounds would not imply our results  (\ref{result 1}), (\ref{result 2a}), (\ref{result 2}), (\ref{result 3}),  and (\ref{Def:LambdaPivotal00}), which characterize the rates of estimating coefficients $\beta(u)$ by $\ell_1$-QR and post-$\ell_1$-QR,  sparsity and model selection properties,  and the data-driven choice of the penalty level.

\section{ Main Results and Main Proofs}\label{Sec:MainFix}

In this section we derive rates of convergence for $\ell_1$-QR and
post-$\ell_1$-QR, sparsity bounds, and model selection results.

\subsection{Bounds on $\Lambda(1-\alpha |  X)$ } We start with a characterization of $\Lambda$ and its $(1-\alpha)$-quantile,
$\Lambda(1-\alpha |  X)$, which determines the magnitude of
our suggested penalty level $\lambda$ via equation (\ref{Def:LambdaPivotal00}).

\begin{theorem}[Bounds on $\Lambda(1-\alpha |  X)$]\label{Thm:BoundLAMBDA}
Let $\W = \max_{u\in\mathcal{U}}1/\sqrt{u(1-u)}$.  There is a universal constant $C_\Lambda$ such that
\begin{enumerate}
\item[(i)] $\displaystyle P \( \Lambda \geq k \cdot C_\Lambda \  \W
\sqrt{n\log p} \ | X \) \leq p^{-k^2+1}$, \item[(ii)]$\displaystyle \Lambda(1-\alpha |  X)
\leq   \sqrt{ 1 + \log
(1/\alpha)/\log p} \cdot C_\Lambda \ \W  \sqrt{n\log p} $  with probability $1$.
\end{enumerate}
\end{theorem}

\subsection{Rates of Convergence}

In this section we establish the rate of convergence of
$\ell_1$-QR. We start with the following preliminary result
which shows that if the penalty level exceeds the specified threshold,
for each $u \in \mathcal{U}$, the estimator $\hat \beta(u) - \beta(u)$ will belong to the restricted set $ A_u:=\{ \delta
\in \RR^p : \ \|\delta_{T_u^c}\|_1 \leq c_0 \|\delta_{T_u}\|_1,
\|\delta_{T_u^c}\|_0\leq n \}$.

\begin{lemma}[Restricted Set]\label{Lemma:E1}
1. Under D.3, with probability at least $1-\gamma$ we have for every
$\delta \in \RR^p$ that \begin{equation}\label{E1normEquiv}
\frac{2}{3}\|\delta\|_{1,n} \leq \|\delta\|_{1} \leq 2
\|\delta\|_{1,n}.\end{equation}
2. Moreover, if for some $\alpha \in
(0,1)$
\begin{equation}\label{Def: lamnbda not} \lambda \geq \lambda_0 := \frac{c_0+3}{c_0-3} \Lambda(1-\alpha|X),
 \end{equation} then with
probability at least $1-\alpha-\gamma$, uniformly in $u \in
\mathcal{U}$, we have (\ref{E1normEquiv}) and
$$ \hat \beta(u) - \beta(u) \in A_u = \{ \delta \in \RR^p : \|\delta_{T_u^c}\|_1 \leq c_0 \|\delta_{T_u}\|_1, \|\delta_{T_u^c}\|_0 \leq n \}.$$
\end{lemma}
  This result is inspired \cite{BickelRitovTsybakov2009}'s analogous
result for least squares.

\begin{lemma}[Identifiability Relations over Restricted Set]\label{Lemma:E2}
Condition D.4, namely  ${\rm RE}(c_0, m)$ and  ${\rm RNI}(c_0)$, implies that for any $\delta \in A_u$ and $u \in \mathcal{U}$,
\begin{eqnarray}
 \label{E.2-RE.addon}
&&
\| (\Ep[x_i x_i'])^{1/2} \delta \| \leq \| J_u^{1/2} \delta \|/\underf^{1/2},
\\
\label{E.2-RE.0}
&& \| \delta_{T_u}\|_1 \leq \sqrt{s} \| J_u^{1/2}\delta\|  /
[\underf^{1/2} \kappa_0], \\
&&\label{E.2-RE.1}
\| \delta\|_1 \leq \sqrt{s}(1+c_0) \| J_u^{1/2}\delta\|  /
[\underf^{1/2} \kappa_0], \\
&&\label{E.2-RE.2}
 \| \delta\| \leq    \left(1+c_0\sqrt{s/m}\right)\|
 J_u^{1/2}\delta\|/[\underf^{1/2}\kappa_m],  \\
&&\label{E.2-RNI}
Q_u(\beta(u) + \delta) - Q_u(\beta(u)) \geq
(\|J_u^{1/2}\delta\|^2/4) \wedge  (q \|J_u^{1/2}\delta\|).
\end{eqnarray}
\end{lemma}
This second preliminary result derives identifiability relations
over $A_u$. It shows that the coefficients $\underf$, $\kappa_0$, and $\kappa_m$ control moduli of continuity
between various norms over the restricted set $A_u$, and the RNI coefficient  $q$
controls the quality of minoration of the objective function by a quadratic function
over $A_u$.

Finally, the third preliminary result derives bounds on the empirical error over $A_u$:
\begin{lemma}[Control of Empirical Error]\label{Lemma:E3}
Under D.1-4, for any $t>0$ let
$$\emperror(t):=\sup_{u \in \mathcal{U},  \delta \in A_u, \|J^{1/2}_u \delta \|\leq t} \left| \hat Q_u(\beta(u) + \delta) - Q_u(\beta(u) + \delta)-\left( \hat Q_u(\beta(u)) - Q_u(\beta(u))\right)\right|.$$
 Then, there is a universal constant $C_E$ such that for any $A > 1$, with probability at least $1-3\gamma-3p^{-A^2}$
 $$\emperror(t) \leq  t \cdot C_E \cdot \frac{ (1+c_0)A }{\underf^{1/2}\kappa_0}\sqrt{\frac{s \log( p\vee [L\underf^{1/2}\kappa_0/t])}{n}}.$$
\end{lemma}
In order to prove the lemma we use a combination of chaining arguments and exponential
inequalities for contractions \cite{LedouxTalagrandBook}.   Our use of the contraction principle is inspired by its
fundamentally innovative use in \cite{vdGeer}; however, the use of the contraction principle alone is not sufficient in our case. Indeed,  first
we need to  make some adjustments to obtain error bounds over the neighborhoods defined by the intrinsic norm $\|J^{1/2}_u \cdot \|$ instead of the $\|\cdot\|_1$ norm; and second, we need to use chaining over $u \in \mathcal{U}$ to get uniformity over $\mathcal{U}$.

Armed with Lemmas \ref{Lemma:E1}-\ref{Lemma:E3}, we establish the
first main result. The result depends on the constants $C_\Lambda$, $C_E$, $C_L$, and $C_f$ defined in Theorem \ref{Thm:BoundLAMBDA}, Lemma \ref{Lemma:E3}, D.2, and D.4.

%

\begin{theorem}[Uniform Bounds on Estimation Error of $\ell_1$-QR]\label{Thm:Rate-r2}
Assume conditions D.1-4 hold, and let $C> 2C_\Lambda \sqrt{ 1 + \log (1/\alpha)/\log p} \vee  [C_E\sqrt{1 \vee [C_L+C_f+1/2]}]$. Let $\lambda_0$ be defined as in (\ref{Def: lamnbda not}).
Then uniformly in the penalty level $\lambda$ such that
\begin{equation}\label{Cond:Lambda} \lambda_0
\leq  \lambda \leq C \cdot \W\sqrt{n\log
p,}\end{equation}
we have that, for any $A>1$ with probability at least $1-\alpha-4\gamma-3p^{-A^2}$,
$$ \sup_{u\in\mathcal{U}}\|J^{1/2}_u(\hat \beta(u) - \beta(u) ) \| \leq 8C \cdot \frac{(1+c_0)\W A}{\underf^{1/2}\kappa_0}\cdot \sqrt{\frac{s\log (p\vee n)}{n}}, \ $$
$$
\sup_{u\in\mathcal{U}} \sqrt{\Ep_{x}[x'(\hat \beta(u) - \beta(u) )]^2} \leq  8C \cdot \frac{(1+c_0)\W A}{\underf \kappa_0}\cdot
\sqrt{\frac{s\log (p\vee n)}{n}}, \ \mbox{and}
$$
$$ \sup_{u\in \mathcal{U}}\|\hat \beta(u) - \beta(u)\| \leq \frac{1 + c_0\sqrt{s/m}}{\kappa_m} \cdot 8C \cdot \frac{(1+c_0)\W A}{\underf \kappa_0}\cdot
\sqrt{\frac{s\log (p\vee n)}{n}},
$$
provided $s$ obeys the growth condition \begin{equation}\label{Rate:SideCondition}2C \cdot
(1+c_0)\W A \cdot \sqrt{s\log (p\vee n)}<
q \underf^{1/2}\kappa_0\sqrt{n}.\end{equation}
\end{theorem}

This result derives the rate of convergence of the
$\ell_1$-penalized quantile regression estimator in the intrinsic norm and other norms of interest uniformly in $u\in \mathcal{U}$
as well as uniformly in the penalty level $\lambda$ in the range specified by (\ref{Cond:Lambda}),
which includes our recommended choice of $\lambda_0$.  We see that the rates of convergence for $\ell_1$-QR generally depend on the
number of significant regressors $s$, the logarithm of the number
of regressors $p$, the strength of identification summarized by
$\kappa_0$, $\kappa_m$, $\underf$, and $q$, and the quantile indices of interest
$\mathcal{U}$ (as expected, extreme quantiles can slow down the rates of convergence).
These rate results parallel the results
of \cite{BickelRitovTsybakov2009} obtained for $\ell_1$-penalized
mean regression. Indeed, the role of the parameter $\underf$ is similar
to the role of the standard deviation of the disturbance in mean regression.  It is worth noting, however, that our results do not
rely on normality and homoscedasticity assumptions, and  our proofs have to address
the non-quadratic nature of the objective function, with parameter $q$ controlling
the quality of quadratization. This parameter $q$ enters the results
only through the growth restriction (\ref{Rate:SideCondition}) on $s$.   At this point we refer the reader to Section
\ref{Sec:Primitive} for a further discussion of this result
in the context of the correlated normal design. Finally, we note that our proof combines
the star-shaped geometry of the restricted set $A_u$
with classical convexity arguments; this insight may be of
interest in other problems.

\begin{proof}[Proof of Theorem \ref{Thm:Rate-r2}]
We let $$t := 8C \cdot
\frac{(1+c_0)\W A}{\underf^{1/2} \kappa_0}\cdot \sqrt{\frac{s\log (p\vee
n)}{n}},$$
and consider the following events:
\begin{enumerate}
\item [(i)] $\Omega_1:=$ the event  that (\ref{E1normEquiv})
and $\hat \beta (u) - \beta(u) \in A_u $, uniformly in $u \in
\mathcal{U}$, hold;
\item[(ii)] $\Omega_2:=$ the event that the bound on empirical error $\emperror(t)$ in
Lemma \ref{Lemma:E3} holds;
\item[(iii)] $\Omega_3:=$ the event in which $\Lambda(1-\alpha|X) \leq
 \sqrt{ 1 + \log (1/\alpha)/\log p} \cdot C_\Lambda  \ \W \sqrt{n\log p} $.
 \end{enumerate}
By the choice of $\lambda$ and Lemma \ref{Lemma:E1}, $P(\Omega_1)\geq 1-\alpha-\gamma$; by Lemma \ref{Lemma:E3} $P(\Omega_2) \geq 1 - 3\gamma - 3p^{-A^2}$; and
 by Theorem \ref{Thm:BoundLAMBDA}  $P(\Omega_3) = 1 $, hence $P(\cap_{k=1}^3 \Omega_k) \geq 1
-\alpha - 4\gamma - 3p^{-A^2}.$

Given the event $\cap_{k=1}^3 \Omega_k$,  we
want to show  the event that
\begin{equation}\label{EVENT:imp} \exists u \in \mathcal{U}, \ \ \|J_u^{1/2}(\hat \beta(u) - \beta(u))\| > t\end{equation} is impossible, which will prove the first bound. The other two bounds then follow from Lemma \ref{Lemma:E2} and the first bound.  First note that the event in (\ref{EVENT:imp}) implies that for some $u \in \mathcal{U}$
 $$ 0> \min_{\delta \in A_u, \|J_u^{1/2}\delta\|\geq t}\hat Q_u(\beta(u)+\delta)  -  \hat Q_u(\beta(u)) + \frac{\lambda\sqrt{u(1-u)}}{n}\left(  \|\beta(u)+\delta\|_{1,n} - \|\beta(u)\|_{1,n}\right).$$ The key observation is that by
convexity of $\hat Q_u(\cdot) +
\|\cdot\|_{1,n}\lambda\sqrt{u(1-u)}/n$ and by the fact that $A_u$ is
a cone,  we can replace  $\|J_u^{1/2}\delta\|\geq t$  by  $\|J_u^{1/2}\delta\|=t$ in the above inequality
and still preserve it:
\begin{eqnarray*} 0 & > & \min_{\delta \in A_u, \|J_u^{1/2}\delta\|=t}\hat Q_u(\beta(u)+\delta)  -  \hat Q_u(\beta(u)) + \frac{\lambda\sqrt{u(1-u)}}{n}\left(  \|\beta(u)+\delta\|_{1,n} - \|\beta(u)\|_{1,n}\right).
\end{eqnarray*}
Also, by inequality (\ref{E.2-RE.0}) in Lemma \ref{Lemma:E2},
 for each $\delta \in A_u$
\begin{eqnarray*}\|\beta(u)\|_{1,n} - \|\beta(u)+\delta\|_{1,n} &
\leq & \|\delta_{T_u}\|_{1,n} \leq 2\|\delta_{T_u}\|_{1}  \leq
2\sqrt{s}\|J^{1/2}_u\delta\|/\underf^{1/2} \kappa_0,\end{eqnarray*}
which then further implies
\begin{eqnarray} \label{EQ: tessie}
0 & > & \min_{\delta \in A_u, \|J_u^{1/2}\delta\|= t}\hat Q_u(\beta(u)+\delta)  -  \hat Q_u(\beta(u)) - \frac{\lambda\sqrt{u(1-u)}}{n}\frac{2\sqrt{s}}{\underf^{1/2} \kappa_0}\|J^{1/2}_u\delta\|.
\end{eqnarray}
Also by Lemma \ref{Lemma:E3}, under our choice of $t
\geq 1/[\underf^{1/2}\kappa_0\sqrt{n}]$, $\log(L\underf\kappa_0^2)\leq (C_L+C_f) \log(n\vee p)$, and under event $\Omega_2$
\begin{equation}\label{EQ: mana}\emperror(t) \leq t C_E\sqrt{1 \vee [C_L+C_f+1/2]}\frac{ (1+c_0)A }{\underf^{1/2} \kappa_0} \sqrt{\frac{s\log(p\vee n)}{n}}.\end{equation}
Therefore, we obtain from (\ref{EQ: tessie}) and (\ref{EQ: mana})
\begin{eqnarray*} 0 & \geq  \displaystyle \min_{\delta \in A_u, \|J_u^{1/2}\delta\|= t} &  Q_u(\beta(u)+\delta)  -  Q_u(\beta(u)) - \frac{\lambda\sqrt{u(1-u)}}{n}\frac{2\sqrt{s}}{\underf^{1/2} \kappa_0}\|J^{1/2}_u\delta\|- \\
& & - t \ C_E\sqrt{1\vee [C_L+C_f+1/2]}\frac{ (1+c_0)A }{\underf^{1/2} \kappa_0} \sqrt{\frac{s\log(p\vee n)}{n}}.
\end{eqnarray*}
Using the identifiability relation  (\ref{E.2-RNI}) stated in Lemma \ref{Lemma:E2}, we further get
\begin{eqnarray*}
0 & > &  \frac{t^2}{4} \wedge (q t) - t \ \frac{\lambda\sqrt{u(1-u)}}{n}\frac{2\sqrt{s}}{\underf^{1/2} \kappa_0} -t \ C_E\sqrt{1 \vee [C_L+C_f+1/2]}\frac{ (1+c_0)A }{\underf^{1/2} \kappa_0} \sqrt{\frac{s\log(p\vee n)}{n}}.
\end{eqnarray*}
Using  the upper bound on $\lambda$ under event $\Omega_3$, we obtain
\begin{eqnarray*}
0 & > &  \frac{t^2}{4} \wedge (q t) - t \ C \ \frac{2\sqrt{s\log p}}{\sqrt{n}}\frac{\W}{\underf^{1/2} \kappa_0}-t \ C_E\sqrt{1 \vee [C_L+C_f + 1/2]}\frac{ (1+c_0)A }{\underf^{1/2} \kappa_0} \sqrt{\frac{s\log(p\vee n)}{n}}.
\end{eqnarray*}
Note that $qt$ cannot be smaller than $t^2/4$ under the growth condition (\ref{Rate:SideCondition}) in the theorem.
Thus, using also the lower bound on $C$  given in the theorem, $\W \geq 1$, and $c_0 \geq 1$, we obtain the relation
$$ 0 > \frac{t^2}{4} - t\cdot 2C
\ \frac{(1+c_0)\W A}{\underf^{1/2} \kappa_0} \cdot \sqrt{\frac{s\log (p\vee
n)}{n}} =0, $$ which is impossible.
\end{proof}

\subsection{Sparsity Properties}

Next, we derive sparsity properties of the solution to
$\ell_1$-penalized quantile regression.  Fundamentally, sparsity is linked to the first order optimality conditions of (\ref{Def:L1QR}) and therefore to the (sub)gradient of the criterion function. In the case of least squares, the gradient is a
smooth  (linear) function of the parameters. In the case of quantile regression, the gradient is a highly non-smooth (piece-wise
constant) function. To control the sparsity of $\hat\beta(u)$ we rely on empirical process arguments to approximate gradients by
smooth functions. In particular, we crucially exploit the fact
that the entropy of all $m$-dimensional submodels of the
$p$-dimensional model is of order $m \log p$, which depends on $p$
only logarithmically.

The statement of the results will depend on the maximal
$k$-sparse eigenvalue of $\Ep\[x_ix_i'\]$ and  $\En\[x_ix_i'\]$:
\begin{equation}\label{Def:EigSparse}
\semax{k} = \max_{\delta \neq 0,\|\delta\|_0 \leq k}
\frac{\Ep\[(x_i'\delta)^2\]}{\delta'\delta} \ \ \mbox{and} \ \ \phi(k) =
\sup_{\delta \neq 0, \|\delta\|_0\leq k} \frac{\En\[ (x_i'\delta)^2\]}{\delta'\delta}
\vee \frac{\Ep\[(x_i'\delta)^2\]}{\delta'\delta}.
\end{equation}
In order to establish our main sparsity result, we need two preliminary lemmas.

\begin{lemma}[Empirical Pre-Sparsity]\label{Lemma:E4}
Let $\hs  = \sup_{u\in\mathcal{U}}\|\hat\beta(u)\|_0$. Under
D.1-4, for any $\lambda > 0$, with probability at least $1-\gamma$
we have $$\hs \leq n \wedge p \wedge [4n^2\phi(\hs)\W^2/\lambda^2].$$
In particular, if $\lambda \geq 2\sqrt{2}\W\sqrt{ n \log(n \vee p) \phi( n/\log(n\vee p) )}$ then
$ \hs \leq n/\log(n \vee p)$.

\comment{Moreover, for any $K \geq 2\sqrt{2}$ let  $\nn = 8n/ ( K^2 \log(n
\vee p))$. If $\lambda \geq K\W\sqrt{ n \log(n \vee p) \phi( \nn
)}$
$$ \hs \leq \nn = \frac{8n}{K^2\log(n \vee p)}\leq \frac{n}{ \log(n \vee p)}.$$} \end{lemma}

This lemma establishes an initial bound on the number of non-zero
components $\widehat s$  as a function of $\lambda$ and $\phi(\widehat s)$.
Restricting $\lambda \geq 2\sqrt{2}\W\sqrt{ n \log(n \vee p) \phi( n/\log(n\vee p))}$ makes the term
$\phi\left(n/\log(n\vee p)\right)$ appear in subsequent bounds
instead of the term $\phi(n)$, which in turn weakens some
assumptions.  Indeed, not only is the first term
smaller than the second, but also there are designs of interest
where the second term diverges while the first does not; for
instance, in Design 1, if $p \geq 2n$, we have
$\phi(n/\log (n\vee p)) \lesssim_P 1$ while $\phi(n) \gtrsim_P
\sqrt{\log p}$ by the Supplementary Material Appendix \ref{Supplementary Material}.

The following lemma establishes a bound on the sparsity as a function of
the rate of convergence.

\begin{lemma}[Empirical Sparsity]\label{Lemma:E5}
Assume D.1-4 and let $\ r =
\sup_{u\in\mathcal{U}}\|J_u^{1/2}(\hat\beta(u)-\beta(u))\|.$
Then, for any $\varepsilon>0$, there is a constant $K_\varepsilon\geq \sqrt{2}$
such that with probability at least $1-\varepsilon-\gamma$
$$ \frac{\sqrt{\hs}}{\W}  \leq  \mmu(\hs) \ \frac{n}{\lambda}  (r \wedge 1)  +
\sqrt{\hs} \ K_\varepsilon \frac{\sqrt{n\log (n\vee p)
\phi(\hs)}}{\lambda},  \ \ \mmu(k):=2\sqrt{\semax{k}} \(1 \vee
2\bar{f}/\underf^{1/2}\).$$
\end{lemma}

Finally, we combine these results to
establish the main sparsity result.  In what follows, we define
$\bar \phi_{\varepsilon}$ as a constant such that $\phi(n/\log(n\vee p)) \leq \bar \phi_{\varepsilon}$ with probability
$1-\varepsilon.$

\begin{theorem}[Uniform Sparsity Bounds]\label{Thm:Sparsity}
Let $\varepsilon>0$ be any constant, assume D.1-4 hold, and let  $\lambda$ satisfy $\lambda
\geq \lambda_0$ and
$$
K\W\sqrt{n\log(n \vee p)} \leq  \lambda \leq K'\W\sqrt{n\log(n \vee p)}
$$
for  some constant $K' \geq K \geq 2K_\varepsilon  \bar \phi^{1/2}_{\varepsilon}$, for $K_\varepsilon$ defined in Lemma
\ref{Lemma:E5}. Then, for any $A>1$ with probability at least
$1-\alpha - 2\varepsilon-4\gamma-p^{-A^2}$
$$\begin{array}{rcl}
\displaystyle  \hs:=  \sup_{u\in\mathcal{U}}
\|\hat\beta(u)\|_0  \leq &\displaystyle   s \cdot  \left[16\mmu \W /\underf^{1/2} \kappa_0\right]^2 [(1+c_0)A K'/K]^2, \\
 \end{array} $$where $\mmu := \mmu(n/\log(n\vee p))$,
 provided that $s$ obeys the growth condition
\begin{equation}\label{Sparsity:SideCondition}  2K'(1+c_0)A\W\sqrt{s \log(n \vee p)} < q \underf^{1/2}\kappa_0
 \sqrt{n}.\end{equation}
\end{theorem}

The theorem states that by setting the penalty level $\lambda$ to be
possibly higher than our initial recommended choice
$\lambda_0$, we can control $\widehat s$, which will be crucial
for good performance of the post-penalized estimator.  As a
corollary, we note that if (a) $\mmu \lesssim 1$, (b) $1/(\underf ^{1/2}\kappa_0)
\lesssim 1$, and  (c) $\bar \phi_{\varepsilon} \lesssim 1$ for each $\varepsilon>0$, then $\widehat s \lesssim s$
with a high probability, so the dimension of the selected model is about the same as the
dimension of the true model.   Conditions (a),
(b), and (c) easily hold for the correlated normal design in Design 1. In particular, (c)  follows from the concentration inequalities and from results in  classical random matrix theory; see the Supplementary Material Appendix \ref{Supplementary Material} for proofs. Therefore the possibly higher $\lambda$ needed to achieve the stated sparsity bound does not slow down the rate of $\ell_1$-QR in this case. The growth condition
(\ref{Sparsity:SideCondition}) on $s$ is also weak in this case.

\begin{proof}[Proof of Theorem \ref{Thm:Sparsity}]
By the choice of $K$ and Lemma
\ref{Lemma:E4},   $\hs \leq n/\log(n\vee p)$ with
probability $1-\varepsilon$. With at least the same probability,   the choice of $\lambda$ yields
$$ K_\varepsilon \frac{\sqrt{n\log (n\vee p) \phi(\hs)}}{\lambda} \leq \frac{K_\varepsilon \bar\phi^{1/2}_\varepsilon}{K\W} \leq \frac{1}{2\W},$$
so that by virtue of Lemma \ref{Lemma:E5} and by $\mmu(\hs) \leq
\mu:=\mmu(n/\log(n\vee p))$,
$$  \frac{\sqrt{\hs}}{\W} \leq \mmu \frac{( r \wedge 1)n}{\lambda} + \frac{\sqrt{\hs}}{2\W} \  \ \text{ or }  \ \  \frac{\sqrt{\hs}}{\W} \leq 2\mmu \frac{( r \wedge 1)n}{\lambda},$$
with probability $1-2\varepsilon$. Since all conditions of Theorem
\ref{Thm:Rate-r2} hold, we obtain the result by plugging in
the upper bound on $r =
\sup_{u\in\mathcal{U}}\|J_u^{1/2}(\hat\beta(u) - \beta(u))\|$ from Theorem 2.
\end{proof}

\subsection{Model Selection Properties}

Next we turn to the model selection properties of $\ell_1$-QR.

\begin{theorem}[Model Selection Properties  of $\ell_1$-QR]\label{Thm:Selection} Let $r^o = \sup_{u \in \mathcal{U}}\|\hat\beta(u) - \beta(u)\|$. If
 $\inf_{u\in\mathcal{U}}\min_{
j\in T_u } |\beta_j(u)|
>  r^o$,
then \begin{equation}\label{Eq:support}
T_u := \supp (\beta(u))  \subseteq \widehat T_u :=\supp (\hat
\beta(u)) \ \ \mbox{for all} \ u \in \mathcal{U}.
\end{equation}
Moreover, the hard-thresholded estimator $\bar \beta(u)$, defined
for any $\gamma \geq 0$ by
 \begin{equation}\label{Def:ThEst}\bar \beta_j(u) = \hat \beta_j(u) 1\left\{ |\hat \beta_j(u)| >  \gamma \right\}, \ u \in \mathcal{U}, \ j=1,\ldots,p, \end{equation} provided that $\gamma$ is chosen such that $r^o < \gamma <\inf_{u\in\mathcal{U}}\min_{ j\in T_u }|\beta_j(u)|-r^o$, satisfies  $$ \supp (\bar \beta(u))  = T_u\ \ \mbox{for all} \ u \in \mathcal{U}.$$
\end{theorem}

These results parallel analogous results in \cite{MY2007} for mean regression.  The first result says that if  non-zero coefficients
are well separated from zero,  then the support of $\ell_1$-QR includes the support of the true model.   The inclusion of the true
support in (\ref{Eq:support}) is in general one-sided; the support of the estimator can
include some unnecessary components having true coefficients
equal to zero. The second result states that if the further conditions are satisfied, additional hard thresholding can eliminate inclusions of such unnecessary components. The value of the hard threshold  must explicitly depend on the unknown value
$\min_{ j\in T_u }|\beta_j(u)|$,  characterizing the separation of non-zero
coefficients from zero.  The additional conditions
stated in this theorem are strong and perfect model selection appears quite unlikely in practice. Certainly it does not work in all real empirical examples we have explored. This  motivates our analysis of the post-model-selected estimator under conditions that allow for imperfect model selection, including cases where we miss some non-zero components or have additional unnecessary components.

\subsection{The post-penalized estimator}

In this section we establish a bound on the rate of convergence of the post-penalized estimator.
The proof relies crucially on the identifiability and
control of the empirical error over the sparse sets $\widetilde A_u(\widetilde m):= \{ \delta \in \RR^p \ : \
\| \delta_{T_u^c} \|_0 \leq \widetilde m \}.$

\begin{lemma}[Sparse Identifiability and Control of Empirical Error]\label{Lemma:E.6}
1. Suppose D.1 and D.5 hold. Then for all $\delta \in \widetilde A_u(\widetilde m), u \in \mathcal{U}$, and $\widetilde m\leq n$, we have
that
\begin{equation}\label{E.2-RNI-tilde-q}
Q_u(\beta(u) + \delta) - Q_u(\beta(u)) \geq
\frac{\|J_u^{1/2}\delta\|^2}{4} \wedge  \left( \widetilde q_{\widetilde
m} \|J_u^{1/2}\delta\| \right).
\end{equation}
2. Suppose D.1-2 and D.5 hold and that $|\cup_{u\in\mathcal{U}}T_u| \leq n$. Then for any  $\varepsilon>0$, there is a constant
$C_\varepsilon$ such that with probability at least
$1-\varepsilon$ the empirical error
$$\emperror_u(\delta):= \left| \hat
Q_u(\beta(u) +\delta) - Q_u(\beta(u)+\delta) - \( \hat
Q_u(\beta(u)) - Q_u(\beta(u)) \)  \right|$$
obeys $$ \sup_{u \in \mathcal{U}, \delta \in
\widetilde A_u(\widetilde m), \delta \neq 0}
\frac{\emperror_u(\delta)}{\|\delta\|} \leq C_\varepsilon
\sqrt{\frac{(\widetilde m \log (n\vee p)+ s\log n)\phi(\widetilde
m+s)}{n}} \text{ for all } \widetilde m \leq n.$$
\end{lemma}
In order to prove this lemma we exploit the crucial fact
that the entropy of all $m$-dimensional submodels of the
$p$-dimensional model is of order $m \log p$, which depends on $p$
only logarithmically. The following theorem establishes the properties of post-model-selection estimators.

\begin{theorem}[Uniform Bounds on Estimation Error of post-$\ell_1$-QR]\label{Thm:MainTwoStep}  Assume the conditions of Theorem \ref{Thm:Rate-r2} hold, assume that
$|\cup_{u\in\mathcal{U}}T_u| \leq n$,  and assume D.5 holds with
$\widehat m := \sup_{u\in\mathcal{U}}\|\hat\beta_{T_u^c}(u)\|_0$
with probability $1-\varepsilon$. Then for any $\varepsilon >0$ there is a constant
$C_\varepsilon$ such that the bounds
\begin{eqnarray}\label{Eq: bound on two step}
&& \ \ \begin{array}{rl}
\displaystyle\sup_{u\in \mathcal{U}} \left\|J^{1/2}_u( \widetilde\beta(u) -
\beta(u) )\right\| &\displaystyle  \leq    \frac{4C_\varepsilon \sqrt{\phi(\widehat m + s )}}{\uglyfh^{1/2}\widetilde\kappa_{\widehat m}} \cdot  \sqrt{ \frac{ \widehat m
\log(n\vee p) + s \log n }{n}} + \\
& \displaystyle + \sup_{u \in \mathcal{U}} 1\{T_u\not\subseteq \widehat T_u \} \cdot \frac{4\sqrt{2  (1+c_0) A}}{\underf^{1/2}\kappa_0}  \cdot C\cdot \W\sqrt{\frac{s\log(n\vee p)}{n}},
\end{array} \\
& & \ \ \sup_{u\in\mathcal{U}} \sqrt{\Ep_{x}[x'(\widetilde \beta(u) - \beta(u) )]^2} \leq \sup_{u\in \mathcal{U}} \left\|J^{1/2}_u( \widetilde\beta(u) - \beta(u) )\right\|/\uglyfh^{1/2},  \nonumber \\
& &  \ \ \sup_{u\in \mathcal{U}} \left\|  \widetilde\beta(u) - \beta(u) \right\|\leq \sup_{u\in \mathcal{U}} \left\|J^{1/2}_u( \widetilde\beta(u) - \beta(u) )\right\|/\uglyfh^{1/2}\widetilde \kappa_{\widehat m}, \nonumber
\end{eqnarray}
hold with probability at least $1-\alpha-3\gamma-3p^{-A^2}-2\varepsilon$,  provided
that $s$ obeys the growth condition $$\widetilde q_{\widehat m} \frac{C_\varepsilon \sqrt{ ( \widehat m
\log(n\vee p) + s \log n ) \phi(\widehat m + s )}}{\sqrt{n}\uglyfh^{1/2}\widetilde\kappa_{\widehat m}} + \sup_{u \in \mathcal{U}}  1\{T_u \not\subseteq \widehat T_u\}  2 A (1+c_0) \cdot C^2 \W^2 \cdot  \frac{s\log (p\vee n)}{n\underf\kappa_0^2} \leq
\widetilde q^2_{\widehat m}.$$
\end{theorem}
This theorem describes the performance of  post-$\ell_1$-QR. However, an inspection of the proof reveals that it can be applied to any post-model selection estimator.  From Theorem \ref{Thm:MainTwoStep}  we can conclude  that in many interesting
cases the rates of post-$\ell_1$-QR could be the same or faster than the rate of $\ell_1$-QR. Indeed, first consider the case where the model selection fails to contain the true model,
i.e., $\sup_{u \in \mathcal{U}}1\{T_u \not\subseteq \widehat T_u\}=1$ with a non-negligible probability.  If (a) $\widehat m \leq \widehat s \lesssim_P s$, (b) $\phi(\widehat m + s) \lesssim_P 1$,  and (c) the constants  $\uglyfh$ and $\widetilde \kappa_{\widehat m}^2$ are of the same order as $\underf$ and $\kappa_0\kappa_{m}$, respectively, then the rate of convergence of post-$\ell_1$-QR is the same as the rate of convergence
of $\ell_1$-QR.  Recall that Theorem \ref{Thm:Sparsity} provides sufficient conditions
needed to achieve (a), which hold in Design 1.  Recall also that in Design 1, (b) holds by concentration of measure
and classical results in random matrix theory, as shown in the Supplementary Material Appendix \ref{Supplementary Material}, and (c) holds by the calculations presented in Section 2.   This verifies our claim
regarding the performance of post-$\ell_1$-QR in the overview, Section 2.4.   The intuition for this
result is that even though $\ell_1$-QR misses true components, it does not miss very important ones, allowing post-$\ell_1$-QR still to perform well.
Second, consider the case where the model selection succeeds in containing the true model, i.e., $\sup_{u \in \mathcal{U}}1\{T_u \not\subseteq \widehat T_u\}=0$ with probability approaching one,
 and that the number of unnecessary components obeys $\widehat m = o_P(s)$. In this case
the rate of convergence of post-$\ell_1$-QR can be faster than the rate of convergence
of $\ell_1$-QR.  In the extreme case of perfect model selection, when $\widehat m = 0$ with a high probability, post-$\ell_1$-QR becomes the oracle estimator with a high probability.
We refer the reader to Section \ref{Sec:Setting} for further discussion, and note that this result could be of interest in other problems.



\begin{proof}[Proof of Theorem \ref{Thm:MainTwoStep}]
Let $$ \ \hat \delta(u) = \hat \beta(u) - \beta(u), \ \tilde \delta(u) := \widetilde
\beta(u) - \beta(u), \   t_u := \| J_u^{1/2} \tilde \delta(u)\|,$$ and $B_n$ be a random variable such that
$ B_n = \sup_{u \in \mathcal{U}}\widehat Q_u(\hat \beta(u)) - \widehat Q_u(\beta(u)).$ By the optimality of $\hat
\beta(u)$ in (\ref{Def:ThEst}), with probability $1-\gamma$ we
have uniformly in $u \in \mathcal{U}$  \begin{equation}\label{endarray} \begin{array}{rcl}
& & \displaystyle\hat Q_u(\hat \beta(u) ) - \hat Q_u(\beta(u) ) \leq  \frac{\lambda\sqrt{u(1-u)}}{n}( \| \beta(u)\|_{1,n} - \|\hat\beta(u)\|_{1,n})   \\
\\
& & \displaystyle \leq   \frac{\lambda\sqrt{u(1-u)}}{n} \|\widehat \delta_{T_u}(u) \|_{1,n} \leq
 \frac{\lambda\sqrt{u(1-u)}}{n} 2\|\widehat \delta_{T_u}(u)\|_1,
\end{array}
\end{equation}
where the last term in (\ref{endarray}) is bounded  by
 \begin{equation}\label{endarray2} \begin{array}{rcl}
 & & \displaystyle \frac{\lambda\sqrt{u(1-u)}}{n} \frac{2\sqrt{s}\|J^{1/2}_u\widehat \delta(u) \|}{\underf^{1/2}\kappa_0} \leq \frac{\lambda\sqrt{u(1-u)}}{n} \frac{2\sqrt{s}}{\underf^{1/2}\kappa_0} \sup_{u\in\mathcal{U}}\|J_u^{1/2}(\widehat\beta(u)-\beta(u))\|, \\
\end{array}\end{equation}
 using that $\|J^{1/2}_u\widehat \delta(u) \| \geq \underf^{1/2}\kappa_0
\|\widehat \delta_{T_u}(u)\|$ from RE$(c_0,0)$ implied by D.4. Therefore, by Theorem \ref{Thm:Rate-r2} we have
$$ B_n \leq \frac{\lambda\sqrt{u(1-u)}}{n} \frac{2\sqrt{s}}{\underf^{1/2}\kappa_0} 8C \cdot \frac{(1+c_0)\W A}{\underf^{1/2}\kappa_0}\cdot \sqrt{\frac{s\log (p\vee n)}{n}}$$
with probability $1-\alpha-3\gamma-3p^{-A^2}$.

For every $u \in \mathcal{U}$, by optimality of $\widetilde
\beta(u)$ in (\ref{Def:TwoStep}),
\begin{equation}\label{2step:Rel1} \hat Q_u( \widetilde \beta(u) )
- \hat Q_u(\beta(u)) \leq  1\{ T_u \not\subseteq \widehat T_u \}
\left(\hat Q_u(\hat \beta(u)) - \hat Q_u(\beta(u))\right)  \leq
1\{ T_u \not\subseteq \widehat T_u \} B_n.\end{equation} Also, by
Lemma \ref{Lemma:E.6}, with probability at least $1-\varepsilon$,
we have
\begin{equation}\label{2stepRel3} \sup_{u\in \mathcal{U}}
\frac{\emperror_u( \widetilde \delta(u))}{ \|\widetilde
\delta(u)\|} \leq C_\varepsilon \sqrt{\frac{{(\widehat m \log
(n\vee p)+ s\log n)\phi(\widehat
m+s)}}{{n}}}=:A_{\varepsilon,n}.\end{equation}
Recall that $ \sup_{u\in \mathcal{U}}\| \widetilde
\delta_{T_u^c}(u)\| \leq \widehat m\leq n$ so that by D.5
$t_u \geq \uglyfh^{1/2}\widetilde \kappa_{\widehat
m} \|\widetilde \delta(u)\|$ for all $u \in \mathcal{U}$ with
probability $1-\varepsilon$. Thus, combining relations
(\ref{2step:Rel1}) and (\ref{2stepRel3}), for every $u \in
\mathcal{U}$  $$ \begin{array}{rcl} \displaystyle
Q_u(\widetilde \beta(u)) - Q_u(\beta(u)) &\leq &
\displaystyle  t_u A_{\varepsilon,n}/[\uglyfh^{1/2}\widetilde \kappa_{\widehat m}] + 1\{ T_u \not\subseteq \widehat T_u \} B_n\\
 \end{array}
 $$ with probability at least $1-2\varepsilon$.
Invoking the sparse identifiability relation  (\ref{E.2-RNI-tilde-q}) of Lemma \ref{Lemma:E.6}, with the same probability, for all $u \in
\mathcal{U}$,  $$ \begin{array}{rcl} \displaystyle
(t_u^2/4) \wedge \left( \widetilde q_{\widehat m}t_u \right) & \leq & \displaystyle  t_u A_{\varepsilon,n}/[\uglyfh^{1/2}\widetilde \kappa_{\widehat m}] + 1\{ T_u \not\subseteq \widehat T_u \} B_n.\\
 \end{array}$$
We then conclude that under the assumed growth condition on $s$, this inequality
implies $$t_u \leq 4A_{\varepsilon,n}/[\uglyfh^{1/2}\widetilde \kappa_{\widehat m}] +
1\{ T_u \not\subseteq \widehat T_u \}\sqrt{4B_n\vee 0}$$
for every $u \in \mathcal{U}$, and the bounds
stated in the theorem now
follow from the definition of $\underf$ and $\tilde \kappa_m$. 
\end{proof}

\section{Empirical Performance}\label{Sec:Simulation}

In order to access the finite sample practical performance of the
proposed estimators, we conducted a Monte Carlo study and an
application to international economic growth.

\subsection{Monte Carlo Simulation}\label{Sec:MC}

In order to assess the finite sample practical performance of the
proposed estimators, we conducted a Monte Carlo study.
We will compare the performance of the $\ell_1$-penalized,
post-$\ell_1$-penalized, and the ideal oracle quantile regression estimators.
Recall that the post-penalized estimator applies
canonical quantile regression to the model selected by the
penalized estimator.  The oracle estimator applies canonical
quantile regression to the true model. (Of course, such an
estimator is not available outside Monte Carlo experiments.) We
focus our attention on the model selection properties of the penalized
estimator and biases and empirical risks of these estimators.

We begin by considering the following regression model:
$$ y = x'\beta(0.5) + \varepsilon, \  \ \beta(0.5) =(1,1,1/2,1/3,1/4,1/5,0,\ldots,0)', $$
where as in Design 1, $x = (1,z')'$ consists of an intercept and covariates $z \sim N(0,\Sigma)$,
and the errors $\varepsilon$ are independently and identically
distributed $\varepsilon \sim N(0,\sigma^2)$. The dimension $p$ of covariates $x$ is $500$, and the dimension $s$ of the true model is $6$, and the sample size $n$ is $100$.  We set the regularization parameter $\lambda$ equal to the $0.9$-quantile of
the pivotal random variable $\Lambda$, following our proposal in
Section \ref{Sec:Setting}. The regressors are correlated with $\Sigma_{ij} = \rho^{|i-j|}$ and $\rho = 0.5$. We consider two levels of noise, namely $\sigma = 1$ and $\sigma=0.1$.

We summarize the model selection performance of the penalized estimator in Figures
\ref{Fig:MCfirst01} and \ref{Fig:MCsecond01}.  In the left panels of the figures, we plot the frequencies of the dimensions of the selected model; in the right panels we plot the frequencies of selecting the correct regressors. From the left panels we see that the frequency of selecting a much larger model than the true model is very small in both designs. In the design with a larger noise, as the right panel of Figure \ref{Fig:MCfirst01} shows, the penalized quantile regression never selects the entire true model correctly, always missing the regressors with small coefficients. However, it almost always includes the three regressors with the largest coefficients.  (Notably, despite this partial failure of the model selection, post-penalized quantile regression still performs well, as we report below.)  On the other hand, we see from the right panel of Figure \ref{Fig:MCsecond01} that in the design with a lower noise level penalized quantile regression rarely misses any component of the true support. These results confirm the theoretical results of Theorem \ref{Thm:Selection}, namely, that when the non-zero coefficients are well separated from zero, the penalized estimator should select a model that includes the true model as a subset.  Moreover, these results also confirm the theoretical result of Theorem \ref{Thm:Sparsity}, namely, that the dimension of the selected model should be of the same stochastic order as the dimension of the true model. In summary, the model selection performance of the penalized estimator agrees very well with our theoretical results.

We summarize results on estimation performance in Table \ref{Table:MC}, which records for each estimator $\tilde \beta$  the norm of the bias $\|\Ep[\tilde \beta] - \beta_0\|$ and also the empirical risk $[\Ep[x_i' ( \tilde \beta - \beta_0 )]^2]^{1/2}$ for recovering the regression function.   Penalized quantile regression has a substantial bias, as we would expect from the definition of the estimator which penalizes large deviations of coefficients from zero. We see that the post-penalized quantile regression drastically improves upon the penalized quantile regression, particularly in terms of reducing the bias, which results in a much lower overall empirical risk.  Notably, despite that under the higher noise level the penalized estimator never recovers the true model correctly the post-penalized estimator still performs well.  This is because the penalized estimator always manages to select the most important regressors.   We also see that the empirical risk of the post-penalized estimator is within a factor of $\sqrt{\log p}$ of the empirical risk of the oracle estimator, as we would expect from our theoretical results.   Under the lower noise level, the post-penalized estimator performs almost identically to the ideal oracle estimator.  We would expect this since in this case the penalized estimator selects the model especially well, making the post-penalized estimator nearly the oracle.   In summary, we find the estimation performance of the penalized and post-penalized estimators to be in close agreement with our theoretical results.

\begin{figure}[!p]
\centering
\includegraphics[width=0.49\textwidth]{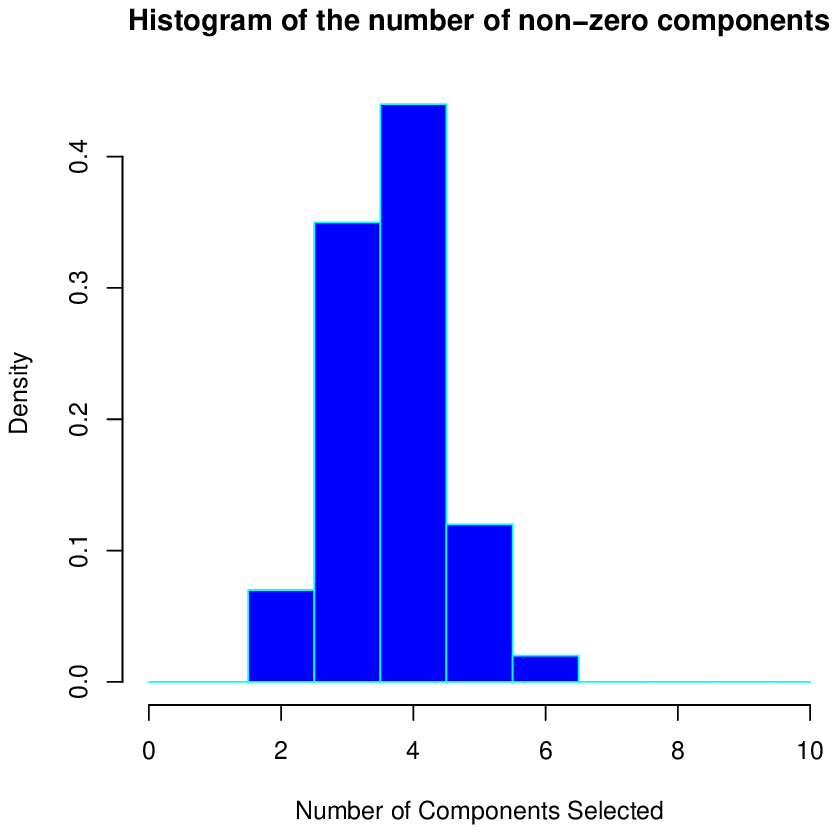}
\includegraphics[width=0.49\textwidth]{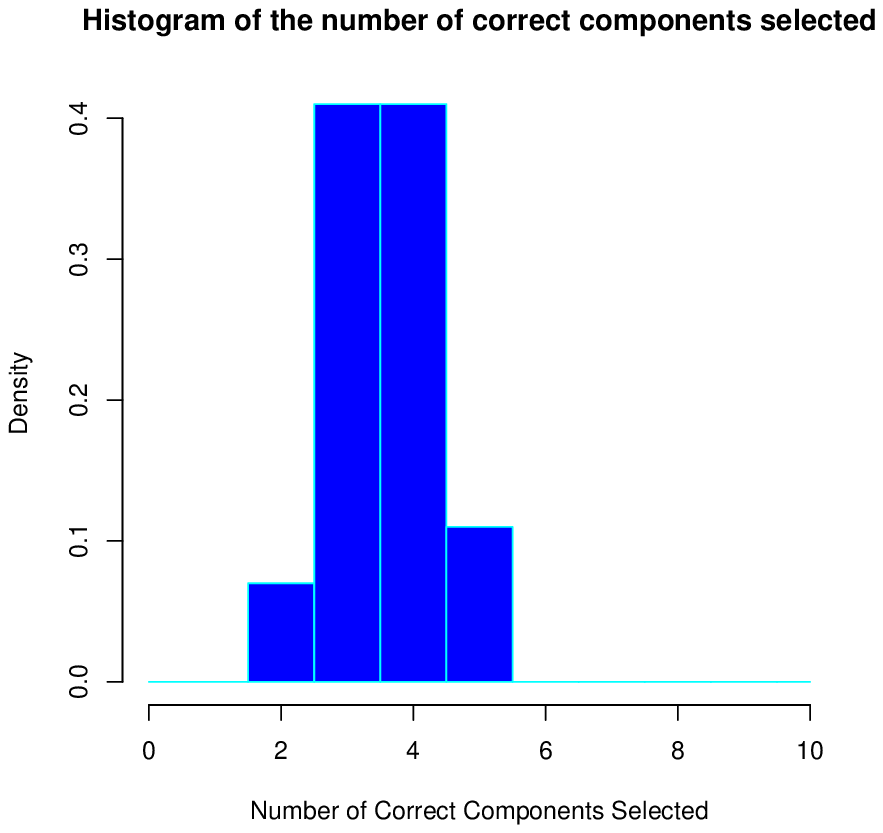}
\caption{ The figure summarizes the covariate selection results
for the design with $\sigma = 1$,  based on  $100$ Monte
Carlo repetitions. The left panel plots the histogram for the
number of covariates selected out of the possible 500 covariates.
The right panel plots the histogram for the number of significant
covariates selected; there are in  total $6$ significant
covariates amongst $500$ covariates. The sample size for each repetition was $n=100$.}\label{Fig:MCfirst01}
\end{figure}

\begin{figure}[!p]
\centering
\includegraphics[width=0.49\textwidth]{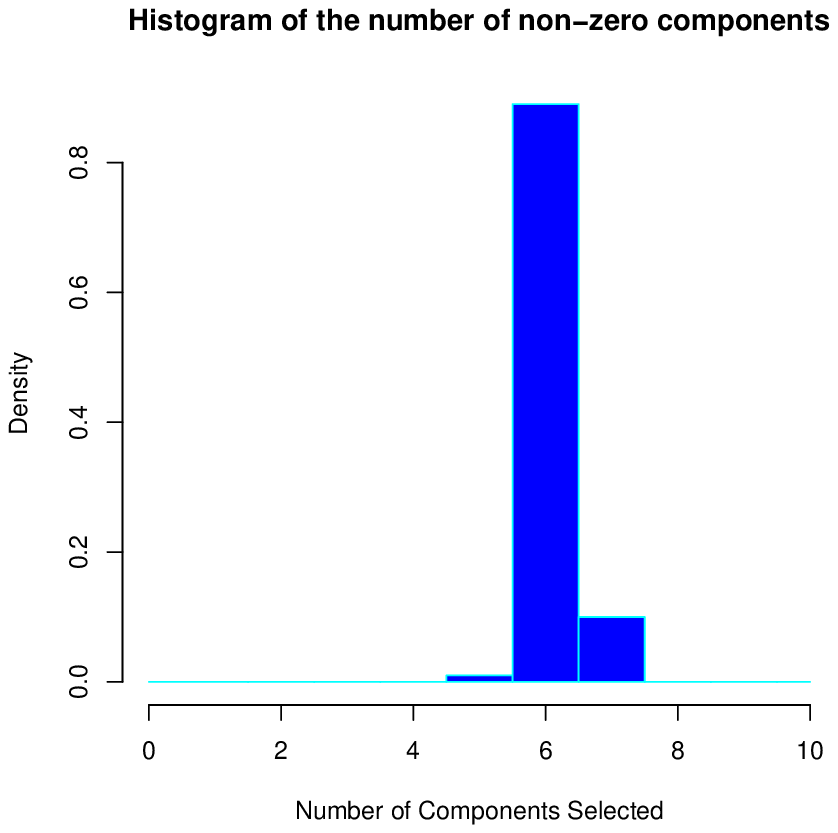}
\includegraphics[width=0.49\textwidth]{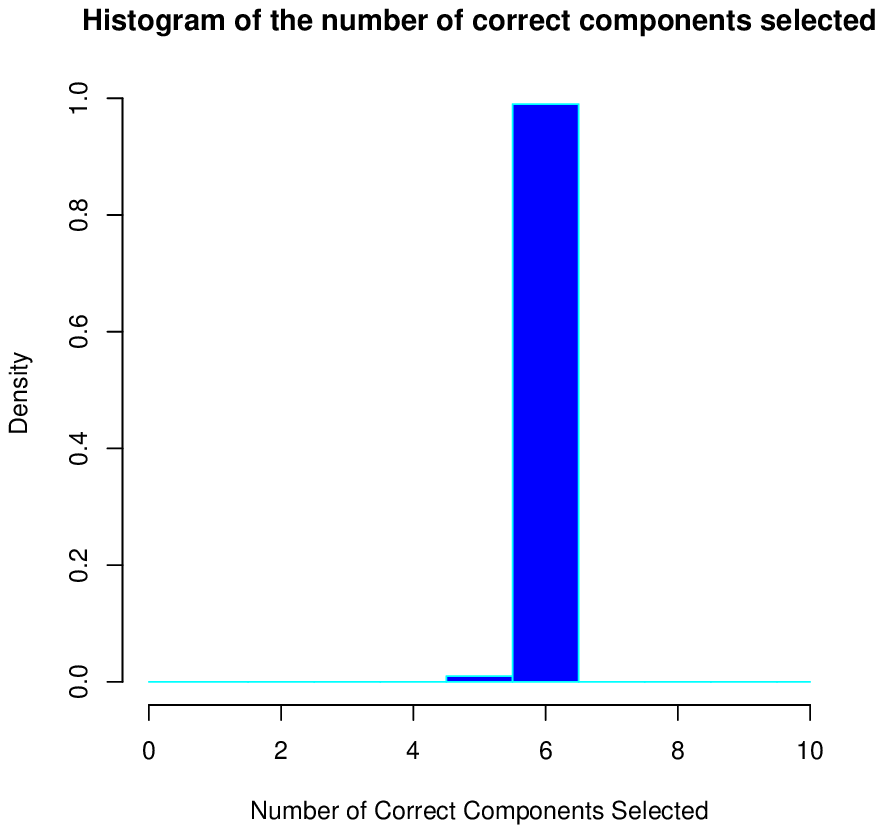}%
\caption{The figure summarizes the covariate selection results for
the design with $\sigma = 0.1$,  based on  $100$ Monte Carlo repetitions.  The left
panel plots the histogram for the number of covariates selected
out of the possible 500 covariates. The right panel plots the
histogram for the number of significant covariates selected; there
are in total $6$ significant covariates amongst $500$ covariates. The sample size for each repetition was $n=100$.}\label{Fig:MCsecond01}
\end{figure}

\begin{table}[!hp]\title{Monte Carlo Results}

\begin{center}

{\bf Design A ($\sigma = 1$)}
\end{center}
\begin{tabular}{lccccc}
\hline
 & Mean $\ell_0$-norm & Mean $\ell_1$-norm & Bias  & Empirical Risk \\
  \hline
Penalized QR & 3.67 & 1.28 & 0.92 & 1.22 \\
  Post-Penalized QR & 3.67 & 2.90 & 0.27 & 0.57 \\
  Oracle QR & 6.00 & 3.31 & 0.03 & 0.33 \\
\hline
\\
\end{tabular}

\begin{center}
{\bf Design B ($\sigma = 0.1$)}
\end{center}
\begin{tabular}{lccccc}
\hline
& Mean $\ell_0$-norm & Mean $\ell_1$-norm & Bias  & Empirical Risk \\
  \hline
Penalized QR & 6.09 & 2.98 & 0.13 & 0.19 \\
  Post-Penalized QR & 6.09 & 3.28 & 0.00 & 0.04 \\
  Oracle QR & 6.00 & 3.28 & 0.00 & 0.03 \\
\hline
\\
\end{tabular}
\caption{The table displays the average $\ell_0$ and $\ell_1$ norm
of the estimators as well as mean bias and empirical risk. We
obtained the results using $ 100 $ Monte Carlo repetitions for
each design.} \label{Table:MC}
\end{table}

\subsection{International Economic Growth Example}

In this section we apply $\ell_1$-penalized quantile regression to
an international economic growth example, using it primarily as a
method for model selection. We use the Barro and Lee data
consisting of a panel of 138 countries for the period of 1960 to
1985. We consider the national growth rates in gross domestic
product (GDP) per capita as a dependent variable $y$ for the periods
1965-75 and 1975-85.\footnote{The growth rate in GDP over a period from $t_1$ to $t_2$ is commonly defined as $\log
(GDP_{t_2}/GDP_{t_1})-1$.}  In our analysis, we will consider
a model with $p=60$ covariates, which allows for a total of
$n=90$ complete observations.  Our goal here is to select a subset
of these covariates and briefly compare the resulting models to
the standard models used in the empirical growth literature (Barro
and Sala-i-Martin \cite{BarroLee1994}, Koenker and Machado
\cite{KoenkerMachado1999}).

One of the central issues in the empirical growth literature is the
estimation of the effect of an initial (lagged) level of GDP per capita on the growth rates of GDP per capita. In particular, a key prediction from the classical Solow-Swan-Ramsey growth model is the hypothesis of convergence, which states that poorer countries should typically grow faster and therefore should tend to catch up with the richer countries. Thus, such a hypothesis states that the effect of the initial level of GDP on the growth rate should be negative. As pointed out in Barro and Sala-i-Martin \cite{BarroSala1995}, this hypothesis is rejected using a simple bivariate regression of growth rates on the initial level of GDP. (In our case, median regression yields a positive coefficient of $0.00045$.) In order to reconcile the data and the theory, the literature has focused on estimating the effect \textit{conditional} on the pertinent characteristics of countries.  Covariates that describe such characteristics can include variables measuring education and science policies, strength of market institutions, trade openness, savings rates and others \cite{BarroSala1995}.  The theory then predicts that  for countries with similar other characteristics the effect of the initial level of GDP on the growth rate should be negative (\cite{BarroSala1995}).  

Given that the number of covariates we can condition on is
comparable to the sample size, covariate selection becomes an
important issue in this analysis (\cite{OneMillion}, \cite{Two
million}).  In particular, previous findings came under
severe criticism for relying on ad hoc procedures for covariate
selection.  In fact, in some cases, all of the previous findings
have been questioned (\cite{OneMillion}).   Since the number of
covariates is high, there is no simple way to resolve the model
selection problem using only classical tools. Indeed the
number of possible lower-dimensional models is very large, although \cite{OneMillion} and \cite{Two million}  attempt to
search over several millions of these models.   Here we use the
Lasso selection device, specifically $\ell_1$-penalized median
regression, to resolve this important issue.

Let us now turn to our empirical results. We performed covariate
selection using $\ell_1$-penalized median regression, where
we initially used our data-driven choice of penalization parameter
$\lambda$.   This initial choice led us to select no covariates,
which is consistent with the situations in which the true
coefficients are not well-separated from zero. We then proceeded
to slowly decrease the penalization parameter in order to allow
for some covariates to be selected.   We present the model
selection results in Table \ref{Table:Growth}. With the first
relaxation of the choice of $\lambda$, we select the black market
exchange rate premium (characterizing trade openness) and a
measure of political instability. With a second relaxation of the
choice of $\lambda$ we select an additional set of educational
attainment variables, and several others reported in the table.
With a third relaxation of $\lambda$ we include yet another set of
variables also reported in the table. We refer the reader to
\cite{BarroLee1994} and \cite{BarroSala1995} for a complete
definition and discussion of each of these variables.

We then proceeded to apply ordinary median 
regression to the selected models and we also report the standard confidence intervals for these estimates.
Table \ref{Table:CondidenceInterval} shows these results.
We should note that the confidence intervals do not take into account that we have selected the models using the data.  (In
an ongoing companion work, we are working on devising procedures
that will account for this.)   We find that in all models with additional selected covariates, the median regression coefficients on the initial
level of GDP is always negative and the standard confidence intervals do not include zero. Similar conclusions also hold for quantile regressions with quantile indices in the middle range.
In summary, we believe that our empirical findings support the
hypothesis of convergence from the classical Solow-Swan-Ramsey
growth model.  Of course, it would be good to find formal
inferential methods to fully support this hypothesis.  Finally,
our findings also agree and thus support the previous findings
reported in  Barro and Sala-i-Martin \cite{BarroLee1994} and
Koenker and Machado \cite{KoenkerMachado1999}.

\begin{table}[!hp]\title{Confidence Intervals after Model Selection for the International Growth Regressions}
\begin{tabular}{ccccc}
\\
\hline
Penalization & & \multicolumn{2}{c}{Real GDP per capita (log)} \\
Parameter & & \\
$\lambda=1.077968$ & & Coefficient & $90\%$ Confidence Interval \\
\hline
$\lambda/2$ & & $-0.01691$      & $[-0.02552, -0.00444]$\\
$\lambda/3$ & & $-0.04121$      & $[-0.05485, -0.02976]$\\
$\lambda/4$ & & $-0.04466$      & $[-0.06510, -0.03410]$\\
$\lambda/5$ & & $-0.05148$      & $[-0.06521, -0.03296]$\\
\hline
\end{tabular}\caption{The table above displays the coefficient and a $90\%$ confidence interval associated with each model selected by the corresponding penalty parameter. The selected models are displayed in Table \ref{Table:Growth}.}\label{Table:CondidenceInterval}
\end{table}


{\small
\begin{table}[!hp]\title{Model Selection Results for the International Growth Regressions}
\renewcommand{\arraystretch}{1}

 \begin{tabular}{ccc}
\hline
Penalization & & \\
Parameter & & Real GDP per capita (log) is included in all models \\
$\lambda=1.077968$ &  & Additional Selected Variables \\
\hline
$\lambda$ &  & -  \\
\rowcolor[gray]{0.9} $\lambda/2$  &  &  Black Market Premium (log)  \\
\rowcolor[gray]{0.9}                 &  &  Political Instability \\
$\lambda/3$  &  &  Black Market Premium (log)  \\
                  &  &  Political Instability \\
                  &  &  Measure of tariff restriction  \\
                  &  &  Infant mortality rate \\
                  &  &  Ratio of real government ``consumption" net of defense and education\\
                  &  &  Exchange rate  \\
                  &  &  \% of ``higher school complete" in female population \\
                  &  &  \% of ``secondary school complete" in male population \\
\rowcolor[gray]{0.9}  $\lambda/4$ &  & Black Market Premium (log)  \\
\rowcolor[gray]{0.9}                  &  & Political Instability \\
\rowcolor[gray]{0.9}                  &  & Measure of tariff restriction  \\
\rowcolor[gray]{0.9}                  &  & Infant mortality rate \\
\rowcolor[gray]{0.9}                  &  & Ratio of real government ``consumption" net of defense and education\\
\rowcolor[gray]{0.9}                  &  & Exchange rate \\
\rowcolor[gray]{0.9}                  &  &  \% of ``higher school complete" in female population \\
\rowcolor[gray]{0.9}                  &  &  \% of ``secondary school complete" in male population \\ \rowcolor[gray]{0.9}                  &  & Female gross enrollment ratio for higher education  \\
\rowcolor[gray]{0.9}                  &  & \% of ``no education" in the male population \\
\rowcolor[gray]{0.9}                  &  & Population proportion over 65 \\
\rowcolor[gray]{0.9}                  &  & Average years of secondary schooling in the male population \\

$\lambda/5$                      &  & Black Market Premium (log)  \\ 
                  &  & Political Instability \\ 
                  &  & Measure of tariff restriction  \\ 
                  &  & Infant mortality rate \\ 
                  &  & Ratio of real government ``consumption" net of defense and education\\ 
                  &  & Exchange rate \\ 
                  &  &  \% of ``higher school complete" in female population \\  
                  &  &  \% of ``secondary school complete" in male population \\ 
                  &  & Female gross enrollment ratio for higher education  \\ 
                  &  & \% of ``no education" in the male population \\ 
                  &  & Population proportion over 65 \\ 
                  &  & Average years of secondary schooling in the male population \\ 
                  &  & Growth rate of population \\ 
                  &  &  \% of ``higher school attained" in male population \\ 
                  &  & Ratio of nominal government expenditure on defense to nominal GDP\\ 
                  &  & Ratio of import to GDP\\ 
\hline
\end{tabular}\caption{For this particular decreasing sequence of penalization parameters we obtained nested models.
}\label{Table:Growth}
\end{table}}



\appendix

\section{Proof of Theorem 1}

\begin{proof}[Proof of Theorem \ref{Thm:BoundLAMBDA}]
We note $ \Lambda \leq \W \max_{1\leq j\leq
p} \sup_{u\in\mathcal{U}} n\En\[ (u- 1\{u_i \leq u\}) x_{ij}/\hat\sigma_j\]$. For any $u \in \mathcal{U}$, $j \in \{1,\ldots,p\}$
we have by Lemma 1.5 in \cite{LedouxTalagrandBook} that
$P(| \Gn[(u- 1\{u_i \leq u\}) x_{ij}/\hat \sigma_j] | \geq \widetilde K) \leq 2 \exp(-\widetilde K^2/2).$
Hence by the symmetrization lemma for probabilities, Lemma 2.3.7 in
\cite{vdV-W},  with $\widetilde K \geq 2\sqrt{\log 2}$ we have
\begin{equation}\label{bound lambda 1}
\begin{array}{rcl} P( \Lambda > \widetilde K \sqrt{n} |X ) & \leq & 4 P\( \sup_{u\in\mathcal{U}}\max_{1\leq j\leq p} |\Gn^o[ (u- 1\{u_i \leq u\}) x_{ij}/\hat\sigma_j] | > \widetilde K/ (4\W) |X \) \\
& \leq & 4 p \max_{1\leq j \leq p} P\( \sup_{u\in\mathcal{U}}  |\Gn^o[ (u- 1\{u_i \leq u\}) x_{ij}/\hat\sigma_j] | > \widetilde K / (4\W) |X \),
\end{array}
 \end{equation}
where $\Gn^o$ denotes the symmetrized empirical process (see \cite{vdV-W}) generated by the Rademacher variables
$\varepsilon_i, i=1,...,n$, which are independent of $U= (u_1,...,u_n)$ and $X=(x_1,..., x_n)$.
Let us condition on  $U$ and $X$, and  define $\mathcal{F}_j
=\{\varepsilon_i x_{ij} (u- 1\{u_i \leq u\})/\hat\sigma_j : u \in \mathcal{U}\}$ for $j=1,\ldots,p$.
The VC dimension of $\mathcal{F}_j$ is at most 6. Therefore, by Theorem
2.6.7 of \cite{vdV-W} for some universal constant $C_1'\geq 1$ the function class
$\mathcal{F}_j$ with envelope function $F_j$ obeys
$$ N(\varepsilon\|F_j\|_{\mathbb{P}_n,2}, \mathcal{F}_j, L_2(\mathbb{P}_n)) \leq n(\varepsilon, \mathcal{F}_j)= C_1'\cdot  6 \cdot (16e)^6 (1/\varepsilon)^{10},$$
where $N(\varepsilon, \mathcal{F}, L_2(\mathbb{P}_n))$ denotes the minimal number of balls of radius $\varepsilon$
with respect to the $L_2(\mathbb{P}_n)$ norm $\|\cdot\|_{\mathbb{P}_n,2}$ needed to cover the class of functions $\mathcal{F}$; see \cite{vdV-W}.

Conditional on the data $U= (u_1,\ldots,u_n)$ and $X=(x_1,\ldots, x_n)$, the symmetrized empirical process $\{\Gn^o(f), f \in \mathcal{F}_j\}$ is sub-Gaussian
with respect to the $L_2(\Pn)$ norm by the Hoeffding inequality; see, e.g., \cite{vdV-W}.  Since $\|F_j\|_{\Pn,2}\leq 1$ and $\rho(\mathcal{F}_j,\Pn) = \sup_{f \in \mathcal{F}_j} \|f\|_{\Pn,2}/ \|F\|_{\Pn,2} \leq 1$, we have $$\|F_j\|_{\Pn,2}\int_{0}^{\rho(\mathcal{F}_j,\Pn)/4}\sqrt{\log n(\varepsilon,\mathcal{F}_j)}d\varepsilon \leq \bar e := (1/4) \sqrt{ \log( 6 C_1' (16e)^6)} + (1/4)\sqrt{10\log 4}.$$ By Lemma \ref{expo1} with $D=1$, there is a universal constant $c$ such that for any $K \geq 1 $: \begin{eqnarray}
P\(  \sup_{f \in \mathcal{F}_j} |\Gn^o(f)| > K c \bar e |X,U\) & \leq & \int_0^{1/2} \varepsilon^{-1} n(\varepsilon, \mathcal{F}_j)^{-(K^2-1)}d\varepsilon \nonumber \\
& \leq & (1/2)[6C_1'(16e)^6]^{-(K^2-1)} \frac{(1/2)^{10(K^2-1)}}{ 10(K^2-1) }\label{bound lambda 2}.
\end{eqnarray}
By (\ref{bound lambda 1}) and (\ref{bound lambda 2}) for any $k \geq 1$ we have
{\small$$\begin{array}{rl}
\displaystyle  P \( \Lambda \geq k \cdot (4\sqrt{2}c\bar{e}) \W \sqrt{n \log p} | X \) & \leq \displaystyle   4p \max_{1\leq j \leq p} \Ep_U P\(  \sup_{f \in \mathcal{F}_j} |\Gn^o(f)| > k\sqrt{2\log p} \ c\bar e |X,U\)\\
 & \leq \displaystyle p^{-6k^2+1}  \leq  p^{-k^2+1} \\
 \end{array}$$}\noindent since $( 2k^2\log p - 1 ) \geq (\log 2 - 0.5) k^2\log p$  for $p\geq 2$.  Thus, result (i) holds with $C_\Lambda := 4\sqrt{2}c\bar e$. Result (ii) follows immediately by choosing $k = \sqrt{ 1 +
\log(1/\alpha) / \log p }$ to make the right side of the display above equal to $\alpha$.
\end{proof}

\section{Proofs of Lemmas 3-5 (used in Theorem 2)}

\begin{proof}[Proof of Lemma \ref{Lemma:E1}] (Restricted Set)
Part 1. By condition D.3, with probability $1-\gamma$, for every $j=1,\ldots,p$ we
have $1/2 \leq \hat\sigma_j \leq 3/2$, which implies (\ref{E1normEquiv}).

Part 2. Denote the true rankscores by $a_i^*(u) = u-1\{y_i\leq
x_i'\beta(u)\}$ for $i=1,\ldots,n$. Next recall that $\hat
Q_u(\cdot)$ is a convex function and $\En\[ x_{i}a_i^*(u)\] \in
\partial \hat Q_u(\beta(u))$. Therefore, we have
$$ \hat Q_u(\hat \beta(u)) \geq \hat Q_u(\beta(u)) + \En\[ x_{i}a_i^*(u)\]'(\hat \beta(u) - \beta(u)).$$
Let $\widehat D = {\rm diag}[\hat\sigma_1,\ldots,\hat \sigma_p]$
and note that $\lambda\sqrt{u(1-u)}(c_0-3)/(c_0+3) \geq
n\|\widehat D^{-1}\En\[ x_{i}a_i^*(u)\]\|_{\infty}$ with
probability at least $1-\alpha$. By optimality of $\hat \beta(u)$
for the $\ell_1$-penalized problem, we have
$$
\begin{array}{rcl}
\displaystyle 0 & \leq & \hat Q_u(\beta(u)) - \hat Q_u(\hat\beta(u)) + \frac{\lambda\sqrt{u(1-u)}}{n}\|\beta(u)\|_{1,n} - \frac{\lambda\sqrt{u(1-u)}}{n}\|\hat \beta(u)\|_{1,n}\\
& \leq & \left| \En\[ x_{i}a_i^*(u)\]'(\hat \beta(u) - \beta(u))\right| + \frac{\lambda\sqrt{u(1-u)}}{n}\(\|\beta(u)\|_{1,n} - \|\hat \beta(u)\|_{1,n} \)\\
& = & \left\|\widehat D^{-1} \En\[ x_{i}a_i^*(u)\]\right\|_{\infty}\left\|\widehat D (\hat \beta(u) - \beta(u))\right\|_1 + \frac{\lambda\sqrt{u(1-u)}}{n}\(\|\beta(u)\|_{1,n} - \|\hat \beta(u)\|_{1,n}\) \\
& \leq  &  \frac{\lambda\sqrt{u(1-u)}}{n}\sum_{j=1}^p \( \frac{c_0-3}{c_0+3}\hat\sigma_j\left|\hat \beta_j(u) - \beta_j(u)\right| + \hat\sigma_j|\beta_j(u)| - \hat\sigma_j|\hat \beta_j(u)|\), \\
\end{array}
$$ with probability at least $1-\alpha$. After canceling $\lambda\sqrt{u(1-u)}/n$ we obtain
 \begin{equation}\label{to show RC 1}
 \(1-\frac{c_0-3}{c_0+3}\) \|\hat \beta(u) - \beta(u)\|_{1,n} \leq \sum_{j=1}^p \hat\sigma_j\(\left|\hat \beta_j(u) - \beta_j(u)\right| + |\beta_j(u)| - |\hat \beta_j(u)| \). \end{equation}
Furthermore, since $\left|\hat \beta_j(u) - \beta_j(u)\right| + |\beta_j(u)| -
|\hat \beta_j(u)| = 0$ if $\beta_j(u) = 0$, i.e. $j \in T_u^c$,
 \begin{equation}\label{to show RC 2}
\sum_{j=1}^p \hat\sigma_j\(\left|\hat \beta_j(u) - \beta_j(u)\right| + |\beta_j(u)| - |\hat \beta_j(u)| \) \leq  2\|\widehat \beta_{T_u}(u) = \beta(u)\|_{1,n}.
 \end{equation}
(\ref{to show RC 1}) and (\ref{to show RC 2}) establish that $ \| \hat \beta_{T_u^c}(u) \|_{1,n} \leq (c_0/3)
\| \hat \beta_{T_u}(u) - \beta(u) \|_{1,n}$ with probability at least $1-\alpha$.
In turn, by Part 1 of this Lemma, $\| \hat \beta_{T_u^c}(u) \|_{1,n}
\geq (1/2) \| \hat \beta_{T_u^c}(u) \|_{1}$ and $\| \hat
\beta_{T_u}(u) - \beta(u) \|_{1,n} \leq (3/2) \| \hat \beta_{T_u}(u) -
\beta(u) \|_{1}$, which holds with probability at least $1-\gamma$.  Intersection
of these two event holds with probability at least
$1-\alpha-\gamma$. Finally, by Lemma \ref{Lemma:E4partI}, $\|\hat \beta(u)\|_0 \leq n$ with probability 1 uniformly in $u\in\mathcal{U}$.
\end{proof}

\begin{proof}[Proof of Lemma \ref{Lemma:E2}] (Identification in Population)  Part 1. Proof of claims (\ref{E.2-RE.addon})-(\ref{E.2-RE.1}).
By ${\rm RE}(c_0,m)$ and by $\delta \in A_u$
$$  \|J^{1/2}_u\delta\| \geq  \|(\Ep[x_ix_i'])^{1/2}\delta\| \underf^{1/2}  \geq \|\delta_{T_u}\|   \underf^{1/2}\kappa_0 \geq \frac{ \underf^{1/2}\kappa_0}{\sqrt{s}} \|\delta_{T_u}\|_1 \geq \frac{ \underf^{1/2}\kappa_0}{\sqrt{s}(1+c_0)} \|\delta\|_1.$$

Part 2. Proof of claim (\ref{E.2-RE.2}). Proceeding similarly to \cite{BickelRitovTsybakov2009}, we note that the $k$th largest
in absolute value component of $\delta_{T_u^c}$ is less than
$\|\delta_{T_u^c}\|_1/k$. Therefore by $\delta \in A_u$ and $|T_u|\leq s$
$$ \|\delta_{(T_u\cup\overline{T}_u(\delta,m))^c}\|^2 \leq \sum_{k\geq m+1} \frac{ \|\delta_{T_u^c}\|_1^2}{k^2} \leq \frac{\|\delta_{T_u^c}\|_1^2}{m} \leq c_0^2\frac{\|\delta_{T_u}\|_1^2}{m} \leq c_0^2\|\delta_{T_u}\|^2\frac{s}{m} \leq c_0^2\|\delta_{T_u\cup\overline{T}_u(\delta,m)}\|^2\frac{s}{m},$$
so that  $\|\delta\| \leq \(1+c_0\sqrt{s/m}\) \|\delta_{T_u\cup\overline{T}_u(\delta,m)}\|$; and the last term is bounded by ${\rm RE}(c_0,m)$,
$$ \(1+c_0\sqrt{s/m}\) \|(\Ep[x_ix_i'])^{1/2}\delta\|/\kappa_m \leq \(1+c_0\sqrt{s/m}\) \|J_u^{1/2}\delta\|/[\underf^{1/2}\kappa_m] .$$

Part 3. The proof of claim  (\ref{E.2-RNI}) proceeds in two steps.   Step 1. (Minoration).  Define the maximal radius over which the criterion function can be minorated by a quadratic function
$$ r_{A_u} = \sup_{r} \left\{ r \ : Q_u(\beta(u)+\tilde \delta) - Q_u(\beta(u)) \geq \frac{1}{4} \|J_u^{1/2} \tilde \delta\|^{2}, \ \mbox{for all} \ \tilde \delta\in A_u, \|J_u^{1/2}\tilde \delta\| \leq r \right\}.$$
Step 2 below shows that  $r_{A_u} \geq 4q$.   By construction of $r_{A_u}$ and
the convexity of $Q_u$,
 $$ \begin{array}{lll}
 && Q_u(\beta(u) + \delta) - Q_u(\beta(u))   \\
&&     \geq \frac{\|J_u^{1/2}\delta\|^2}{4} \wedge \left\{ \frac{\|J_u^{1/2}\delta\|}{r_{A_u}} \cdot \inf_{\tilde \delta \in A_u, \| J_u^{1/2}\tilde \delta\| \geq r_{A_u}} \!\! Q_u(\beta(u)+\tilde \delta) - Q_u(\beta(u)) \right\}\\
&&  \geq   \frac{\|J_u^{1/2}\delta\|^2}{4} \wedge \left\{ \frac{\|J_u^{1/2}\delta\|}{r_{A_u}} \frac{r_{A_u}^2}{4}\right\}
 \geq    \frac{\|J_u^{1/2}\delta\|^2}{4} \wedge \left\{ q\|J_u^{1/2}\delta\|\right\}, \text{  for any $\delta \in A_u$. }
\end{array}$$

Step 2. ($r_{A_u} \geq 4q$) Let $F_{y|x}$ denote the conditional distribution of
$y$ given $x$. From \cite{Knight1998},
for any two scalars $w$ and $v$  we have that
\begin{equation}\label{Eq:TrickRho}
\rho_u(w-v) - \rho_u(w) = -v (u - 1\{w\leq 0\}) + \int_0^v(
1\{w\leq z\} - 1\{w\leq 0\})dz.
\end{equation}
Using (\ref{Eq:TrickRho}) with $w=y - x'\beta(u)$ and $v =
x'\delta$ we conclude $\Ep\[-v(u-1\{w\leq
0\})\] = 0$. Using the law of iterated expectations and mean value
expansion, we obtain for $\tilde z_{x,z} \in [0,z]$
\begin{equation}\label{Eq:Knight} \begin{array}{rcl}
&& Q_u(\beta(u) + \delta) - Q_u(\beta(u)) =  \Ep\[ \int_0^{x'\delta} F_{y|x}(x'\beta(u) + z) - F_{y|x}(x'\beta(u)) dz \] \\
&& =  \Ep\[ \int_0^{x'\delta} zf_{y|x}(x'\beta(u)) + \frac{z^2}{2}f'_{y|x}(x'\beta(u)+\tilde z_{x,z}) dz \] \\
&& \geq  \frac{1}{2} \|J_u^{1/2}\delta\|^2  - \frac{1}{6}\bar f ' \Ep[|x'\delta|^3] \geq \frac{1}{4} \|J_u^{1/2}\delta\|^2  + \frac{1}{4} \underf \Ep[|x'\delta|^2] - \frac{1}{6} \bar f' \Ep[|x'\delta|^3].\\
\end{array}
\end{equation}
Note that for $\delta \in A_u$, if $ \|J_u^{1/2}\delta\| \leq  4q \leq
 (3/2) \cdot (\underf^{3/2}/\bar{f'}) \cdot \inf_{\delta\in A_u, \delta \neq 0} \Ep\[|x'\delta|^2\]^{3/2}/\Ep\[|x'\delta|^3\]$,
it follows that  $(1/6)\bar f'\Ep[|x'\delta|^3] \leq
(1/4) \underf \Ep[|x'\delta|^2]$. This and (\ref{Eq:Knight}) imply $r_{A_u} \geq 4q$.
\end{proof}

\begin{proof}[Proof of Lemma \ref{Lemma:E3}] (Control of Empirical Error) We divide the proof in four steps.

Step 1. (Main Argument) Let
 $$\mathcal{A}(t) :=\emperror(t)\sqrt{n} = \sup_{u \in \mathcal{U}, \|J^{1/2}_u\delta\|\leq t, \delta \in A_u}  |\mathbb{G}_n [ \rho_u(y_i - x_i'(\beta(u) + \delta)) - \rho_u(y_i - x_i'\beta(u)) ]| $$
Let $\Omega_1$ be the event in which
$ \max_{1\leq j \leq p} \left|
\hat \sigma_j - 1 \right| \leq 1/2$, where $P(\Omega_1) \geq 1- \gamma$.

In order to apply the symmetrization lemma,  Lemma 2.3.7 in \cite{vdV-W}, to bound the tail probability of $\mathcal{A}(t)$
first note that for any fixed $\delta \in
A_u$, $u \in \mathcal{U}$ we have
 \begin{eqnarray*} & & \textrm{var}\(\bG_n\[\rho_u(y_i - x_i'(\beta(u)+\delta)) - \rho_u(y_i - x_i'\beta(u))\]\)
 \leq \Ep\[(x_i'\delta)^2\] \leq  t^2/\underf
 \end{eqnarray*}
 Then application of the symmetrization lemma for probabilities, Lemma 2.3.7 in \cite{vdV-W}, yields
 \begin{equation}\label{EQ: bound PA}  P( \mathcal{A}(t) \geq M ) \leq \frac{2P( \mathcal{A}^o(t) \geq M/4 )}{1- t^2/(\underf M^2)} \leq \frac{2P( \mathcal{A}^o(t) \geq M/4 | \Omega_1 )+ 2P(\Omega_1^c)}{1- t^2/(\underf M^2)}, \end{equation}
where $\mathcal{A}^o(t)$ is the symmetrized version of
$\mathcal{A}(t)$, constructed by replacing the empirical process $\mathbb{G}_n$ with its symmetrized version $\mathbb{G}^o_n$, and $P(\Omega_1^c)\leq \gamma$. We set $M> M_1:= t (3/\underf)^{1/2}$, which makes the denominator on right side of (\ref{EQ: bound PA})
greater than $2/3$. Further, Step 3 below shows that
$P( \mathcal{A}^o(t) \geq M/4 | \Omega_1) \leq p^{-A^2}$ for
$$M/4 \geq M_2:= t \cdot A \cdot 18\sqrt{2} \cdot \Gamma \cdot \sqrt{2\log p + \log ( 2 + 4\sqrt{2}L\underf^{1/2}\kappa_0/t)
},  \ \ \Gamma = \sqrt{s}(1+c_0)/[\underf^{1/2}\kappa_0]. $$
We conclude that with probability at least $1-3\gamma-3p^{-A^2}$,
$\mathcal{A}(t) \leq  M_1 \vee (4M_2).$

\comment{Under conditions D.2 and D.4 $\log (L\underf^{1/2}\kappa_0)  \lesssim \log
(p\vee n)$.}
Therefore, there is a universal constant
$C_E$ such that  with probability at least $1-3\gamma-3p^{-A^2}$, $$\mathcal{A}(t) \leq   t \cdot C_E\cdot
\frac{(1+c_0)A }{\underf^{1/2}\kappa_0}\sqrt{s \log (p \vee [L\underf^{1/2}\kappa_0/t])}$$ and the result follows.

Step 2. (Bound on $P( \mathcal{A}^o(t) \geq K|\Omega_1)$).     We begin by noting that
Lemma \ref{Lemma:E1} and \ref{Lemma:E2} imply that $ \|\delta\|_{1,n}  \leq \frac{3}{2} \sqrt{s}(1+c_0) \| J^{1/2}_u \delta\|/[\underf^{1/2} \kappa_0]$ so that for all $u \in \mathcal{U}$
 \begin{equation}\label{Eq:SetInclusion} \{
\delta \in A_u \ : \ \| J^{1/2}_u \delta\| \leq
t  \} \subseteq \{\delta \in \RR^p : \|\delta\|_{1,n} \leq
2t\Gamma\}, \ \  \Gamma:= \sqrt{s}(1+c_0)/[\underf^{1/2}\kappa_0]. \end{equation}
Further, we let $\mathcal{U}_k = \{\hat u_1,\ldots,\hat u_k\}$ be an
$\varepsilon$-net of quantile indices in $\mathcal{U}$ with
\begin{equation}\label{eq: choice of varepsilon} \varepsilon \leq t \Gamma/(2\sqrt{2 s} L )   \text{ and }  \ k \leq  1/\varepsilon.
\end{equation}
By $\rho_u(y_i - x_i'(\beta(u) + \delta)) - \rho_u(y_i - x_i'\beta(u))= u x_i'\delta + w_i(x_i'\delta, u)$, for
$w_i(b,u):=(y_i-x_i'\beta(u)-b)_- - (y_i-x_i'\beta(u))_-$, and by (\ref{Eq:SetInclusion})
we have that  $\mA^o(t)  \ \leq \ \mB^o(t) + \mC^o(t)$, where
\begin{eqnarray*}
& & \mB^o(t) :=   \sup_{u \in \mathcal{U}, \|\delta\|_{1,n} \leq 2t\Gamma} |\mathbb{G}_n^o [x_i'\delta]  | \text{ and }
  \mC^o(t) := \sup_{u \in \mathcal{U}, \|\delta\|_{1,n} \leq 2t\Gamma} |\mathbb{G}_n^o [ w_i(\delta,u)]  |.  \end{eqnarray*}
Then we compute the bounds
\begin{eqnarray*}
P[ \mB^o(t) > K | \Omega_1 ] & \leq&   \min_{\lambda \geq 0} e^{-\lambda K} \Ep[e^{\lambda \mB^o(t)}| \Omega_1]  \text {  by Markov  } \\
 & \leq &  \min_{\lambda \geq 0} e^{-\lambda K} 2 p  \exp(  (2\lambda t\Gamma )^2/2  )   \text{  by Step 3  }  \\
 & \leq &  2p  \exp(-K^2/ ( 2\sqrt{2} t\Gamma )^2 ) \text{ by setting $\lambda =  K/ (2 t\Gamma )^2 $ }, \\
P[ \mC^o(t) > K | \Omega_1  ] & \leq&    \min_{\lambda \geq 0} e^{-\lambda K}  \Ep[e^{\lambda \mC^o(t)} | \Omega_1,X ]  \text {  by Markov   } \\
 & \leq &    \min_{\lambda \geq 0} \exp(-\lambda K) 2 (p/\varepsilon)\exp(  (16\lambda t\Gamma )^2/2  )   \text{  by Step 4  }  \\
 & \leq &  \varepsilon^{-1} 2p \exp(-K^2/ (16\sqrt{2} t\Gamma)^2 ) \text{ by setting $\lambda =  K/ (16 t\Gamma)^2 $ },
 \end{eqnarray*}
so that
\begin{eqnarray*}
P[\mA^o(t) >   2\sqrt{2} K + 16 \sqrt{2} K | \Omega_1] &\leq &   P[\mB^o(t) > 2\sqrt{2} K | \Omega_1] + P[\mC^o(t) > 16\sqrt{2} K | \Omega_1 ]  \\
& \leq &  2p (1 + \varepsilon^{-1})
\exp(-K^2/(t\Gamma)^2) .
 \end{eqnarray*}
Setting     $K =  A \cdot t\cdot \Gamma \cdot \sqrt{ \log
\{2p^2(1 + \varepsilon^{-1})\}}$, for $A \geq 1$, we get
$
P[ \mA^o(t) \geq  18\sqrt{2} K | \Omega_1 ] \leq   p^{-A^2}.
$

Step 3. (Bound on $\Ep[e^{\lambda \mB^o(t)}|\Omega_1]$) We bound
\begin{eqnarray*}
 \Ep[e^{\lambda \mB^o(t)}|\Omega_1]
 & \leq &   \Ep[\exp (  2\lambda t\Gamma \max_{j \leq p} | \bG_n^o(x_{ij})/\hat \sigma_j | )  |\Omega_1]    \\
 & \leq &    2 p \max_{j \leq p}\Ep[\exp \(  2\lambda t\Gamma  \bG_n^o(x_{ij})/\hat \sigma_j \) |\Omega_1]  \leq    2 p  \exp(  (2\lambda t\Gamma
)^2/2  ),
\end{eqnarray*}
where the first inequality follows from $
 \left|\mathbb{G}_n^o [x_i'\delta]  \right|\leq 2
\|\delta\|_{1,n} \max_{1\leq j \leq p} |
\bG_n^o(x_{ij})/\hat \sigma_j|$ holding under event $\Omega_1$, the penultimate inequality follows from the simple bound
$$\Ep[\max_{j\leq p}e^{|z_j|}] \leq p \max_{j\leq p}\Ep[ e^{|z_j|}]
\leq   p \max_{j\leq p} \Ep[e^{z_j} + e^{-z_j}] \leq  2 p
\max_{j\leq p} \Ep[ e^{z_j}]$$ holding for symmetric random variables
$z_j$, and the last inequality follows from the law of iterated expectations and from
$
\Ep[\exp \(  2\lambda t\Gamma  \bG_n^o(x_{ij})/\hat \sigma_j \) |\Omega_1, X]  \leq \exp((2\lambda t\Gamma
)^2/2) $
holding by the Hoeffding inequality (more precisely, by the intermediate step in
the proof of the Hoeffding inequality, see, e.g., p. 100 in \cite{vdV-W}).  Here $\Ep[\cdot|\Omega_1, X]$
denotes the expectation over the symmetrizing Rademacher variables entering the definition
of the symmetrized process $\mathbb{G}_n^o$ .

Step 4. (Bound on $\Ep[e^{\lambda \mC^o(t)}|\Omega_1]$)   We bound
\begin{eqnarray*}
\mC^o(t) & \leq & \sup_{u \in \mathcal{U}, |u-\hat u|\leq \varepsilon, \hat u \in \mathcal{U}_k} \sup_{\|\delta \|_{1,n} \leq 2t\Gamma} \left|\mathbb{G}_n^o [w_i(x_i'(\delta+\beta(u)-\beta(\hat u)),\hat u)]  \right|  \\
&  + & \sup_{u \in \mathcal{U}, |u - \hat u|\leq \varepsilon, \hat u \in \mathcal{U}_k}
 \left|\mathbb{G}_n [w_i ( x_i'(\beta(u) - \beta(\hat u)), \hat u)]  \right|\\
& \leq &   2  \sup_{\hat u \in \mathcal{U}_k, \|\delta\|_{1,n} \leq 4t\Gamma} |\mathbb{G}_n^o [w_i(x_i'\delta,\hat u)] |=:\mathcal{D}^o(t),
\end{eqnarray*}
where the first inequality is elementary, and  the second inequality follows from the inequality
$$\sup_{|u-\hat u|\leq \varepsilon} \|\beta(u)-\beta(\hat u)\|_{1,n}  \leq  \sqrt{2s}L(2\max_{1\leq j\leq p} \sigma_j)\varepsilon \leq \sqrt{2s}L(2\cdot 3/2)\varepsilon \leq 2t\Gamma,$$
holding by our choice (\ref{eq: choice of varepsilon}) of $\varepsilon$ and by event  $\Omega_1$.

Next we bound $\Ep[ e^{\mathcal{D}^o(t)} | \Omega_1]$
\begin{eqnarray*}
\Ep [ e^{\lambda \mathcal{D}^o(t)} |\Omega_1]  &\leq &    (1/\varepsilon) \max_{\hat u \in \mathcal{U}_k}\Ep[\exp(2\lambda  \sup_{\|\delta\|_{1,n} \leq 4t\Gamma} |\mathbb{G}_n^o [w_i(x_i'\delta,\hat u)]|) |\Omega_1]  \\
&\leq &     (1/\varepsilon)  \max_{\hat u \in \mathcal{U}_k} \Ep [\exp(4\lambda  \sup_{\|\delta\|_{1,n} \leq 4t\Gamma} |\mathbb{G}_n^o [x_i'\delta]  |) |\Omega_1]  \\
& \leq &   2 (p/\varepsilon) \max_{j \leq p}\Ep[\exp \(  16\lambda t\Gamma   \bG_n^o(x_{ij})/\hat \sigma_j \) |\Omega_1]
 \leq  2 (p/\varepsilon)\exp(  (16\lambda
t\Gamma )^2/2  ),\end{eqnarray*}
where the first inequality follows from the definition of $w_i$ and by $k \leq 1/\varepsilon$, the second inequality follows from the exponential
moment inequality for contractions (Theorem 4.12 of Ledoux and Talagrand \cite{LedouxTalagrandBook})
and from the contractive property $|w_i(a,\hat u)-w_i(b,\hat u)| \leq |a-b|$, and the last two inequalities follow exactly as in Step 3. \end{proof}

\section{Proof of Lemmas 6-7 (used in Theorem 3)}\label{App:LP}

In order to characterize the sparsity properties of $\hat
\beta(u)$, we will exploit the fact that (\ref{Def:L1QR}) can be
written as the following linear programming problem:
\begin{equation}\label{DefQRLP}
\begin{array}{rlcl}
\displaystyle \min_{\xi^+, \xi^-,\beta^+, \beta^- \in \mathbb{R}^{2n + 2p}_+} &  \displaystyle \En\[ u \xi^+_i + (1-u) \xi^-_i\] + \frac{\lambda\sqrt{u(1-u)}}{n} \sum_{j=1}^p \hat\sigma_j(\beta^+_j + \beta^-_j) \\
& \xi^+_i - \xi^-_i = y_i - x_i'(\beta^+ -\beta^-),  \ \ i = 1,\ldots,n.\\
\end{array}
\end{equation}

Our theoretical analysis of the sparsity of $\hat \beta(u)$ relies
on the dual of (\ref{DefQRLP}):
\begin{equation}\label{Def:DualL1}
\begin{array}{rlcl}
\displaystyle \max_{a \in \mathbb{R}^n}  &\  \En\[ y_ia_i \] \ \\
& \ | \En\[ x_{ij} a_i \] | \leq  \lambda\sqrt{u(1-u)}\hat\sigma_j/n, \ \ j =1,\ldots,p, \\
&  \ (u-1) \leq a_i \leq u,  \ \ i=1,\ldots,n.
\end{array}
\end{equation}  The dual program maximizes the correlation between the response variable and the rank scores subject to the condition
requiring the rank scores to be approximately uncorrelated with the regressors.
The optimal solution $\hat a(u)$ to (\ref{Def:DualL1})  plays a key role in determining the sparsity of $\hat\beta(u)$.

\begin{lemma}[Signs and Interpolation Property]\label{Lemma:E4partI} (1)  For any
$j \in \{1,\ldots,p\}$
\begin{equation}\label{Def:Sign}
\begin{array}{lcl}
\displaystyle \hat \beta_j(u) > 0 \ \ & \mbox{iff} &\displaystyle \ \ \En\[ x_{ij}\hat a_i(u)\] = \ \lambda\sqrt{u(1-u)}\hat\sigma_j/n, \\
\displaystyle\hat \beta_j(u) < 0 \ \ & \mbox{iff} &\displaystyle \ \ \En\[ x_{ij}\hat a_i(u)\] = -\lambda\sqrt{u(1-u)}\hat\sigma_j/n,

\end{array}
\end{equation} (2) $\|\hat\beta(u)\|_0 \leq n \wedge p$ uniformly over $u \in \mathcal{U}$. (3) If $y_1,\ldots,y_n$ are absolutely continuous conditional on  $x_1,\ldots,x_n$, then the  number of interpolated data
points, $I_u=|\{i: y_i = x_i'\hat \beta(u)\}|$, is equal to
$\|\hat\beta(u)\|_0$ with probability one uniformly over
$u\in\mathcal{U}$.
\end{lemma}
\begin{proof}[Proof of Lemma \ref{Lemma:E4partI}]  Step 1. Part (1) follows from the complementary slackness condition for linear programming problems,
see Theorem 4.5 of \cite{BertsimasTsitsiklis}.

Step 2.   To show part (2) consider any $u \in \mathcal{U}$. Trivially we have
$\|\hat\beta(u)\|_0 \leq p$.  Let $Y = (y_1,\ldots,y_n)'$, $\hat
\sigma= (\hat \sigma_1,\ldots,\hat\sigma_p)'$, $X$
be the $n\times p$ matrix with rows $x_i', i=1,\ldots, n$,  $c_u =
( ue', (1-u)e', \lambda\sqrt{u(1-u)} \hat\sigma', \lambda\sqrt{u(1-u)} \hat\sigma')'$,
and  $A = [ I  \  -I \ \ X \ -X ]$, where   $e = (1,1,\ldots,1)'$
denotes an n-vectors of ones, and $I$ denotes
the $n\times n$ identity matrix.  For $w=(\xi^+,\xi^-,\beta^+,\beta^-)$, the primal
problem (\ref{DefQRLP}) can be written as $\min_{w}\{ c_u'w \ : \ Aw=Y, w\geq
0\}$. Matrix  $A$ has rank $n$, since it has linearly independent rows. By Theorem 2.4 of
\cite{BertsimasTsitsiklis} there is at least one optimal basic
solution $\hat w(u)=(\hat \xi^+(u),\hat \xi^-(u),\hat\beta^+(u),\hat\beta^-(u))$, and all basic solutions have at most $n$ non-zero components. Since $\hat \beta(u)=\hat \beta^+(u) - \hat \beta^-(u)$, $\hat \beta(u)$ has at most $n$ non-zero components.

Let $I_u$ denote the number of interpolated points in (\ref{Def:L1QR}) at the quantile index $u$. We have that $n-I_u$
components of $\hat \xi^+(u)$ and $\hat\xi^-(u)$ are non-zero. Therefore,  $\|\hat \beta(u)\|_0 + (n-I_u) \leq n$, which leads to $\|\hat \beta(u)\|_0 \leq I_u$.  By step 3 below this holds with equality with probability 1 uniformly over $u\in \mathcal{U}$, thus establishing part (3).

Step 3. Consider the dual problem $\max_{a}\{ Y'a : A'a \leq c_u\}$ for all $u \in \mathcal{U}$. Conditional on $X$ the feasible region of this problem is the polytope  $R_u=\{ a : A'a \leq c_u\}$. Since $c_u>0$,  $R_u$ is non-empty for all $u\in \mathcal{U}$. Moreover, the form of $A'$ implies that $R_u \subset [-1,1]^n$ so $R_u$ is bounded. Therefore, if
the solution of the dual is not unique for some $u \in
\mathcal{U}$ there exist vertices $a^1, a^2$ connected by an edge of $R_u$ such that
$Y'(a^1-a^2) = 0$. Note that the matrix $A'$ is the same for all
$u\in \mathcal{U}$ so that the direction
$\frac{a^1-a^2}{\|a^1-a^2\|}$ of the edge linking $a^1$ and $a^2$ is generated by a finite number of intersections of hyperplanes associated with the rows of $A'$. Thus, the event $Y'(a^1-a^2) =
0$ is a zero probability event uniformly in $u\in\mathcal{U}$ since $Y$ is absolutely continuous conditional on $X$ and the number of different edge directions is finite. Therefore the dual problem has a unique solution with
probability one uniformly in $u\in \mathcal{U}$. If the dual basic solution is unique, we have
that the primal basic solution is non-degenerate, that is, the
number of non-zero variables equals $n$, see
\cite{BertsimasTsitsiklis}.  Therefore, with probability
one $ \|\hat \beta(u)\|_0 + (n-I_u) = n$, or  $\|\hat
\beta(u)\|_0 = I_u$ for all $u \in \mathcal{U}$.  \end{proof}

\begin{proof}[Proof of Lemma \ref{Lemma:E4}] (Empirical Pre-Sparsity)
That $\hs \leq n \wedge p$ follows from Lemma \ref{Lemma:E4partI}.
We proceed to show the last bound.

Let $\hat a(u)$ be the solution of the dual problem
(\ref{Def:DualL1}), $\widehat T_u = \text{ support} (\hat
\beta(u))$, and $\hs_u = \|\hat \beta(u)\|_0 = |\widehat T_u|$.
For any $j \in \widehat T_u$, from (\ref{Def:Sign}) we have
$(X'\hat a(u))_j =
\sign(\hat\beta_j(u))\lambda\hat\sigma_j\sqrt{u(1-u)}  $ and, for
$j \notin \widehat T_u$ we have $\sign(\hat\beta_j(u)) = 0$.
Therefore, by the Cauchy-Schwarz inequality, and by D.3, with
probability $1-\gamma$ we have
$$\begin{array}{rcl}
 \hs_u\lambda  & = & \sign(\hat\beta(u))'\sign(\hat\beta(u))\lambda
  \leq  \sign(\hat\beta(u))'(X'\hat a(u)) / \min_{j=1,\ldots,p}\hat\sigma_j \sqrt{u(1-u)}\\
 & \leq & 2 \|X\sign(\hat\beta(u))\| \|\hat a(u)\|/\sqrt{u(1-u)}  \leq   2\sqrt{n\phi( \hs_u )} \|\sign(\hat\beta(u))\| \|\hat a(u) \|/\sqrt{u(1-u)},\end{array}
 $$ where we used that $\|\sign(\hat\beta(u))\|_0 = \hs_u$ and $\min_{1\leq j\leq p}\hat\sigma_j \geq 1/2$ with probability $1-\gamma$.
Since $\|\hat a(u) \| \leq \sqrt{n}$,
and  $\|\sign(\hat\beta(u))\| = \sqrt{\hs_u}$ we have $
\hs_u\lambda \leq 2n \sqrt{\hs_u\phi(\hs_u)}\W$. Taking the
supremum over $u \in \mathcal{U}$ on both sides yields the first
result.

To establish the second result, note that
$\widehat s$  $\leq$ $\bar m$ $=$ $\max\left\{ m : m \leq n \wedge p \wedge
4n^2 \phi(m)\W^2/\lambda^2\right\}$. Suppose that  $\bar m > \nn = n/\log(n \vee p)$,
so that $\bar m = m_0 \ell$ for some $\ell > 1$, since $\bar{m} \leq n$ is finite. By
definition, $\bar m$ satisfies  $\bar m \leq 4n^2 \phi(\bar
m)\W^2/\lambda^2.$ Insert the lower bound on
$\lambda$, $\nn$, and  $\bar m = \nn \ell$ in this inequality, and using Lemma
\ref{Lemma:SparseEigenvalue} we obtain:
$$
\bar{m} = \nn\ell  \leq  \frac{4n^2\W^2}{8 \W^2 n \log(n \vee
p)}\frac{ \phi(\nn\ell)}{\phi(\nn)} \leq \frac{n}{2 \log(n \vee
p)} \ceil{\ell} <\frac{n}{\log(n \vee p)} \ell  =  \nn\ell,
$$
which is a contradiction.\end{proof}

\begin{proof}[Proof of Lemma \ref{Lemma:E5}] (Empirical Sparsity) It is convenient to define:
\begin{itemize}
\item[1.] the true rank scores, $a^*_i(u) = u - 1\{y_i \leq
x_i'\beta(u)\} $ for $i=1,\ldots,n$; \item[2.] the estimated rank
scores,
 $ a_i(u) = u - 1\{y_i \leq x_i'\hat \beta(u)\}
$ for $i=1,\ldots,n$; \item[3.] the dual optimal rank scores, $
\hat a(u) $, that solve the dual program (\ref{Def:DualL1}).
\end{itemize}

Let $\widehat T_u$ denote the support of $\hat \beta(u)$, and
$\hs_u = \|\hat \beta(u)\|_0$. Let $\tilde x_{i \widehat
T_u}=(x_{ij}/\hat\sigma_j, j \in \widehat T_u)'$, and
$\hat\beta_{\widehat T_u}(u)=(\hat \beta_j(u), j \in \widehat
T_u)'$. From the complementary slackness characterizations
(\ref{Def:Sign})
\begin{equation}\label{Eq:sqrt-m}
 \sqrt{\hs_u} = \| {\sign}(\hat\beta_{\widehat
T_u}(u)) \| = \left\| \frac{ n \En
\[ \tilde x_{i\widehat T_u}\hat a_i(u)\]}{\lambda\sqrt{u(1-u)}}\right\|.
\end{equation} Therefore we can bound the number $\hs_u$ of non-zero components of $\hat\beta(u)$ provided we can bound the empirical expectation in (\ref{Eq:sqrt-m}). This is achieved in the next step by combining the maximal inequalities and assumptions on the design matrix.

Using the triangle inequality in (\ref{Eq:sqrt-m}), write
{\small $$
\lambda \sqrt{\hs}  \leq \sup_{u\in \mathcal{U}}\left\{\frac{\left\|
 n \En \[\tilde x_{i\widehat T_u}(\hat a_i(u)-a_i(u))\] \right\| + \left\| n
\En\[ \tilde x_{i\widehat T_u}(a_i(u)-a_i^*(u))\] \right\| +
\left\|  n \En\[ \tilde x_{i\widehat T_u}a^*_i(u)\] \right\|}{\sqrt{u(1-u)}}\right\}.
$$} This leads to the inequality
{\small $$
\begin{array}{rc}
\displaystyle \lambda
\sqrt{\hs}  & \leq  \displaystyle \frac{\W}{\displaystyle\min_{j=1,\ldots,p}\hat\sigma_j}\left(\sup_{u \in \mathcal{U}}\left\|
 n \En \[ x_{i\widehat T_u}(\hat a_i(u)-a_i(u))\] \right\| + \sup_{u \in \mathcal{U}} \left\| n
\En\[ x_{i\widehat T_u}(a_i(u)-a_i^*(u))\] \right\|
\right)+\\
 &  \ \ \displaystyle +  \sup_{u\in\mathcal{U}} \left\|  n \En\[  \tilde x_{i\widehat T_u}a^*_i(u)/\sqrt{u(1-u)}\]
\right\|.\\
\end{array}
$$}
Then we bound each of the three components in this display.

(a) To bound the first term, we observe that $\hat a_i(u) \neq
a_i(u)$ only if $y_i = x_i'\hat\beta(u)$. By Lemma
\ref{Lemma:E4partI} the penalized quantile regression fit can
interpolate at most $\hs_u \leq \hs$ points with probability one uniformly
over $u\in \mathcal{U}$. This implies that $\En\[  |\hat a_i(u) -
a_i(u)|^2\] \leq \hs/n$. Therefore,
$$
\begin{array}{rcl}
& & \displaystyle \sup_{u \in \mathcal{U}} \left\| n \En\[  x_{i\widehat T_u}(\hat
a_i(u) -
a_i(u))\] \right\|  \leq  \displaystyle n \sup_{\|\alpha\|_0\leq \hs, \|\alpha\|\leq 1} \sup_{u \in \mathcal{U}} \ \En\[ | \alpha'  x_i| \ |\hat a_i(u) - a_i(u)|  \] \\
& &  \leq \displaystyle    n  \sup_{\|\alpha\|_0\leq \hs, \|\alpha\|\leq 1} \sqrt{\En\[ | \alpha'x_i|^2\]} \sup_{u \in \mathcal{U}}\sqrt{\En\[  |\hat a_i(u) - a_i(u)|^2\]}  \leq  \sqrt{n \phi(\hs) \hs}.
\end{array}
$$

(b) To bound the second term, note that
$$
\begin{array}{rl}
& \displaystyle \sup_{u \in \mathcal{U}} \left\| n \En
\[x_{i\widehat T_u}(a_i(u)-a_i^*(u))\] \right\| \\
&  \leq
\displaystyle \sup_{u \in \mathcal{U}}\left\|  \sqrt{n} \ \Gn\(
x_{i\widehat T_u}(a_i(u)-a_i^*(u))\) \right\| + \sup_{u \in \mathcal{U}} \left\| n \Ep\[
x_{i\widehat T_u}(a_i(u)-a_i^*(u))\]
\right\| \\
 & \leq  \sqrt{n} \epsilon_1(r,\hs)
+\sqrt{n} \epsilon_2(r,\hs).
\end{array}
$$
where for $\psi_i(\beta,u) = (1\{y_i\leq x_i'\beta\}-u)x_i$,
\begin{equation}\label{define errors}
\begin{array}{rll}
\displaystyle \epsilon_1(r,m)& := &\displaystyle   \sup_{ u \in
\mathcal{U}, \beta \in R_u(r,m), \alpha \in \mathbb{S}(\beta)}
\displaystyle  | \Gn (\alpha'\psi_i(\beta,
u)) - \Gn (\alpha'\psi_i(\beta(u),u))|, \\
\displaystyle \epsilon_2(r,m)& := & \displaystyle  \sup_{ u \in
\mathcal{U},  \beta \in R_u(r,m), \alpha \in \mathbb{S}(\beta)}
\displaystyle  \sqrt{n } | \Ep [\alpha'\psi_i(\beta, u)] -
\Ep[\alpha'\psi_i(\beta(u),u)]|,  \text{ and }
\end{array}
 \end{equation}
\begin{equation}\label{define sets R and S}
\begin{array}{rll}
R_u(r,m) & := & \{ \beta \in \RR^p :  \ \beta - \beta(u) \in A_u \ : \|\beta\|_0 \leq m, \ \| J_u^{1/2}( \beta - \beta(u) ) \| \leq r \ \}, \\
\mathbb{S}(\beta) &:= & \{ \alpha \in \RR^{p} : \|\alpha\| \leq 1, \supp (\alpha) \subseteq \supp(\beta) \}.
\end{array}
 \end{equation}
By Lemma \ref{Lemma:error-rate-1} there is a constant $A_{\varepsilon/2}^1$ such that $\sqrt{n}\epsilon_1(r,
\hs)\leq  A_{\varepsilon/2}^1\sqrt{n \hs \log(n\vee p) }\sqrt{\phi(\hs)}$ with probability $1-\varepsilon/2$. By Lemma \ref{Lemma:error-rate-2} we have $\sqrt{n} \epsilon_2(r,\hs)\leq n (\mmu(\hs)/2)  (r \wedge 1 )$.

(c) To bound the last term, by Theorem \ref{Thm:BoundLAMBDA} there exists a constant $A_{\varepsilon/2}^0$ such that with probability $1-\varepsilon/2$
$$
\sup_{u \in \mathcal{U}}\left\| n\En\[ \tilde x_{i\widehat
T_u}a^*_i(u) /\sqrt{u(1-u)}\] \right\| \leq \sqrt{\hs}\Lambda \leq
\sqrt{\hs} \  A_{\varepsilon/2}^0 \ \W\sqrt{n \log p },
$$
where we used that $a_i^*(u) = u - 1\{ u_i \leq u\}$, $i=1,\ldots,n,$ for $u_1,\ldots,u_n$ i.i.d. uniform $(0,1)$.

Combining bounds in (a)-(c), using that
$\min_{j=1,\ldots,p}\hat\sigma_j \geq 1/2$ by condition D.3 with probability $1-\gamma$, we have
$$ \frac{\sqrt{\hs}}{\W}  \leq  \mmu(\hs) \ \frac{n}{\lambda}  (r \wedge 1)  +
\sqrt{\hs}  K_\varepsilon \frac{\sqrt{n\log (n\vee p)
\phi(\hs)}}{\lambda}, $$
with probability at least $1-\varepsilon - \gamma$, for $K_\varepsilon = 2(1 + A_{\varepsilon/2}^0 +
A_{\varepsilon/2}^1)$.
\end{proof}

Next we control the linearization error $\epsilon_2$ defined in (\ref{define errors}).

\begin{lemma}[Controlling linearization error $\epsilon_2$]\label{Lemma:error-rate-2} Under D.1-2
$$
\epsilon_2(r,m) \leq \sqrt{n} \sqrt{\semax{m}}  \left\{ 1 \wedge
 \( 2 [\bar{f}/\underf^{1/2}]  r \) \ \right\}  \text{ for all $r>0$ and $m \leq n$. }$$
\end{lemma}
\begin{proof} By definition
$$ \epsilon_2(r,m) =  \sup_{u \in \mathcal{U}, \beta \in R_u(r,m), \alpha \in \mathbb{S}(\beta)} \sqrt{n}| \Ep[ (\alpha'x_i)\( 1\{y_i \leq x_i'\beta\} - 1\{y_i<x_i'\beta(u)\} \)]|.
$$
Using that $\semax{m} = \sup_{\|\alpha\|\leq 1, \|\alpha\|_0\leq m} \Ep[
|\alpha'x_{i}|^2]$ note that
$$\begin{array}{rl}
| \Ep[ (\alpha'x_i)\( 1\{y_i \leq x_i'\beta\} - 1\{y_i<x_i'\beta(u)\} \)]| & \leq  \Ep[ |\alpha'x_i| \ | 1\{y_i \leq x_i'\beta\} - 1\{y_i<x_i'\beta(u)\} |]  \\
& =_{(1)}  \Ep[ |\alpha'x_i| \Ep[ | 1\{y_i \leq x_i'\beta\} - 1\{y_i<x_i'\beta(u)\} | \big| x] ]    \\
& \leq_{(2)} \Ep[ |\alpha'x_i| \bar f|x_i'\beta-x_i'\beta(u)| ]  \\
& \leq_{(3)} \bar f \sqrt{\Ep[ |\alpha'x_i|^2]} \sqrt{E[|x_i'\beta-x_i'\beta(u)|^2 ]}  \\
& \leq_{(4)} (\bar f/\underf^{1/2}) \sqrt{\semax{m}} r  \\
\end{array}$$
where the equality (1) follows by the law of iterated expectations, (2) follows since $\bar f$ is an upper bound on the conditional density function of $y$, (3) follows by Cauchy-Schwarz, and (4) follows from $\(\Ep[|x_i'(\beta-\beta(u))|^2] \)^{1/2} \leq  \|J_u^{1/2}(\beta-\beta(u))\|/\underf^{1/2}$ by  Lemma \ref{Lemma:E2}.

On the other hand, directly by  Cauchy-Schwarz, we have another bound
$$ \epsilon_2(r,m) \leq \sqrt{n} \sqrt{\semax{m}} \sup_{u \in \mathcal{U}, \beta \in R_u(r,m) } \sqrt{\Ep[\(1\{y_i \leq x_i'\beta\} - 1\{y_i<x_i'\beta(u)\} \)^2 ]} \leq \sqrt{n} \sqrt{\semax{m}}.
$$ since $\sqrt{\Ep[\(1\{y_i \leq x_i'\beta\} - 1\{y_i<x_i'\beta(u)\} \)^2 ]} \leq 1$.

The result follows.
\end{proof}

 Next we proceed to control the empirical error $\epsilon_1$
defined in (\ref{define errors}).  We shall need the following preliminary result on the uniform $L_2$ covering numbers (\cite{vdV-W})
of a relevant function class.

\begin{lemma}\label{Lemma:UEI}
(1) Consider a fixed subset $T \subset \{1,2,\ldots,p\}$, $|T|=m$. The
class of functions
$$
\mathcal{F}_T = \left\{ \alpha'( \psi_i(\beta,u) -
\psi_i(\beta(u),u)) \  : u \in \mathcal{U}, \alpha \in
\mathbb{S}(\beta), \supp(\beta) \subseteq T  \right\}
$$
 has a VC index bounded by $cm$ for some universal constant $c$. (2) There are universal constants
 $C$ and $c$ such that for any $m\leq n$  the function class
$$\mathcal{F}_{m} = \{
\alpha'(\psi_i(\beta,u)-\psi_i(\beta(u),u)) :  u \in \mathcal{U}, \beta \in \RR^p,
\|\beta\|_0\leq m, \alpha \in \mathbb{S}(\beta) \} \ $$
has the the uniform covering numbers bounded as
$$
\sup_{Q} N(\epsilon\|F_{m}\|_{Q,2}, \mathcal{F}_{m}, L_2(Q))
\leq C\( \frac{16e}{\epsilon}\)^{2(cm-1)}\(\frac{ep}{m}\)^m, \ \ \epsilon>0.$$
\end{lemma}
\begin{proof} The proof involves standard combinatorial arguments
and is relegated to the Supplementary Material Appendix \ref{Supplementary Material}. \comment{The proof of part (1) follows by showing that the corresponding
subgraph class is created by at most $K$ operations
of taking unions, intersections, and complements of VC classes of sets
with VC index at most $m$, and then appealing to \cite{vdV-W} Lemma 2.6.17.
We relegate the details to the Supplementary Material Appendix \ref{Supplementary Material} for brevity.

To show part (2) let $\mathcal{F}_T$ denote a restriction of $\mathcal{F}_{m}$
for a particular choice of $m$ non-zero components. Part (1) implies $ N(\epsilon\|F_T\|_{Q,2}, \mathcal{F}_T,L_2(Q))
\leq C(cm)(16e)^{cm}( 1/\epsilon)^{2(cm-1)},$ where $C$
is a universal constant (see \cite{vdV-W} Theorem 2.6.7). Since we have at most $\binom{p}{m}
\leq (ep/m)^m $  different restrictions $T$, the total covering
number is bounded according the statement of the lemma.}
\end{proof}

\begin{lemma}[Controlling empirical error $\epsilon_1$]\label{Lemma:error-rate-1}
Under D.1-2 there exists a universal constant $A$ such that with
probability $1-\delta$ $$
 \epsilon_1(r,m) \leq A \delta^{-1/2} \sqrt{m \log(n\vee p)}\sqrt{ \phi(m)} \ \ \text{ uniformly for all } r>0 \text{ and } m \leq n.
$$\end{lemma}
\begin{proof}
By definition $ \epsilon_1(r,m) \leq
\sup_{f \in \mathcal{F}_{m}}| \Gn(f)|.$ From Lemma
\ref{Lemma:UEI} the uniform covering number of
$\mathcal{F}_{m}$ is bounded by
$C\(16e/\epsilon\)^{2(cm-1)}(ep/m)^m.$ Using Lemma \ref{master
lemma} with $N=n$ and $\bm = p$ we have that uniformly in $m \leq n$, with
probability at least $1-\delta$
 \begin{equation}\label{e1-eq:1}
 \sup_{f \in \mathcal{F}_{m}} | \Gn (f ) | \leq A \delta^{-1/2}  \sqrt{m \log(n\vee p)}\max\left\{  \sup_{f\in \mathcal{F}_{m}}\Ep[f^2]^{1/2}, \sup_{f\in \mathcal{F}_{m}}\En[f^2]^{1/2} \right\}\end{equation}
By $ | \alpha'\(\psi_i(\beta,u) - \psi_i(\beta(u),u)\) |\leq |\alpha'x_i|$ and definition of $\phi(m)$
\begin{equation}\label{e1-eq:2}
 \En[f^2] \leq \En[|\alpha'x_i|^2] \leq \phi(m) \ \ \mbox{and} \ \  \Ep[ f^2 ] \leq \Ep[|\alpha'x_i|^2] \leq \phi(m). \end{equation}
 Combining (\ref{e1-eq:2}) with
(\ref{e1-eq:1}) we obtain the result.
\end{proof}

(c) The next lemma provides a  bound on maximum $k$-sparse
eigenvalues, which we used in some of the derivations presented
earlier.

\begin{lemma}\label{Lemma:SparseEigenvalue}
Let $M$ be a semi-definite positive matrix and  $ \phi_M(k) = \sup
\{ \ \alpha'M\alpha \ : \alpha \in \RR^p, \|\alpha\|=1, \|\alpha\|_0 \leq k \ \}$. For any
integers $k$ and $\ell k$ with $\ell \geq 1$,  we have $
\phi_M(\ell k) \leq \lceil \ell \rceil \phi_M(k).$ 
\end{lemma}
\begin{proof}  Let $\bar \alpha$ achieve $\phi_M(\ell k)$. Moreover let
$\sum_{i=1}^{\lceil \ell \rceil} \alpha_i = \bar \alpha$ such that
$\sum_{i=1}^{\lceil \ell \rceil} \|\alpha_i\|_0 = \|\bar
\alpha\|_0$. We can choose $\alpha_i$'s such that $\|\alpha_i\|_0
\leq k$ since $\lceil \ell \rceil k \geq \ell k$. Since $M$ is
positive semi-definite, for any $i, j$ w $\alpha_i'M\alpha_i +
\alpha_j'M\alpha_j \geq 2\left|\alpha_i'M\alpha_j\right|.$
Therefore
$$
\begin{array}{rcl}
\phi_M(\ell k) & = &  \bar \alpha' M \bar\alpha =  \displaystyle \sum_{i=1}^{\lceil \ell \rceil} \alpha_i'M\alpha_i + \sum_{i=1}^{\lceil \ell \rceil} \sum_{j\neq i} \alpha_i'M\alpha_j \leq \sum_{i=1}^{\lceil \ell \rceil} \left\{ \alpha_i'M\alpha_i + ({\lceil \ell \rceil}-1)\alpha_i'M\alpha_i \right\}  \\
& \leq & \displaystyle {\lceil \ell \rceil} \sum_{i=1}^{\lceil
\ell \rceil} \|\alpha_i\|^2 \phi_M(\|\alpha_i\|_0)  \leq {\lceil \ell \rceil} \max_{i=1,\ldots,{\lceil
\ell \rceil}} \phi_M(\|\alpha_i\|_0) \leq {\lceil \ell \rceil}
\phi_M(k)
\end{array}
$$ where we used that $\sum_{i=1}^{\lceil \ell \rceil} \|\alpha_i\|^2=1$. \end{proof}

\section{Proof of Theorem 4}

\begin{proof}[Proof of Theorem \ref{Thm:Selection}]  By assumption
$\sup_{u\in \mathcal{U}}\| \hat \beta(u) - \beta(u) \|_{\infty} \leq \sup_{u\in \mathcal{U}}\| \hat \beta(u)
- \beta(u) \| \leq r^o < \inf_{u\in \mathcal{U}} \min_{ j\in T_u }
|\beta_j(u)|,$ which immediately implies the inclusion event (\ref{Eq:support}), since the converse of this event implies
$\| \hat \beta(u) - \beta(u) \|_{\infty}  \geq  \inf_{u\in \mathcal{U}}\min_{ j\in T_u }
|\beta_j(u)|.$

Consider the hard-thresholded estimator next.  To establish the
inclusion, we note that $\inf_{u\in\mathcal{U}}\min_{j \in T_u} |\hat \beta_j(u)|   \geq  \inf_{u\in\mathcal{U}}\min_{j \in T_u} \{ |\beta_j(u)| - |\beta_j(u) - \hat \beta_j(u)| \}
 >  \inf_{u\in\mathcal{U}}\min_{j \in T_u} |\beta_j(u)| - r^o > \gamma,$
by assumption on $\gamma$. Therefore $\inf_{u\in \mathcal{U}}\min_{j \in T_u} |\hat
\beta_j(u)|  > \gamma$ and
 $\text{support } (\beta(u)) \subseteq \text{support }
(\bar\beta(u))$ for all $u \in \mathcal{U}$. To establish the opposite inclusion, consider
$e_n$ $=$ $\sup_{u\in\mathcal{U}}$$\max_{j \notin T_u}$ $|\hat \beta_j(u)|.$
By definition of $r^o$,  $e_n \leq r^o$ and therefore $e_n < \gamma$ by the assumption on $\gamma$.
By the hard-threshold rule, all components smaller than
$\gamma$ are excluded from the support of $\bar \beta(u)$ which yields $ \text{support } (\bar\beta(u)) \subseteq \text{support }
(\beta(u))$.
\end{proof}

\section{Proof of Lemma 8 (used in Theorem 5)}

\noindent\begin{proof}[Proof of Lemma \ref{Lemma:E.6}] (Sparse Identifiability and Control of Empirical Error)
The proof of claim (\ref{E.2-RNI-tilde-q}) of this lemma follows  identically  the
proof of claim (\ref{E.2-RNI}) of Lemma \ref{Lemma:E2}, given in Appendix B,
after replacing $A_u$ with $\widetilde A_u$.  Next we bound the empirical error
 \begin{equation}\label{Def:epsilon3}
\begin{array}{rcl}\displaystyle\sup_{u \in \mathcal{U}, \delta \in\widetilde A_u(\widetilde m),  \delta\neq0 }\frac{ \left|\emperror_u(\delta) \right|}{\|\delta\|}  & \leq &\displaystyle \sup_{u \in \mathcal{U}, \delta \in\widetilde A_u(\widetilde m), \delta \neq 0 }  \frac{1}{\|\delta\|\sqrt{n}} \left| \int_0^{1} \delta'\Gn(\psi_i( \beta(u) + \gamma \delta,u))   d\gamma\right| \\
&  \leq & \displaystyle \frac{1}{\sqrt{n}}\epsilon_3(\widetilde m)
\end{array}
 \end{equation} where $\epsilon_3(\widetilde m):=\sup_{f\in\mathcal{\widetilde F}_{\widetilde m}}|\Gn(f)|$ and the class of functions $\mathcal{\widetilde F}_{\widetilde m}$ is defined in Lemma \ref{Lemma:VC-AB}. The result
 follows from the bound on $\epsilon_3(\widetilde m)$ holding uniformly in $\widetilde m \leq n$ given in Lemma \ref{Lemma:error-rate-A-B}.
\end{proof}

Next we control the empirical error $\epsilon_3$ defined in (\ref{Def:epsilon3}) for $\mathcal{\widetilde F}_{\widetilde{m}}$ defined below. We first bound uniform covering numbers of $\mathcal{\widetilde F}_{\widetilde{m}}$.

\begin{lemma}\label{Lemma:VC-AB}
Consider a fixed subset $T \subset \{1,2,\ldots,p\}$,
$T_u=\supp(\beta(u))$ such that $|T\setminus T_u| \leq \widetilde{m}$ and $|T_u| \leq s$ for some $u \in \mathcal{U}$. The class of functions
$$
\mathcal{F}_{T,u} = \left\{ \alpha'x_i( 1\{y_i\leq x_i'\beta\} - u
) \  : \alpha \in \mathbb{S}(\beta), \supp(\beta) \subseteq T
\right\} $$
 has a VC index bounded by $c(\widetilde m + s) + 2$. The class of functions
$$\mathcal{\widetilde F}_{\widetilde{m}} = \left\{ \mathcal{F}_{T,u} : \ u \in \mathcal{U},  T \subset \{1,2,\ldots,p\}, |T\setminus T_u| \leq \widetilde{m}  \right\}, $$
obeys, for some universal constants $C$ and $c$ and each
$\epsilon > 0$,
$$
\sup_{Q} N(\epsilon\|{\widetilde F}_{\widetilde{m}}\|_{Q,2}, \mathcal{\widetilde F}_{\widetilde{m}}, L_2(Q)) \leq C\(
32e/\epsilon\)^{4(c(\widetilde m + s)+2)} p^{2\widetilde
m}\left|\cup_{u \in \mathcal{U}} T_u\right|^{2s}.$$
\end{lemma}
\begin{proof}   The proof involves standard combinatorial arguments
and is relegated to the Supplementary Material Appendix \ref{Supplementary Material}.
\end{proof}

\begin{lemma}[Controlling empirical error $\epsilon_3$]\label{Lemma:error-rate-A-B}
Suppose that D.1 holds and $\left| \cup_{u \in \mathcal{U}} T_u\right| \leq n$. There exists a universal constant $A$ such that with
probability at least $1-\delta$,   $$
 \epsilon_3(\widetilde m):=\sup_{f \in \mathcal{\widetilde F}_{\widetilde m}} | \Gn (f ) | \leq A\delta^{-1/2}\sqrt{(\widetilde m \log (n\vee p) + s \log n) \phi(\widetilde m+s)} \text{ for all $\widetilde m \leq n$. }
$$
\end{lemma}
\begin{proof}
\comment{Lemma \ref{Lemma:VC-AB} bounds the uniform covering number of
$\mathcal{\widetilde F}_{\widetilde m}$. Using Lemma \ref{master
lemma} with $m = \widetilde m$ and $\bm = p^2 \cdot n^{2s/\widetilde
m}$,  we conclude that uniformly
in $\widetilde m \leq n$
\begin{equation}\label{e1-eq:3}
 \sup_{f \in \mathcal{\widetilde F}_{\widetilde m}} | \Gn (f ) |  \leq A\delta^{-1/2}  \sqrt{ \widetilde  m \log ( n \vee  \bm)} \cdot  \max\left\{  \sup_{f\in \mathcal{\widetilde F}_{\widetilde m}}\Ep[f^2]^{1/2}, \sup_{f\in \mathcal{\widetilde F}_{\widetilde m}}\En[f^2]^{1/2} \right\}
 \end{equation}}
Lemma \ref{Lemma:VC-AB} bounds the uniform covering number of
$\mathcal{\widetilde F}_{\widetilde m}$. Using Lemma \ref{master
lemma} with $N \leq 2n$, $m = \widetilde m+s$ and $\bm = p^{2([m-s]/m)} \cdot n^{2(s/m)} = p^{2(\widetilde m/[\widetilde m+s])}  \cdot n^{2(s/[\widetilde
m+s])}$,  we conclude that uniformly
in $0\leq \widetilde m \leq n$
\begin{equation}\label{e1-eq:3}
 \begin{array}{c}
 \displaystyle \sup_{f \in \mathcal{\widetilde F}_{\widetilde m}} | \Gn (f ) |  \leq A\delta^{-1/2}  \sqrt{ (\widetilde  m+s) \log ( n \vee  \bm)} \cdot  \max\left\{  \sup_{f\in \mathcal{\widetilde F}_{\widetilde m}}\Ep[f^2]^{1/2}, \sup_{f\in \mathcal{\widetilde F}_{\widetilde m}}\En[f^2]^{1/2} \right\}\\
\leq A'\delta^{-1/2}  \sqrt{ \widetilde  m \log ( n \vee  p) + s \log n} \cdot  \max\left\{  \sup_{f\in \mathcal{\widetilde F}_{\widetilde m}}\Ep[f^2]^{1/2}, \sup_{f\in \mathcal{\widetilde F}_{\widetilde m}}\En[f^2]^{1/2} \right\}\\

 \end{array}
 \end{equation}
with probability at least $1-\delta$.
The result follows, since for any $f \in \mathcal{\widetilde F}_{\widetilde m}$, the
corresponding vector $\alpha$ obeys$\|\alpha\|_0 \leq
\widetilde m+s$, so that $ \En[f^2] \leq \En[ |\alpha'x_i|^2 ] \leq \phi(\widetilde m+s)$ and $ \Ep[f^2] \leq \Ep[ |\alpha'x_i|^2 ] \leq \phi(\widetilde m+s)$ by definition of $\phi(\widetilde m+s)$.
\end{proof}

\section{Maximal Inequalities for a Collection of Empirical
Processes}\label{Sec:SubGaussian}

The main results here are  Lemma \ref{expo1} and Lemma \ref{master lemma},
used in the proofs of Theorem \ref{Thm:BoundLAMBDA} and Theorems \ref{Thm:Sparsity} and \ref{Thm:MainTwoStep}, respectively.
Lemma \ref{master lemma} gives a maximal inequality that controls the empirical process
uniformly over a collection of classes of functions using
class-dependent bounds.   We need this lemma  because the standard
maximal inequalities applied to the union of function classes
yield a single class-independent bound that is too large for our
purposes. We prove Lemma \ref{master lemma} by first stating Lemma \ref{expo1},
giving a bound on tail probabilities of a separable sub-Gaussian
process, stated in terms of uniform covering numbers. Here we want
to explicitly trace the impact of covering numbers on the tail
probability, since these covering numbers grow rapidly under
increasing parameter dimension and thus help to tighten the probability bound.  Using the symmetrization
approach, we then obtain Lemma \ref{expo3}, giving a bound on tail
probabilities of a general separable empirical process, also
stated in terms of uniform covering numbers. Finally, given a
growth rate on the covering numbers, we obtain Lemma
\ref{master lemma}.

\begin{lemma}[Exponential Inequality for Sub-Gaussian Process]\label{expo1} Consider any linear zero-mean separable process $\{\mathbb{G}(f):f\in \mathcal{F}\}$,  whose index set $\mathcal{F}$ includes zero, is equipped with a $L_2(P)$ norm, and has envelope $F$. Suppose further that the process is sub-Gaussian, namely for each $g \in \mathcal{F} - \mathcal{F}$:
$ \mathbb{P}\left\{ \left\vert \mathbb{G}(g) \right\vert >\eta
\right\} \leq 2\exp \left( -\frac{1}{2}\eta ^{2}/D^{2}\|g\|_{P,2}
^{2}\right)$ for any $\eta >0$, with $D$ a positive constant; and
suppose that we have the following upper bound on the  $L_2(P)$
covering numbers for $\mathcal{F}$:
$$N(\epsilon\|F\|_{P,2}, \mathcal{F}, L_2(P)) \leq
n(\epsilon, \mathcal{F},P) \text{ for each } \epsilon >0, $$ where
$n(\epsilon, \mathcal{F},P)$ is increasing in $1/\epsilon$, and
 $\epsilon \sqrt{\log n(\epsilon, \mathcal{F},P)} \to 0$ as $1/\epsilon \to \infty$ and is decreasing in $1/\epsilon$. Then for $K>D$, for some universal constant $c < 30$, $\rho(\mathcal{F},P) := \sup_{f \in \mathcal{F}} \|f\|_{P,2}/ \|F\|_{P,2}$,
$$
\mathbb{P}\left \{ \frac{\sup_{f \in \mathcal{F}}| \mathbb{G}(f)
|}{ \|F\|_{P,2} \int_{0}^{\rho(\mathcal{F},P)/4} \sqrt{ \log  n
(x, \mathcal{F},P)} d x } >  c K   \right \} \leq
\int_{0}^{\rho(\mathcal{F},P)/2}  \epsilon^{-1} n (\epsilon,
\mathcal{F},P)^{- \{ (K/D)^{2} -1 \}} d\epsilon.
$$
\end{lemma}

The result of Lemma \ref{expo1} is in spirit of the Talagrand tail
inequality for Gaussian processes. Our result is less sharp than Talagrand's result in the Gaussian case (by
a log factor), but it applies to more general sub-Gaussian processes.

In order to prove a bound on tail probabilities of a general
separable empirical process, we need to go through a
symmetrization argument. Since we use a data-dependent threshold, we
need an appropriate extension of the classical symmetrization
lemma to allow for this.  Let us call a threshold function
$x:\RR^n \mapsto \RR$ $k$-sub-exchangeable if, for any $v, w \in
\RR^n$ and any vectors $\tilde v, \tilde w$ created by the
pairwise exchange of the components in $v$ with components in $w$,
we have that $ x(\tilde v) \vee x(\tilde w) \geq [x(v) \vee
x(w)]/k.$ Several functions satisfy this property, in particular
$x(v) = \|v\|$ with $k = \sqrt{2}$ and constant functions with
$k=1$. The following result generalizes the standard
symmetrization lemma for probabilities (Lemma 2.3.7 of
\cite{vdV-W}) to the case of a random threshold $x$ that is
sub-exchangeable.

\begin{lemma}[Symmetrization with Data-dependent Thresholds]\label{Lemma:SymTau}
Consider arbitrary independent stochastic processes
$Z_1,\ldots,Z_n$ and arbitrary functions $\mu_1,\ldots,\mu_n :
\mathcal{F}\mapsto \RR$. Let $x(Z)=x(Z_1,\ldots,Z_n)$ be a
$k$-sub-exchangeable random variable and for any $\tau \in (0,1)$
let $q_\tau$ denote the $\tau$ quantile of $x(Z)$, $\bar{p}_\tau
:= P(x(Z) \leq q_\tau) \geq \tau$, and $p_\tau := P(x(Z) <
q_\tau)\leq \tau$. Then
$$  P\( \left\| \sum_{i=1}^n Z_i \right\|_{\mathcal{F}} > x_0 \vee x(Z) \) \leq \frac{4}{\bar{p}_\tau } P\( \left\|\sum_{i=1}^n \varepsilon_i \(Z_i - \mu_i\)  \right\|_{\mathcal{F}} > \frac{x_0 \vee x(Z)}{4k} \) + p_\tau $$
where $x_0$ is a constant such that $ \inf_{f \in \mathcal{F}} P\(
\left|\sum_{i=1}^n Z_i(f)\right| \leq \frac{x_0}{2} \) \geq 1  -
\frac{\bar{p}_\tau}{2}$.
\end{lemma}

Note that we can recover the classical symmetrization lemma for
fixed thresholds by setting $k=1$, $\bar p_{\tau}=1$, and
$p_{\tau} = 0$.

\begin{lemma}[Exponential inequality for separable empirical process]\label{expo3} Consider a separable empirical process $\mathbb{G}_n(f) = n^{-1/2} \sum_{i=1}^n \{f(Z_i) - \Ep[f(Z_i)]\}$ and the empirical
measure $\mathbb{P}_n$ for $Z_1,\ldots,Z_n$, an underlying i.i.d. data sequence.  Let $K>1$ and $\tau \in (0,1)$ be constants,  and  $e_n(\mathcal{F}, \mathbb{P}_n)=e_n(\mathcal{F},Z_1,\ldots,Z_n)$ be a $k$-sub-exchangeable random variable, such that
$$
\|F\|_{\mathbb{P}_n,2} \int_{0}^{\rho(\mathcal{F},\mathbb{P}_n)/4}
\sqrt{\log  n (\epsilon, \mathcal{F}, \mathbb{P}_n)} d \epsilon  \leq
e_n(\mathcal{F}, \mathbb{P}_n)\text{ and }  \sup_{f \in
\mathcal{F}} \text{var}_{\mathbb{P}}f \leq \frac{\tau}{2} (4kcK
e_n(\mathcal{F},\mathbb{P}_n) )^2
$$
for the same constant $c>0$ as in Lemma \ref{expo1}, then
$$
\mathbb{P}\left\{\sup_{f \in \mathcal{F}} |\mathbb{G}_n(f)| \geq
4kcK e_n(\mathcal{F}, \mathbb{P}_n)\right\}  \leq  \frac{4}{\tau}
\Ep_{\mathbb{P}}\left(\left[ \int_{0}^{\rho(\mathcal{F},
\mathbb{P}_n)/2}  \epsilon^{-1} n (\epsilon, \mathcal{F}, \mathbb{P}_n)^{-\{ K^2
-1 \}} d\epsilon  \right]\wedge 1 \right) + \tau.
$$
\end{lemma}

Finally, our main result in this section is as follows.

\begin{lemma}[Maximal Inequality for a Collection of Empirical Processes]\label{master lemma} Consider a collection of separable empirical processes $\mathbb{G}_n(f) = n^{-1/2} \sum_{i=1}^n \{f(Z_i) - \Ep[f(Z_i)]\}$, where $Z_1,\ldots,Z_n$ is an underlying i.i.d. data sequence, defined over function classes $\mathcal{F}_m, m=1,\ldots,N$ with envelopes $F_m = \sup_{f \in \mathcal{F}_m} |f(x)|, m=1,\ldots,N$, and with upper bounds on the uniform covering numbers of $\mathcal{F}_m$ given for all $m$ by
$$
n(\epsilon, \mathcal{F}_m,\mathbb{P}_n) =  (N \vee n\vee \bm)^m (
\omega/\epsilon )^{\upsilon m},  \ 0 < \epsilon < 1,
$$
with some constants $\omega>1$, $\upsilon>1$, and $\theta_m \geq \theta_0$. For a constant
$C:=(1+\sqrt{2\upsilon})/4$ set
$$
e_n(\mathcal{F}_m, \mathbb{P}_n) = C\sqrt{m \log (N \vee n \vee \bm \vee
\omega)}\max\left\{ \sup_{f \in \mathcal{F}_m}
\|f\|_{\mathbb{P},2}, \ \sup_{f \in \mathcal{F}_m}
\|f\|_{\mathbb{P}_n,2} \right\}. $$ Then, for any $\delta \in
(0,1/6)$, and any constant $K\geq \sqrt{2/\delta}$ we have
$$
\sup_{f \in \mathcal{F}_m} | \mathbb{G}_n(f)| \leq  4\sqrt{2}c K
e_n(\mathcal{F}_m, \mathbb{P}_n) , \text{ for all } m \leq N,
$$
with probability at least $1-\delta$, provided that $N \vee n \vee \theta_0 \geq 3$; the constant $c$
is the same as in Lemma \ref{expo1}.
\end{lemma}
\begin{proof}[Proof of Lemma \ref{expo1}] The strategy of the proof is similar to the proof of Lemma 19.34 in \cite{vdV01}, page 286 given for the expectation
of a supremum of a process; here we instead bound tail
probabilities and also compute all constants explicitly.

Step 1.  There exists   a sequence of nested partitions of
$\mathcal{F}$,
 $\{
 ( \mathcal{F}_{qi}, i =1,\ldots, N_q), q= q_0, q_0 + 1, \ldots
 \}$
 where the $q$-th partition consists of sets of $L_2(P)$ radius at most $\|F\|_{P,2} 2^{-q}$, and $q_0$ is the largest positive integer such that $ 2^{-q_0} \leq \rho(\mathcal{F},P)/4$ so that $q_0 \geq 2$. The existence of such a partition follows from a standard argument, e.g. \cite{vdV01}, page 286.

Let $f_{qi}$ be an arbitrary point of $\mathcal{F}_{qi}$. Set
$\pi_q (f) = f_{qi}$ if $f \in  \mathcal{F}_{qi}$. By separability
of the process, we can replace $\mathcal{F}$ by $\cup_{q,i}
f_{qi}$, since the supremum norm of the process can be computed by
taking this set only. In this case, we can decompose $f -
\pi_{q_0}(f) = \sum_{q =q_0+1}^{\infty} (\pi_{q}(f) -
\pi_{q-1}(f))$. Hence by linearity $ \mathbb{G}(f) -
\mathbb{G}(\pi_{q_0}(f))  = \sum_{q
=q_0+1}^{\infty} \mathbb{G}(\pi_{q}(f) - \pi_{q-1}(f)),$
so that \begin{eqnarray*}
 \mathbb{P}\Big\{ \sup_{f \in \mathcal{F}}|\mathbb{G}(f) |  >  \sum_{q =q_0}^{\infty} \eta_q \Big \}  & \leq &   \sum_{q ={q_0+1}}^{\infty}
\mathbb{P}\left\{ \max_{f} |\mathbb{G}(\pi_{q}(f) -\pi_{q-1}(f))|
>  \eta_q \right\} \\
& + & \mathbb{P}\left\{ \max_{f}
|\mathbb{G}(\pi_{q_0}(f))| >  \eta_{q_0} \right\},
\end{eqnarray*}for constants $\eta_q$ chosen below.

Step 2.  By construction of the partition sets $\|\pi_{q}(f) -
\pi_{q-1}(f)\|_{P,2} \leq 2 \|F\|_{P,2} 2^{-(q-1)} \leq  4
\|F\|_{P,2} 2^{-q}, \ \mbox{for} \ q \geq q_0+1.$ Setting  $\eta
_{q} = 8 K  \|F\|_{P,2} {2^{-q} } \sqrt{  \log  N_q}, $ using
sub-Gaussianity, setting $K>D$, using that $ 2\log N_q \geq \log
N_{q} N_{q-1} \geq \log n_q $, using that $q \mapsto \log n_q$ is
increasing in $q$, and $2^{-q_0} \leq \rho(\mathcal{F}, P)/4$,  we
obtain \begin{eqnarray*} & & \sum_{q =q_0+1}^{\infty}
\mathbb{P}\left\{ \max_{f} |\mathbb{G}(\pi_{q}(f) -\pi_{q-1}(f))|
>  \eta_q \right\}  \leq
\sum_{q =q_0+1}^{\infty} N_{q}  N_{q-1} 2\exp \left( -\eta_q ^{2}/(4 D  \|F\|_{P,2}2^{-q})^{2}\right) \\
 & & \leq   \sum_{q =q_0+1}^{\infty} N_{q}  N_{q-1}2\exp \left( -  (K/D)^{2}  2 \log N_q\right)   \leq   \sum_{q =q_0+1}^{\infty} 2\exp \left( - \{ (K/D)^{2} -1 \} \log n_q  \right) \\
&  & \leq   \int_{q_0}^{\infty} 2\exp \left( - \{ (K/D)^{2} -1 \}
\log n_{q}  \right) dq =  \int_{0}^{\rho(\mathcal{F},P)/4}  (x\ln
2)^{-1}  2 n(x, \mathcal{F},P)^{- \{ (K/D)^{2} -1 \}} dx.
 \end{eqnarray*}
By Jensen's inequality we have  $ \sqrt{\log N_q} \leq a_q:=\sum_{j=q_0}^q
\sqrt{\log n_j}$, so that we obtain $\sum_{q =q_0+1}^{\infty} \eta_q$  $\leq$
$8 \sum_{q =q_0+1}^{\infty} K \|F\|_{P,2} 2^{-q} a_q.$ Letting
$b_q = 2 \cdot 2^{-q}$, noting  $a_{q+1} - a_{q} = \sqrt{ \log
n_{q+1} }$ and $b_{q+1} - b_q = - 2^{-q}$, we get using summation
by parts \begin{eqnarray*}  & & \sum_{q =q_0+1}^{\infty}
2^{-q} a_q
    =    - \sum_{q =q_0+1}^{\infty}  (b_{q+1} - b_q ) a_q
     = - a_q b_{q}|^{\infty}_{q_0+1} + \sum_{q =q_0+1}^\infty b_{q+1} (a_{q+1}-a_q) \\
    &&   =   2 \cdot 2^{-(q_0+1)} \sqrt{ \log n_{q_0+1} }  +  \sum_{q =q_0+1}^{\infty} 2 \cdot 2^{-(q+1)} \sqrt{ \log n_{q+1} }
     = 2 \sum_{q =q_0+1}^{\infty}  2^{-q} \sqrt{ \log n_{q} },
 \end{eqnarray*}
where we use the assumption that $2^{-q }\sqrt{ \log  n_q} \to 0$
as $q \to \infty$, so that $- a_q b_{q}|^{\infty}_{q_0+1} = 2
\cdot 2^{-(q_0+1)} \sqrt{ \log n_{q_0+1} }$.  Using that $ 2^{-q}
\sqrt{ \log n_{q} }$ is decreasing in $q$ by assumption,
$$
2 \sum_{q =q_0+1}^{\infty}  2^{-q} \sqrt{ \log n_{q} }
 \leq 2\int_{q_0}^{\infty} 2^{-q} \sqrt{\log  n (2^{-q}, \mathcal{F},P)} d q.
$$
Using a change of variables and that $2^{-q_0} \leq
\rho(\mathcal{F}, P)/4$, we finally conclude that
$$
\sum_{q =q_0+1}^{\infty} \eta_q  \leq K \|F\|_{P,2}\frac{16 }{\log
2} \int_{0}^{\rho(\mathcal{F},P)/4} \sqrt{\log  n (x,
\mathcal{F},P)} d x.
$$

Step 3. Letting $\eta_{q_0} = K \|F\|_{P,2} \rho(\mathcal{F},P)
\sqrt{2 \log N_{q_0} } $, recalling that $N_{q_0} = n_{q_0}$,
using that $\|\pi_{q_0}(f)\|_{P,2} \leq  \|F\|_{P,2}$ and
sub-Gaussianity, we conclude  \begin{eqnarray*}
& & \mathbb{P}\Big\{  \max_{f} |\mathbb{G}(\pi_{q_0}(f) )| >  \eta_{q_0} \Big \} \leq   n_{q}  2\exp \left( -  (K/D)^{2} \log n_q\right)   \leq    2\exp \left( - \{ (K/D)^{2} -1 \} \log n_q  \right) \\
& & \leq   \int_{q_0-1}^{q_0} 2\exp \left( - \{ (K/D)^{2} -1 \}
\log n_{q}  \right) dq =
\int_{\rho(\mathcal{F},P)/4}^{\rho(\mathcal{F},P)/2}  (x\ln
2)^{-1}  2 n(x, \mathcal{F},P)^{- \{ (K/D)^{2} -1 \}} dx.
 \end{eqnarray*}
Also, since $n_{q_0} = n(2^{-q_0}, \mathcal{F},P)$, $2^{-q_0} \leq
\rho(\mathcal{F},P)/4$, and $n(x, \mathcal{F},P)$ is increasing in
$1/x$, we obtain
 $
\eta_{q_0} \leq 4 \sqrt{ 2} K \|F\|_{P,2} \int_{0}^{
\rho(\mathcal{F},P)/4}  \sqrt{\log n(x, \mathcal{F},P)} dx. $

Step 4. Finally, adding the bounds on tail probabilities from
Steps 2 and 3 we obtain the tail bound stated in the main text.
Further, adding bounds on $\eta_q$ from Steps 2 and 3, and using
$ c= 16/{\log 2} + 4 \sqrt{2} < 30$,  we obtain $ \sum_{q
=q_0}^{\infty} \eta_q  \leq c K \|F\|_{P,2}
\int_{0}^{\rho(\mathcal{F},P)/4} \sqrt{\log  n (x, \mathcal{F},P)} d
x. $

\end{proof}

\begin{proof}[Proof of Lemma \ref{Lemma:SymTau}]
The proof proceeds analogously to the proof of Lemma 2.3.7 (page
112) in \cite{vdV-W} with the necessary adjustments. Letting
$q_\tau$ be the $\tau$ quantile of $x(Z)$ we have
$$ P\left \{ \left\| \sum_{i=1}^n Z_i \right\|_{\mathcal{F}} > x_0 \vee x(Z) \right \} \leq P \left \{ x(Z) \geq q_\tau, \left\| \sum_{i=1}^n Z_i \right\|_{\mathcal{F}} > x_0 \vee x(Z) \right \} + P\{ x(Z) < q_\tau \}.$$ Next we bound the first term of the expression above. Let $Y=(Y_1,\ldots,Y_n)$ be an independent copy of $Z=(Z_1,\ldots,Z_n)$, suitably defined on a product space. Fix a realization of $Z$ such that $x(Z) \geq q_\tau$ and $\left\| \sum_{i=1}^n Z_i\right\|_{\mathcal{F}} > x_0 \vee x(Z)$. Therefore $\exists f_Z \in \mathcal{F}$ such that $\left| \sum_{i=1}^n Z_i(f_Z) \right| > x_0 \vee x(Z)$. Conditional on such a $Z$ and using the triangular inequality we have that
$$\begin{array}{rcl}
 P_Y\left\{ x(Y) \leq q_\tau, \left| \sum_{i=1}^n Y_i(f_Z) \right| \leq \frac{x_0}{2} \right\} & \leq & P_Y\left\{ \left| \sum_{i=1}^n (Y_i - Z_i)(f_Z)\right| > \frac{x_0 \vee x(Z) \vee x(Y)}{2}\right\}\\
 & \leq & P_Y\left\{ \left\| \sum_{i=1}^n (Y_i - Z_i)\right\|_{\mathcal{F}} > \frac{x_0 \vee x(Z) \vee x(Y)}{2}\right\}.
 \end{array}$$
By definition of $x_0$ we have
$\inf_{f\in\mathcal{F}}P\left\{\left| \sum_{i=1}^n Y_i(f) \right|
\leq \frac{x_0}{2} \right\} \geq 1 - \bar{p}_\tau /2$. Since
$P_Y\left\{ x(Y) \leq q_\tau\right\} = \bar{p}_\tau$, by
Bonferroni inequality we have that the left hand side is bounded
from below by $\bar{p}_\tau - \bar{p}_\tau/2 = \bar{p}_\tau/2$.
Therefore, over the set $\{Z : x(Z) \geq q_\tau, \left\|
\sum_{i=1}^n Z_i \right\|_{\mathcal{F}} > x_0 \vee x(Z)\}$ we have
$\frac{\bar{p}_\tau}{2} \leq P_Y\left\{ \left\| \sum_{i=1}^n (Y_i
- Z_i)\right\|_{\mathcal{F}} > \frac{x_0 \vee x(Z) \vee
x(Y)}{2}\right\}.$ Integrating over $Z$ we obtain {\small $$
\frac{\bar{p}_\tau}{2} P\left\{ x(Z) \geq q_\tau, \left\|
\sum_{i=1}^n Z_i \right\|_{\mathcal{F}} > x_0 \vee x(Z)\right\}
\leq P_ZP_Y\left\{ \left\| \sum_{i=1}^n (Y_i -
Z_i)\right\|_{\mathcal{F}} > \frac{x_0 \vee x(Z) \vee
x(Y)}{2}\right\}.$$} Let $\varepsilon_1,\ldots,\varepsilon_n$ be
an independent sequence of Rademacher random variables.   Given
$\varepsilon_1,\ldots,\varepsilon_n$, set $(\tilde Y_i = Y_i,
\tilde Z_i = Z_i)$ if $\varepsilon_i =1$ and $(\tilde Y_i = Z_i,
\tilde Z_i = Y_i)$ if $\varepsilon_i=-1$. That is, we create
vectors $\tilde Y$ and $\tilde Z$ by pairwise exchanging their
components; by construction, conditional on  each
$\varepsilon_1,\ldots,\varepsilon_n$, $(\tilde Y, \tilde Z)$ has
the same distribution as $(Y, Z)$.  Therefore, {\small $$
P_ZP_Y\left\{ \left\| \sum_{i=1}^n (Y_i -
Z_i)\right\|_{\mathcal{F}} > \frac{x_0 \vee x(Z) \vee
x(Y)}{2}\right\} = \Ep_\varepsilon P_ZP_Y\left\{ \left\|
\sum_{i=1}^n (\tilde Y_i - \tilde Z_i)\right\|_{\mathcal{F}} >
\frac{x_0 \vee x(\tilde Z) \vee x(\tilde Y)}{2}\right\}.$$} \!By
$x(\cdot)$ being $k$-sub-exchangeable, and since
$\varepsilon_i(Y_i-Z_i) = (\tilde Y_i - \tilde Z_i)$, we have that
{\small $$ \Ep_\varepsilon P_ZP_Y\left\{ \left\| \sum_{i=1}^n
(\tilde Y_i - \tilde Z_i)\right\|_{\mathcal{F}} > \frac{x_0 \vee
x(\tilde Z) \vee x(\tilde Y)}{2}\right\} \leq \Ep_\varepsilon
P_ZP_Y\left\{ \left\| \sum_{i=1}^n \varepsilon_i(Y_i -
Z_i)\right\|_{\mathcal{F}} > \frac{x_0 \vee x(Z) \vee
x(Y)}{2k}\right\}.$$} By the triangular inequality and removing
$x(Y)$ or $x(Z)$, the latter is bounded by
$$ P\left\{ \left\| \sum_{i=1}^n \varepsilon_i(Y_i - \mu_i)\right\|_{\mathcal{F}} > \frac{x_0  \vee x(Y)}{4k}\right\} + P\left\{ \left\| \sum_{i=1}^n \varepsilon_i(Z_i - \mu_i)\right\|_{\mathcal{F}} > \frac{x_0 \vee x(Z)}{4k} \right\}.$$
\end{proof}

\begin{proof}[Proof of Lemma \ref{expo3}]


Let $\mathbb{G}^o_n(f) =
n^{-1/2} \sum_{i=1}^n \{\varepsilon_i f(Z_i)\}$ be the symmetrized empirical process, where
$\varepsilon_1,\ldots,\varepsilon_n$ are i.i.d. Rademacher random
variables, i.e., $P(\varepsilon_i = 1) = P(\varepsilon_i = -1) =
1/2$, which are independent of $Z_1,\ldots,Z_n$.
By the Chebyshev's inequality and the assumption on
$e_n(\mathcal{F},\Pn)$, namely $ \sup_{f \in \mathcal{F}}\text{var}_{\mathbb{P}}f  \leq (\tau/2)(4kcK e_n(\mathcal{F},\Pn))^2$, we have for the constant $\tau$ fixed in
the statement of the lemma
$$ P(|\mathbb{G}_n(f)|> 4kcKe_n(\mathcal{F}, \mathbb{P}_n) \ ) \leq \tau/2.
$$
Therefore, by the symmetrization Lemma \ref{Lemma:SymTau}  we
obtain
$$
\mathbb{P}\left\{\sup_{f \in \mathcal{F}} |\mathbb{G}_n(f)|> 4k c
Ke_n(\mathcal{F}, \mathbb{P}_n)\right\} \leq \frac{4}{\tau}
\mathbb{P}\left\{\sup_{f \in \mathcal{F}} | \mathbb{G}_n^o(f)| > c
Ke_n(\mathcal{F}, \mathbb{P}_n)  \right\} + \tau.
$$ We then condition on the values of $Z_1,\ldots,Z_n$, denoting the conditional probability measure as $\mathbb{P}_{\varepsilon}$. Conditional on $Z_1,\ldots,Z_n$, by the Hoeffding inequality the symmetrized process $\mathbb{G}_n^o$ is sub-Gaussian for the $L_2(\mathbb{P}_n)$ norm, namely, for $g \in \mathcal{F}-\mathcal{F}$,
$ \mathbb{P}_{\varepsilon}\{\mathbb{G}_n^o(g) >x \} \leq 2 \exp( -
x^2/[2\|g\|_{\mathbb{P}_n,2}^2]).$ Hence by Lemma \ref{expo1} with
$D=1$, we can bound

$$
\mathbb{P}_{\varepsilon}\left\{\sup_{f \in \mathcal{F}}
|\mathbb{G}_n^o(f)| \geq cKe_n(\mathcal{F}, \mathbb{P}_n)\right\}
\leq   \left[\int_{0}^{\rho(\mathcal{F}, \mathbb{P}_n)/2}
\epsilon^{-1} n (\epsilon, \mathcal{F},P)^{-\{ K^{2} -1 \}}
d\epsilon  \right]\wedge 1.
$$
The result follows from taking the expectation over
$Z_1,\ldots,Z_n$.
\end{proof}

\begin{proof}[Proof of Lemma \ref{master lemma}]
Step 1. (Main Step)  In this step we prove the main result. First, we observe
that the bound $\epsilon \mapsto n(\epsilon, \mathcal{F}_m, \mathbb{P}_n)$
satisfies the monotonicity hypotheses of Lemma \ref{expo3}
uniformly in $m \leq N$.

Second, recall $e_n(\mathcal{F}_m, \mathbb{P}_n) := C
\sqrt{m\log(N\vee n\vee \bm \vee \omega )} \max \{ \sup_{f \in \mathcal{F}_m}
\|f\|_{\mathbb{P},2}, \  \sup_{f \in \mathcal{F}_m}
\|f\|_{\mathbb{P}_n,2}\}$ for $C = (1+\sqrt{2\upsilon})/4$. Note
that  $\sup_{f \in \mathcal{F}_m}\|f\|_{\mathbb{P}_n,2}$ is
$\sqrt{2}$-sub-exchangeable and $\rho(\mathcal{F}_m, \mathbb{P}_n)
:= \sup_{f \in \mathcal{F}_m} \| f\|_{\mathbb{P}_n,2}/
\|F_m\|_{\mathbb{P}_n, 2} \geq 1/\sqrt{n}$ by Step 2 below.
 Thus, uniformly  in $m \leq N$:
  \begin{eqnarray*}
& &  \|F_m\|_{\mathbb{P}_n,2} \int_{0}^{\rho(\mathcal{F}_m,\mathbb{P}_n)/4}  \sqrt{\log  n (\epsilon, \mathcal{F}_m,\mathbb{P}_n)} d \epsilon   \\
 & & \leq  \displaystyle  \|F_m\|_{\mathbb{P}_n,2} \int_{0}^{\rho(\mathcal{F}_m,\mathbb{P}_n)/4} \sqrt{m\log (N\vee n \vee \bm)  + \upsilon m \log(\omega / \epsilon)} d \epsilon  \\
& & \leq  \displaystyle (1/4)\sqrt{m\log (N\vee n \vee \bm)}\sup_{f \in \mathcal{F}_m} \| f\|_{\mathbb{P}_n,2}  +  \|F_m\|_{\mathbb{P}_n,2} \int_{0}^{\rho(\mathcal{F}_m,\mathbb{P}_n)/4}\sqrt{\upsilon m\log(\omega / \epsilon)}d\epsilon \\
& &  \leq  \displaystyle \sqrt{m\log (N\vee n \vee \bm \vee
\omega)}\sup_{f \in \mathcal{F}_m} \| f\|_{\mathbb{P}_n,2}\( 1 +
\sqrt{2\upsilon} \)/4 \leq  e_n(\mathcal{F}_m, \mathbb{P}_n),
\end{eqnarray*} which follows by $\int_0^{\rho}
\sqrt{\log(\omega/\epsilon)} d\epsilon \leq \left( \int_0^{\rho} 1
d\epsilon \right)^{1/2}\left(\int_0^{\rho} \log(\omega/\epsilon)
d\epsilon \right)^{1/2}
\leq \rho \sqrt{2\log (n \vee \omega)},$ for  $1/\sqrt{n} \leq
\rho \leq 1$.

Third, for any $K \geq  \sqrt{2/\delta} >1$ we have $(K^2 -1) \geq
1/\delta$, and let $\tau_m = \delta/(4m\log (N\vee n\vee \theta_0))$. Recall that
$4\sqrt{2}cC>4$ where $4<c<30$ is defined in Lemma \ref{expo1}.
Note that for any $m\leq N$ and $f \in \mathcal{F}_m$, we have by
Chebyshev's inequality and since $e_n(\mathcal{F}_m, \mathbb{P}_n) \geq C\sqrt{m\log(N\vee n\vee \theta_m\vee \omega)}\sup_{f\in \mathcal{F}_m}\|f\|_{\mathbb{P},2}$
$$
P(|\mathbb{G}_n(f)|> 4\sqrt{2}cKe_n(\mathcal{F}_m, \mathbb{P}_n) \ )  
\leq
\frac{\delta/2}{(4\sqrt{2}cC)^2m\log (N \vee n \vee \theta_0)} \leq \tau_m/2.
$$

By Lemma \ref{expo3} with our choice of $\tau_m$, $m\leq N$,
$\omega
>1$, $\upsilon > 1$,  and $\rho(\mathcal{F}_m,\Pn)\leq 1$,
\begin{eqnarray*}  &&\mathbb{P}\Big \{ \sup_{f \in \mathcal{F}_m} |
\mathbb{G}_n(f)|  >   4 \sqrt{2}c K e_n(\mathcal{F}_m,
\mathbb{P}_n), \exists   m \leq N
\Big \} \\
&&  \leq  \sum_{m=1}^N \mathbb{P}\left\{ \sup_{f \in \mathcal{F}_m} | \mathbb{G}_n(f)| >  4\sqrt{2} c K e_n(\mathcal{F}_m, \mathbb{P}_n)\right\} \\
 &&  \leq   \sum_{m=1}^N \left[\frac{4(N\vee n \vee \bm)^{-m/\delta}}{\tau_m}\int_0^{1/2} (\omega/\epsilon)^{(-\upsilon m/\delta) + 1} d \epsilon + \tau_m \right]\\
&&  \leq  4\sum_{m=1}^N \frac{(N\vee n \vee \bm)^{-m/\delta}}{\tau_m} \frac{1}{\upsilon m/\delta} + \sum_{m=1}^N \tau_m   \\
 &&  <   16\frac{(N\vee n\vee \theta_0)^{-1/\delta}}{1-(N\vee n \vee \theta_0)^{-1/\delta}}\log (N\vee n \vee \theta_0) +
\frac{\delta}{4}\frac{(1+\log N)}{\log (N\vee n \vee \theta_0)} \leq  \delta,
\end{eqnarray*}
\!\!where the last inequality follows by $N\vee n \vee \theta_0 \geq 3$ and $\delta \in (0,1/6)$.

\comment{If $ (n\vee \theta_0) \geq e $ we have $\frac{1 + \log n}{2\log( n\vee \theta_0)} = \frac{\log(en)}{\log( n^2\vee \theta_0^2 )} \leq 1$ if $n\vee \theta_0 \geq e$.

If $ (n\vee\theta_0) \geq 3^\delta$ we have the denominator of the first term to be at least $2/3$. Then we will bound $ (n\vee \theta_0)^{-1/\delta} \log(n\vee \theta_0) \leq \delta/48$. Taking logs, it suffices to have $(1/\delta) \log(n\vee \theta_0) \geq \log 48 + \log \log((n\vee \theta_0)/\delta)$ which holds if $(1/\delta)\log(n\vee\theta_0) \geq 5.6$. So it suffices to have $\delta \in (0,1/6)$ and $(n\vee \theta_0)\geq 3$.}

Step 2. (Auxiliary calculations.) To establish that $\sup_{f \in
\mathcal{F}_m}\|f\|_{\mathbb{P}_n,2}$ is
$\sqrt{2}$-sub-exchangeable, let $\tilde Z$ and $\tilde Y$ be created
by exchanging any components in $Z$ with corresponding components in $Y$. Then
\begin{eqnarray*}
&& \text{$\sqrt{2}(\sup_{f \in \mathcal{F}_m}\|f\|_{\mathbb{P}_n(\tilde
Z),2} \vee \sup_{f \in \mathcal{F}_m}\|f\|_{\mathbb{P}_n(\tilde
Y),2})$  $\geq$  $(\sup_{f \in
\mathcal{F}_m}\|f\|_{\mathbb{P}_n(\tilde Z),2}^2 + \sup_{f \in
\mathcal{F}_m} \|f\|_{\mathbb{P}_n(\tilde Y),2}^2)^{1/2}$ } \\
&&\text{$\geq$  $(\sup_{f \in \mathcal{F}_m} \En[ f(\tilde Z_i)^2 ] + \En[ f(\tilde Y_i)^2 ])^{1/2}$
  $=$  $(\sup_{f \in \mathcal{F}_m} \En[ f(Z_i)^2] + \En[ f(Y_i)^2 ])^{1/2}$}\\
&& \text{$\geq$
  $(\sup_{f \in \mathcal{F}_m}\|f\|_{\mathbb{P}_n(Z),2}^2$ $\vee$ $\sup_{f \in \mathcal{F}_m} \|f\|_{\mathbb{P}_n(Y),2}^2)^{1/2}$$=$  $\sup_{f \in \mathcal{F}_m}\|f\|_{\mathbb{P}_n(Z),2}$ $\vee$ $\sup_{f \in \mathcal{F}_m}\|f\|_{\mathbb{P}_n(Y),2}.$}
\end{eqnarray*}

Next we show that $\rho(\mathcal{F}_m, \mathbb{P}_n) := \sup_{f
\in \mathcal{F}_m} \| f\|_{\mathbb{P}_n,2}/ \|F_m\|_{\mathbb{P}_n,
2} \geq 1/\sqrt{n}$ for $m\leq N$. The latter bound follows from $\En\[
F_m^2 \] = \En[ \sup_{f \in \mathcal{F}_m} |f(Z_i)|^2 ] $ $\leq$ $
\sup_{i\leq n} \sup_{f \in \mathcal{F}_m} |f(Z_i)|^2, $ and from
$\sup_{f \in \mathcal{F}_m} \En[|f(Z_i)|^2]  \geq  \sup_{f \in
\mathcal{F}_m} \sup_{i\leq n} |f(Z_i)|^2/n$.
\end{proof}



\section{Supplementary Material}\label{Supplementary Material}

Recall the notation for the unit sphere $\mathbb{S}^{n-1}
= \{ \alpha \in \RR^n : \|\alpha\| = 1 \}$ and the $k$-sparse unit
sphere $\mathbb{S}_p^k = \{ \alpha \in \RR^p : \|\alpha\|=1,
\|\alpha\|_0 \leq k \}$. Define also the sparse sphere associated with a given vector
$\beta$ as
$ \mathbb{S}(\beta) = \{ \alpha \in \RR^{p} : \|\alpha\| \leq 1, \supp (\alpha) \subseteq \supp(\beta) \}$.

\subsection{Proofs for Examples of Simple Sufficient Conditions}

In this section we provide the proofs for Lemmas 1 and 2 which show that two designs of interest imply conditions D.1-5 under mild assumptions. We restate the designs for the reader's convenience.

{\sc Design 1: Location Model with Correlated Normal Design}. Let us consider estimating a standard  location model
$$ y = x'\beta^o + \varepsilon, $$
where $\varepsilon \sim N(0,\sigma^2)$, $\sigma>0$ is fixed, $x=( 1, z' )'$,
with $z \sim N(0, \Sigma),$ where $\Sigma$ has ones in the diagonal, minimum eigenvalue bounded away from zero by constant $\kappa^2$ and maximum eigenvalue bounded from above, uniformly in $n$.

\begin{lemma}
Under Design 1 with $\mathcal{U} = [\xi,1-\xi]$, $\xi > 0$, conditions D.1-D.5 are satisfied with
$$\bar f = 1/[\sqrt{2\pi}\sigma], \ \ \bar f' = \sqrt{e/[2\pi]}/\sigma^2, \ \ \underf  = 1/\sqrt{2\pi\xi}\sigma, $$
$$\|\beta(u)\|_0 \leq\|\beta^o\|_0 + 1, \ \ \gamma = 2p\exp(-n/24), \ \ L =\sigma /\xi$$ $$ \kappa_m \wedge \widetilde \kappa_{\widetilde m} \geq \kappa,  \ \ q\wedge \widetilde q_{\widetilde m} \geq (3/[32 \xi^{3/4}])\sqrt{\sqrt{2\pi}\sigma/e} .$$
\end{lemma}
\begin{proof}
This model implies a linear quantile model with coefficients
$\beta_1(u) = \beta^o_1 + \sigma\Phi^{-1}(u)$ and $\beta_{j}(u)
= \beta^o_{j}$ for $j=2,\ldots,p$. Let
$$ \ \bar f' = \sup_z
\phi'(z/\sigma)/\sigma^2,  \ \bar f = \sup_{z}
\phi(z/\sigma)/\sigma, \ \ \underf = \min_{u \in
\mathcal{U}} \phi( \Phi^{-1}(u))/\sigma, $$
so that D.1 holds with the constants $\bar f$ and $\bar f'$.
 D.2 holds, since $\|\beta(u)\|_0  \leq \|\beta^o\|_0+1$ and $u \mapsto \beta(u)$ is Lipschitz over $\mathcal{U}$
with the constant $L= \sup_{u \in \mathcal{U}} \sigma /\phi(\Phi^{-1}(u))$, which  trivially obeys $\log L \lesssim \log (n \vee p)$.
D.4 also holds, in particular by Chernoff's  tail bound
$$ P \left\{\max_{1\leq j\leq p}  |\hat\sigma_j-1| \leq 1/2\right\} \geq 1-\gamma = 1-2p\exp(-n/24),$$ where $1-\gamma$ approaches 1 if $n/\log p \to \infty$.   Furthermore, the smallest eigenvalue of the population design matrix $\Sigma$
is at least $(1-|\rho|)/(2+2|\rho|)$ and the maximum eigenvalue is
at most $(1+|\rho|)/(1-|\rho|)$. Thus, D.4 and D.5 hold with
$$
\kappa_m \wedge \widetilde \kappa_{\widetilde m} \geq \kappa, \ \ 
$$
for all $m, \widetilde m \geq 0$. If the covariates $x$ have a log-concave density, then  $$q
\geq 3 \underf^{3/2} / (8K_\ell \bar f') \text{ for a universal constant
$K_\ell$ }.$$ In the case of normal variables you can take $K_\ell = 4/\sqrt{2\pi}$. The bound follows from $ \Ep[|x_i'\delta|^3]
\leq K_\ell \Ep[|x_i'\delta|^2]^{3/2} $ holding for log-concave $x$
for some universal constant $K_{\ell}$ by Theorem 5.22 of
\cite{LovaszVempala2007}. The bound for $\widetilde q_{\widetilde m}$ is the same.
\end{proof}

{\sc Design 2: Location-scale shift with bounded regressors}. Let us consider estimating a standard  location model
$$ y = x'\beta^o + \varepsilon \cdot x'\eta, $$
where $\varepsilon \sim F$ independent of $x$, with probability density function $f$. We assume that the population design matrix $\Ep[xx']$ has ones in the diagonal and has eigenvalues uniformly bounded away from zero and from above,  $x_1 =1$, $\max_{1\leq j\leq p} |x_{j}| \leq K_B$. Moreover, the vector $\eta$ is such that $0 < \upsilon \leq  x'\eta \leq \Upsilon < \infty$ for all values of $x$.

\begin{lemma}
Under Design 2 with $\mathcal{U} = [\xi,1-\xi]$, $\xi > 0$, conditions D.1-D.5 are satisfied with $$\bar f \leq \max_{\varepsilon} f(\varepsilon)/\upsilon, \ \ \bar f' \leq \max_\varepsilon f'(\varepsilon) / \upsilon^2, \ \ \underf =  \min_{u \in \mathcal{U}} f(F^{-1}(u))/\Upsilon, \ \ $$
$$\|\beta(u)\|_0 \leq\|\beta^o\|_0 + \|\eta\|_0 + 1, \ \ \gamma = 2p\exp(-n/[8K_B^4]), $$ $$ \kappa_m \wedge \widetilde \kappa_{\widetilde m} \geq \kappa, \ \ L =  \|\eta\| \underf$$ $$ q \geq
 \frac{3}{8} \frac{\underf^{3/2}}{\bar{f'}}  \kappa /[10 K_B \sqrt{s}],  \widetilde q_{\widetilde m} \geq   \frac{3}{8} \frac{\underf^{3/2}}{\bar{f'}}  \kappa /[ K_B \sqrt{s+\widetilde m}],$$
\end{lemma}

\begin{proof}
This model implies a linear quantile model with coefficients
$\beta(u) = \beta^o_1 + F^{-1}(u)\eta$. We have
$$ \ \bar f' = \max_y f'(y) / \upsilon^2,  \ \bar f = \max_{y} f(y)/\upsilon, \ \ \underf \geq \min_{u \in \mathcal{U}} f(F^{-1}(u))/\Upsilon, $$
so that D.1 holds with the constants $\bar f$ and $\bar f'$.
 D.2 holds, since $\|\beta(u)\|_0 \leq \|\beta^o\|_0 + \|\eta\|_0 + 1$ and $u \mapsto \beta(u)$ is Lipschitz over $\mathcal{U}$
with the constant $L= \|\eta\|\max_{u \in \mathcal{U}} \Upsilon /f(F^{-1}(u))$ uniformly in $n$, which  obeys $\log L \lesssim \log (n \vee p)$.
Next recall that $x_{ij}^2\leq K_B^2$. Then, by Hoeffding inequality we have
$$ P( |\En[ x_{ij}^2 - 1 ] | \geq 1/2 ) \leq 2\exp( -n/[8K_B^4]).$$
Applying a union bound D.3 holds with $\gamma = 2p\exp( -n/[8K_B^4])$ which approaches 0 if $n/\log p \to \infty$.   Furthermore, the smallest eigenvalue of the population design matrix is bounded away from zero. Thus, D.4 and D.5 hold with $c_0 = 9$ (in fact with any $c_0>0$) and
$$
\kappa_m \wedge \widetilde \kappa_{\widetilde m} \geq \sqrt{{\rm min eig}(\Ep[xx'])}, 
$$
for all $m, \widetilde m \geq 0$.

Finally, the restricted nonlinear impact coefficient satisfies
$q \geq 3 \underf^{3/2} \kappa_0 / (8 \bar f' K_B (1+c_0)
\sqrt{s} )$. Indeed, the latter bound follows from $\Ep[|x_i'\delta|^3] \leq  \Ep[|x_i'\delta|^2] K_B \|\delta\|_1 \leq
\Ep[|x_i'\delta|^2] K_B (1+c_0)\sqrt{s}\|\delta_{T_u}\| \leq
\Ep[|x_i'\delta|^2]^{3/2} K_B (1+c_0)\sqrt{s} / \kappa_0 $ holding
since $\delta \in A_u$ so that $\|\delta\|_1 \leq
(1+c_0)\|\delta_{T_u}\|_1 \leq
\sqrt{s}(1+c_0)\|\delta_{T_u}\|$. Similarly, one can show $ \widetilde q_{\widetilde m} \geq (3/8) \uglyf^{3/2}
\widetilde \kappa_{\widetilde m}/ (\bar f' K_B \sqrt{\widetilde m + s} )$.
\end{proof}

\subsection{Lemma 9: Proof of the VC index bound}

\begin{lemma}\label{Lemma:VC-F}
Consider a fixed subset $T \subset \{1,2,\ldots,p\}$, $|T|=m$. The
class of functions
$$
\mathcal{F}_T = \left\{ \alpha'( \psi_i(\beta,u) -
\psi_i(\beta(u),u)) \  : u \in \mathcal{U}, \alpha \in
\mathbb{S}(\beta), \supp(\beta) \subseteq T  \right\}
$$
 has a VC index bounded by $cm$ for some universal constant $c$.
\end{lemma}
\begin{proof} Consider the classes of functions $\mathcal{W}:= \{ x'\alpha :
\supp(\alpha) \subseteq T\}$ and $\mathcal{V}:= \{ 1\{y \leq
x'\beta \} : \supp(\beta) \subseteq T\}$ (for convenience let $Z =
(y,x)$), $\mathcal{Z} := \{ 1\{ y \leq x'\beta(u)\} : u \in
\mathcal{U}\}$. Since $T$ is fixed and has cardinality $m$, the VC
indices of  $\mathcal{W}$ and $\mathcal{V}$ are bounded by $m+2$;
see, for example, van der Vaart and Wellner \cite{vdV-W} Lemma
2.6.15. On the other hand, since $u\mapsto x'\beta(u)$ is
monotone, the VC index of $\mathcal{Z}$ is 1. Next consider $ f
\in \mathcal{F}_T$ which can be written in the form $f(Z) := g(Z)
( 1\{h(Z)\leq 0\} - 1\{p(Z)\leq 0\} )$ where $g \in \mathcal{W}$,
$1\{h \leq 0 \} \in \mathcal{V}$, and $1\{ p \leq 0\} \in
\mathcal{Z}$.  The VC index of $\mathcal{F}_T$ is by definition
equal to the VC index of the class of sets $\{(Z,t): f(Z) \leq
t\}, f \in \mathcal{F}_T, t \in \mathbb{R}$. We have that
$$
\begin{array}{rcl}
 \{ (Z,t) : f(Z) \leq t \} & = & \{ (Z,t) : g(Z)( 1\{h(Z)\leq 0\} - 1\{ p(Z) \leq 0\}) \leq t \}\\
 & = & \{ (Z,t) : h(Z) > 0, p(Z) > 0 ,\ t \geq 0 \} \ \cup \\
 & \cup & \{ (Z,t) : h(Z) \leq 0, p(Z) \leq 0, \ t \geq 0 \} \ \cup \\
& \cup & \{ (Z,t) : h(Z) \leq 0, p(Z) > 0, g(Z) \leq t \} \ \cup \\
& \cup & \{ (Z,t) : h(Z) > 0, p(Z) \leq 0, g(Z) \geq t \}. \\
\end{array}
$$
Thus each set $\{ (Z,t) : f(Z) \leq t \}$ is created by taking finite unions, intersections, and complements of the basic sets $\{Z: h(Z) > 0\}$,
$\{Z: p(Z) \leq 0\}$, $\{t \geq 0\}$, $\{(Z,t): g(Z) \geq t\}$, and $\{(Z,t): g(Z) \leq t\}$. These basic sets form VC classes, each having VC index of order $m$.
Therefore, the VC index of a class of sets $\{(Z,t): f(Z) \leq t\}, f \in \mathcal{F}_T, t \in \mathbb{R}$ is of the same order as the sum of the VC indices of the set classes formed by the basic VC sets; see van der Vaart and Wellner \cite{vdV-W} Lemma 2.6.17. 
\end{proof}

\subsection{Gaussian Sparse Eigenvalue}

It will be convenient to recall the following result.

\begin{lemma}\label{Lemma:SparseEigenvalue}
Let $M$ be a semi-definite positive matrix and  $ \phi_M(k) = \sup
\{ \ \alpha'M\alpha \ : \alpha \in \mathbb{S}_p^k \ \}$. For any
integers $k$ and $\ell k$ with $\ell \geq 1$,  we have $
\phi_M(\ell k) \leq \lceil \ell \rceil \phi_M( k).$\end{lemma}

The following lemmas characterize the behavior of the maximal sparse eigenvalue for the case of correlated Gaussian regressors. We start by establishing an upper bound on $\phi_{\En[x_ix_i']}(k)$ that holds with high probability.

\begin{lemma}\label{Lemma:SparseEigvenNormal} Consider $x_i = \Sigma^{1/2} z_i$, where $z_i \sim N(0,I_p)$, $p \geq n$, and $\sup_{\alpha \in \mathbb{S}_p^k} \alpha'\Sigma\alpha \leq \sigma^2(k)$. Let
$\phi_{\En[x_ix_i']}(k)$ be the maximal $k$-sparse eigenvalue of $\En \[x_i x_i'\]$, for $k \leq n$. Then  with probability converging to one, uniformly in $k \leq n$,
$$\sqrt{\phi_{\En[x_ix_i']}(k)} \lesssim \sigma( k) \(1+ \sqrt{k/n}\sqrt{\log p}\).$$
\end{lemma}

\begin{proof} By Lemma \ref{Lemma:SparseEigenvalue} it suffices to establish the result for $k \leq n/2$.   Let $Z$ be the $n\times p$ matrix collecting vectors $z_i'$, $i=1,\ldots,n$ as rows. \ Consider the  Gaussian process   $ \mathcal{G}_k: (\alpha, \tilde \alpha) \mapsto \tilde \alpha' Z\alpha/\sqrt{n}$, where $(\alpha, \tilde \alpha) \in \mathbb{S}^k_p\times \mathbb{S}^{n-1}$. Note that \begin{equation*}
 \|\mathcal{G}_k\|  = \sup_{(\alpha, \tilde \alpha) \in \mathbb{S}^k_p\times \mathbb{S}^{n-1}}  |\tilde \alpha' Z\alpha/\sqrt{n}| = \sup_{\alpha \in \mathbb{S}^k_p} \sqrt{ \alpha'\En [z_i z_i']\alpha} =\sqrt{\phi_{\En[z_iz_i']}(k)}.
 \end{equation*}
Using Borell's concentration inequality for the Gaussian process  (see van der Vaart and Wellner \cite{vdV-W} Lemma A.2.1) we have that $
P(\|\mathcal{G}_k\| - \text{median} \|\mathcal{G}_k\|>r ) \leq e^{ - n r^2/2}.$
Also, by classical results on the behavior of the maximal eigenvalues of the Gaussian covariance matrices (see German \cite{German1980}),  as $n \to \infty$, for any $k/n \to \gamma \in [0,1]$, we have that
 $
\lim_{k, n}  (\text{median} \|\mathcal{G}_k\|  - 1 - \sqrt{k/n}) = 0.$
Since $k/n$ lies within $[0,1]$, any sequence $k_n/n$ has convergent subsequence with limit in $[0,1]$. Therefore, we can conclude that, as $n \to \infty$,
 $
\limsup_{k_n, n} (\text{median} \|\mathcal{G}_{k_n}\|  - 1 - \sqrt{k_n/n}) \leq 0.$
This further implies
$
\limsup_n \sup_{k \leq n } (\text{median} \|\mathcal{G}_{k}\|  - 1 - \sqrt{k/n}) \leq 0.$
Therefore, for any $r_0>0$ there exists $n_0$ large enough such that for all $n \geq n_0$ and all $k \leq n$, $
P(\|\mathcal{G}_k\| > 1 + \sqrt{k/n} + r + r_0 ) \leq e^{ - n r^2/2}.$
There are at most $\binom{p}{k}$ subvectors of $z_i$ containing $k$ elements, so we conclude that for $n \geq n_0$,
$$
P ( \sup_{\alpha \in \mathbb{S}_p^k} \sqrt{\alpha' \En [z_i z_i']\alpha } > 1 + \sqrt{k/n} + r_k + r_0  ) \leq  \binom{p}{k} e^{ - n r_k^2/2}.
$$
Summing over $k \leq n$ we obtain
$$
\sum_{k=1}^n P ( \sup_{\alpha \in \mathbb{S}_p^k} \sqrt{\alpha' \En [z_i z_i']\alpha } > 1 + \sqrt{k/n} + r_k + r_0  ) \leq \sum_{k=1}^n \binom{p}{k}  e^{ - n r_k^2/2}.
$$
Setting $r_k = \sqrt{c k/n \log p}$ for $c>1$ and using that  $\binom{p}{k} \leq p^k$, we bound the right side by $\sum_{k=1}^n e^{ (1-c) k \log p} \to 0 $ as $n \to \infty$. We conclude that with probability converging to one, uniformly for all $k$:
$
 \sup_{\alpha \in \mathbb{S}_p^k} \sqrt{\alpha' \En [z_i z_i']\alpha } \lesssim 1 + \sqrt{k/n} \sqrt{\log p}.$
Furthermore, since  $\sup_{\alpha \in \mathbb{S}_p^k} \alpha'\Sigma\alpha \leq \sigma^2(k)$, we conclude that with probability converging to one, uniformly for all $k$:
$$
 \sup_{\alpha \in \mathbb{S}_p^k} \sqrt{\alpha' \En [x_i x_i']\alpha }\lesssim \sigma(k) ( 1 + \sqrt{k/n} \sqrt{\log p}).
$$
\end{proof}

Next, relying on Sudakov's minoration, we show a lower bound on the expectation of the maximum $k$-sparse eigenvalue.  We do not use the lower bound in the  analysis, but the result shows that the upper bound is sharp in terms of the rate dependence on $k, p$, and $n$.

\begin{lemma}\label{Lemma:Sudakov}  Consider $x_i = \Sigma^{1/2} z_i$, where $z_i \sim N(0,I_p)$, and $\inf_{\alpha \in \mathbb{S}_p^k} \alpha'\Sigma\alpha \geq \underline{\sigma}^2(k)$.  Let
$\phi_{\En[x_ix_i']}(k)$ be the maximal $k$-sparse eigenvalue of $\En\[ x_i x_i'\]$, for $k \leq n < p$. Then for any even $k$ we have that: $$ (1) \ \ \Ep \[\sqrt{\phi_{\En[x_ix_i']}(k)}\] \geq \frac{\underline{\sigma}(2k)}{3\sqrt{n}} \sqrt{{(k/2)} \log(p-k)} \text { and }$$
 $$(2) \ \ \sqrt{\phi_{\En[x_ix_i']}(k)} \gtrsim_P \frac{\underline{\sigma}(2k)}{3\sqrt{n}} \sqrt{{(k/2)} \log(p-k)}. $$

\end{lemma}
\begin{proof}  Let $X$ be the $n\times p$ matrix collecting vectors $x_i'$, $i=1,\ldots,n$ as rows. \ Consider the  Gaussian process   $ (\alpha, \tilde \alpha) \mapsto \tilde \alpha' X\alpha/\sqrt{n}$, where $(\alpha, \tilde \alpha) \in \mathbb{S}^k_p\times \mathbb{S}^{n-1}$. Note that $\sqrt{\phi_{\En[x_ix_i']}(k)}$ is the supremum of this Gaussian process \begin{equation}
\label{suprmum representation} \sup_{(\alpha, \tilde \alpha) \in \mathbb{S}^k_p\times \mathbb{S}^{n-1}}  |\tilde \alpha' X\alpha/\sqrt{n}| = \sup_{\alpha \in \mathbb{S}^k_p} \sqrt{ \alpha'\En [x_i x_i']\alpha} =\sqrt{\phi_{\En[x_ix_i']}(k)}.
 \end{equation}

 Hence we proceed in three steps: In Step 1, we consider the uncorrelated case and prove the lower bound (1) on the expectation of the supremum using Sudakov's minoration, using a lower bound on a relevant packing number. In Step 2, we derive the lower bound on the packing number.  In Step 3, we generalize Step 1 to the correlated case.  In Step 4, we prove the lower bound (2) on the supremum itself using Borell's concentration inequality.

Step 1.  In this step we consider the case where $\Sigma = I$ and show the result using Sudakov's minoration.
By fixing $\tilde \alpha = (1,\ldots,1)'/ \sqrt{n} \in \mathbb{S}^{n-1}$, we have $\sqrt{\phi_{\En[x_ix_i']}(k)} \geq \sup_{\alpha \in \mathbb{S}^k_p} \En[ x_i'\alpha] =   \sup_{\alpha \in \mathbb{S}^k_p} Z_{\alpha}$, where $\alpha \mapsto Z_{\alpha}: = \En[ x_i'\alpha]$ is a Gaussian process on $\mathbb{S}^k_p$.  We will bound $\displaystyle E[\sup_{\alpha \in \mathbb{S}^k_p} Z_{\alpha}]$ from below using Sudakov's minoration.

We consider the standard deviation metric on $\mathbb{S}^k_p$ induced  by $Z$: for any $t,s \in \mathbb{S}^k_p$,
$$ d(s,t) = \sqrt{\sigma^2( Z_t - Z_s )} =  \sqrt{\Ep\[ (Z_t - Z_s )^2 \]}= \sqrt{\Ep[ \En \[(x_i'(t - s))^2] \]}  =  \|t-s\|/\sqrt{n}.$$
Consider the packing number $D(\epsilon, \mathbb{S}^k_p, d)$, the largest number of disjoint closed balls of radius $\epsilon$ with respect to the metric $d$ that can be packed into $\mathbb{S}^k_p$, see \cite{Dudley2000}. We will bound the packing number from below for $\epsilon = \frac{1}{\sqrt{n}}$.  In order to do this  we restrict  attention to the collection $\mathcal{T}$ of elements $t =(t_1,\ldots,t_p) \in \mathbb{S}^k_p$ such that $t_i = 1/\sqrt{k}$ for exactly $k$ components and $t_i=0$ in the remaining $p-k$ components. There are $|T| = \binom{p}{k}$ of such elements. Consider any $s,t \in \mathcal{T}$ such that the support of $s$ agrees with the support of $t$ in at most ${k/2}$ elements. In this case
 \begin{equation}\label{EQ: diff} \|s - t\|^2 = \sum_{j=1}^p |t_j - s_j|^2 \geq \sum_{j \in \supp(t) \atop {\setminus \supp(s)}} \frac{1}{k} + \sum_{j \in \supp(s) \atop{\setminus \supp(t)}} \frac{1}{k} \geq 2 \frac{k}{2}\frac{1}{k} = 1.
 \end{equation}
Let $\mathcal{P}$ be the set  of the maximal cardinality, consisting of elements in $\mathcal{T}$ such that $|\supp(t)\setminus \supp(s)| \geq {k/2}$ for every $s,t \in \mathcal{P}$. By the inequality (\ref{EQ: diff}) we have that $ D(1/\sqrt{n}, \mathbb{S}^k_p, d) \geq | \mathcal{P}|.$  Furthermore, by Step 2 given below we have that $|\mathcal{P}|  \geq (p-k)^{{k/2}}$.

 Using Sudakov's minoration (\cite{Fernique1997}, Theorem 4.1.4), we conclude that
$$ \Ep\Big [ \sup_{t \in \mathbb{S}^k_p} Z_t \Big ] \geq  \sup_{\epsilon>0} \frac{\epsilon}{3} \sqrt{\log D(\epsilon, \mathbb{S}^k_p, d)} \geq  \sqrt{ \log D(1/\sqrt{n}, \mathbb{S}^k_p, d)} \geq \frac{1}{3} \sqrt{ k \log(p-k)/(2n)},  $$
proving the claim of the lemma for the case $\Sigma = I$.

Step 2. In this step we show that $|\mathcal{P}|  \geq (p-k)^{{k/2}}$.

It is convenient to identify every element $t \in \mathcal{T}$ with the set $\supp(t)$, where  $ \supp(t) = ( j \in (1,\ldots,p) : t_j = 1/\sqrt{k} )$, which has cardinality  $k$. For any $t\in \mathcal{T}$ let $\mathcal{N}(t) = ( s \in \mathcal{T} \ : \ |\supp(t) \setminus \supp(s)| \leq {k/2})$.   By construction  we have that $\max_{t \in \mathcal{T}}|\mathcal{N}(t)||\mathcal{P}| \geq  |\mathcal{T}|$.  Since as shown below $\max_{t \in \mathcal{T}}|\mathcal{N}(t)| \leq K := \binom{ k}{{k/2} } \binom{p-{k/2}}{{k/2}}$ for every $t$,  we conclude that $|\mathcal{P}| \geq  |\mathcal{T}|/K = \binom{p}{k}/K  \geq (p-k)^{{k/2}}$.

It remains only to show that $|\mathcal{N}(t)| \leq \binom{ k}{{k/2} } \binom{p-{k/2}}{{k/2}}$. Consider an arbitrary $t \in \mathcal{T}$. Fix any  ${k/2}$ components of $\supp(t)$, and generate elements $s \in \mathcal{N}(t)$ by switching any of the remaining ${k/2}$ components in $\supp(t)$ to any of the possible $p-{k/2}$ values. This gives us at most $\binom{p-{k/2}}{{k/2} }$ such elements $s \in \mathcal{N}(t)$. Next let us repeat this procedure for all other combinations of initial ${k/2}$ components of $\supp(t)$, where the number of such combinations is bounded by $\binom{ k}{{k/2} }$.  In this way we generate every element $s \in \mathcal{N}(t)$.   From the construction we conclude that $|\mathcal{N}(t)| \leq \binom{ k}{{k/2} } \binom{p-{k/2}}{{k/2}}$.

\text{Step 3}. The  case where $\Sigma \neq I$ follows similarly noting that the new metric, $ d(s,t) = \sqrt{\sigma^2( Z_t - Z_s )} =  \sqrt{\Ep\[ (Z_t - Z_s )^2 \]}$, satisfies $$ d(s,t) \geq \underline{\sigma}(2k) \|s-t\|/\sqrt{n} \ \mbox{ since} \ \  \|s-t\|_0 \leq 2k.$$

Step 4. Using  Borell's concentration inequality  (see van der Vaart and Wellner \cite{vdV-W} Lemma A.2.1)  for  the supremum of the Gaussian process defined in (\ref{suprmum representation}), we have
$
P(|\sqrt{\phi_{\En[x_ix_i']}(k)} - E[\sqrt{\phi_{\En[x_ix_i']}(k)}]| >r ) \leq 2 e^{ - n r^2/2},
$
which  proves the second claim of the lemma. \end{proof}

Next we combine the previous lemmas to control the empirical sparse eigenvalues of the following example.

\begin{example}[Correlated Normal Design]\label{Ex:Id}
Let us consider estimating the median ($u=1/2$) of the following regression model
$$ y = x'\beta_0 + \varepsilon, $$
where the covariate $x_1 =1$ is the intercept and the covariates $x_{-1} \sim N(0,\Sigma)$, where
$\Sigma_{ij} = \rho^{|i-j|}$ and $ -1< \rho < 1$ is fixed.
\end{example}

\begin{lemma}\label{Lemma:PhiNormal}
For $k \leq n$, under the design of Example 1 with $p \geq 2n$, we have
$$ \phi_{\En[x_ix_i']}(k) \lesssim_P \frac{1+|\rho|}{1-|\rho|} \(1 + \sqrt{\frac{k\log p}{n}}\) \ \ \mbox{and} \ \ \phi_{\En[x_ix_i']}(k) \gtrsim_P \frac{1-|\rho|}{1+|\rho|} \(1 + \sqrt{\frac{k\log p}{n}}\).$$
\end{lemma}
\begin{proof}
Consider first the design in Example 1 with $\rho = 0$.    Let $x_{i,-1}$ denote the ith observation without the first component. Write $$\En\[x_ix_i'\] = \[\begin{matrix} 1 & \En\[ x_{i,-1}'\] \\ \En\[ x_{i,-1}\] & { 0} \end{matrix} \]+ \En\[\begin{matrix} 0 & {0} \\ { 0} & \En\[ x_{i,-1}x_{i,-1}'\] \end{matrix}\] = M + N.$$

We first  bound $\phi_N(k)$. Letting $N_{-1,-1}= \En\[x_{i,-1}x_{i,-1}'\]$ we have $\phi_N(k) = \phi_{N_{-1,-1}}(k)$. Lemma \ref{Lemma:SparseEigvenNormal} implies that $ \phi_N(k) \lesssim_P 1+ \sqrt{k/n}\sqrt{\log p}$. Lemma \ref{Lemma:Sudakov} bounds $\phi_N(k)$ from below because $ \phi_{N_{-1,-1}}(k) \gtrsim_P \sqrt{{(k/2n)} \log(p-k)}$.

We then bound  $\phi_M(k)$. Since $M_{11} = 1$, we have $\phi_M(1) \geq 1$. To produce an upper bound let $w = (a,b')'$ achieve $\phi_M(k)$ where $a \in \RR$, $b \in \RR^{p-1}$. By definition we have $\|w\| =1$, $\|w\|_0 \leq k$. Note that $|a|\leq 1$, $\|b\| = \sqrt{1-|a|^2}\leq 1$, $\|b\|_1 \leq \sqrt{k}\|b\|$. Therefore
$$\begin{array}{rcl}
\phi_M(k) = w'Mw & = & a^2 + 2ab'\En\[ x_{i,-1}\] \leq  1 + 2b'\En\[ x_{i,-1}\]  \\
 &\leq& 1 + 2 \|b\|_1 \|\En\[ x_{i,-1}\]\|_\infty \leq 1 + 2\sqrt{k}\|b\|\|\En\[ x_{i,-1}\]\|_\infty.\end{array}$$
Next we bound $\|\En\[ x_{i,-1}\]\|_\infty = \max_{j=2,\ldots,p}| \En\[ x_{ij}\]|$. Since $ \En\[ x_{ij}\] \sim N(0,1/n)$ for $j=2,\ldots,p$, we have $\|\En\[ x_{i,-1}\]\|_\infty \lesssim_P \sqrt{(1/n)\log p}$. Therefore we have $ \phi_M(k) \lesssim_P 1 + 2\sqrt{k/n}\sqrt{\log p}$.

Finally, we bound  $\phi_{\En[x_ix_i']}$. Note that $\phi_{\En[x_ix_i']}(k) = \sup_{\alpha \in \mathbb{S}_p^k}\alpha'(M+N)\alpha = \sup_{\alpha \in \mathbb{S}_p^k} \alpha'M\alpha + \alpha'N\alpha \leq \phi_M(k) + \phi_N(k)$.
On the other hand, $\phi_{\En[x_ix_i']}(k) \geq 1 \vee \phi_{N_{-1,-1}}(k)$ since the covariates contain an intercept. The result follows by using the bounds derived above.

The proof for an arbitrary $\rho$ in Example 1 is similar. Since $-1<\rho<1$ is fixed, the bounds on the eigenvalues of the population design matrix $\Sigma$ are given by $\sigma^2(k)=\sup_{\alpha \in \mathbb{S}_p^k} \alpha'\Sigma\alpha \leq (1+|\rho|)/(1-|\rho|)$ and $\underline{\sigma}^2(k)=\inf_{\alpha \in \mathbb{S}_p^k} \alpha'\Sigma\alpha \geq \frac{1}{2}(1-|\rho|)/(1+|\rho|)$. So we can apply Lemmas \ref{Lemma:SparseEigvenNormal} and \ref{Lemma:Sudakov}. To bound $\phi_M(k)$ we use comparison theorems, e.g. Corollary 3.12 of \cite{LedouxTalagrandBook}, which allows for the same bound as for the uncorrelated design to hold.
\end{proof}

\section*{Acknowledgements}
We would like to thank Arun Chandrasekhar, Denis Chetverikov, Moshe
Cohen, Brigham Fradsen, Joonhwan Lee, Ye Luo, and Pierre-Andre Maugis
for thorough proof-reading of several versions of this paper and their
detailed comments that helped us considerably improve the paper.
We also would like to thank  Don Andrews, Alexandre Tsybakov, the editor
Susan Murphy, the Associate Editor, and three anonymous referees
for their comments that also helped us considerably improve the paper.
We would also like to thank the participants of seminars in Brown
University, CEMMAP Quantile Regression conference at UCL, Columbia
University, Cowles Foundation Lecture at the Econometric Society Summer
Meeting, Harvard-MIT, Latin American Meeting 2008 of the Econometric
Society, Winter 2007 North American Meeting of the Econometric Society,
London Business School, PUC-Rio, the Stats in the Chateau,  the Triangle
Econometrics Conference, and University College London.
%


\footnotesize \linespread{1}

\end{document}